\documentclass[a4paper,reqno, 10pt]{amsart}
\usepackage{amsmath,amssymb,amsfonts,amsthm, mathrsfs}
\usepackage{lmodern}
\usepackage{makecell}
\usepackage{diagbox}
\usepackage{multirow}
\usepackage{booktabs}
\usepackage{verbatim,wasysym,cite}
\newcommand{\xp}{x^{\perp}}
\newcommand{\scaa}{L_{t,x}^\frac{5\alpha}{2}}
\newcommand{\isca}{L_{t}^\frac{5\alpha}{2}L_x^{\frac{30\alpha}{15\alpha-8}}}

\usepackage{microtype}
\usepackage{color,enumitem,graphicx}
\usepackage[colorlinks=true,urlcolor=blue, citecolor=red,linkcolor=blue,
linktocpage,pdfpagelabels, bookmarksnumbered,bookmarksopen]{hyperref}
\usepackage[english]{babel}
\usepackage[symbol]{footmisc}
\renewcommand{\epsilon}{{\varepsilon}}
\numberwithin{equation}{section}
\newtheorem{theorem}{Theorem}[section]
\newtheorem{Conjection}{Conjecture}[section]
\newtheorem{lemma}[theorem]{Lemma}
\newtheorem{remark}[theorem]{Remark}
\newtheorem{definition}[theorem]{Definition}
\newtheorem{proposition}[theorem]{Proposition}
\newtheorem{corollary}[theorem]{Corollary}

\newcommand{\R}{\mathbb R}
\newcommand{\N}{\mathbb N}

\newcommand{\norm}{\big\|}
\newcommand{\pn}{\phi_n}
\newcommand{\cn}{\chi_n}
\newcommand{\lamn}{\lambda_n}
\newcommand{\psie}{\psi_{\varepsilon}}
\newcommand{\Hsc}{\dot{H}^{s_c}}

\newcommand{\vn}{\tilde{v}_n}
\newcommand{\DeltaO}{\Delta_{\Omega}}
\newcommand{\DeltaOn}{\Delta_{\Omega_n}}
\newcommand{\RRT}{\R\times\R^3}
\newcommand{\RO}{\R\times\Omega}

\newcommand{\On}{\Omega_n}
\def\({\left(}
\def\){\right)}
\def\<{\left\langle}
\def\>{\right\rangle}

\DeclareMathOperator{\supp}{supp}

\newcommand{\qtq}[1]{\quad\text{#1}\quad}

\begin{document}

	\title[3d NLS  outside  a convex obstacle]
	{Scattering  theory for the defocusing 3d NLS in the exterior of a strictly convex obstacle }
	
	\author[X. Liu]{Xuan Liu}
	\address{School of Mathematics, Hangzhou Normal University, \ Hangzhou ,\ 311121, \ China} 
	\email{liuxuan95@hznu.edu.cn}

	\author{Yilin Song}
	\address{Yilin Song
		\newline \indent The Graduate School of China Academy of Engineering Physics,
		Beijing 100088,\ P. R. China}
	\email{songyilin21@gscaep.ac.cn}
	
	\author{Jiqiang Zheng}
	\address{Jiqiang Zheng
		\newline \indent Institute of Applied Physics and Computational Mathematics,
		Beijing, 100088, China.
		\newline\indent
		National Key Laboratory of Computational Physics, Beijing 100088, China}
	\email{zheng\_jiqiang@iapcm.ac.cn, zhengjiqiang@gmail.com}

	\begin{abstract}
		In this paper, we investigate the global well-posedness and scattering theory for the defocusing nonlinear Schr\"odinger equation $iu_t + \Delta_\Omega u = |u|^\alpha u$ in the exterior domain $\Omega$ of a smooth, compact and strictly convex obstacle in $\mathbb{R}^3$.
		It is conjectured that in Euclidean space, if the solution has a prior bound in the critical Sobolev space, that is, $u \in L_t^\infty(I; \dot{H}_x^{s_c}(\mathbb{R}^3))$ with $s_c := \frac{3}{2} - \frac{2}{\alpha} \in (0, \frac{3}{2})$, then $u$ is global and scatters.  In this paper, assuming that this conjecture holds, we prove that if $u$ is a solution to the  nonlinear Schr\"odinger equation in exterior domain  $\Omega$ with Dirichlet boundary condition and satisfies $u \in L_t^\infty(I; \dot{H}^{s_c}_D(\Omega))$ with $s_c \in \left[\frac{1}{2}, \frac{3}{2}\right)$, then $u$ is global and scatters.

		The proof of the main results relies on the concentration-compactness/rigidity argument of Kenig and Merle [Invent. Math. {\bf 166} (2006)].  The main difficulty is to construct minimal counterexamples when the  scaling and translation invariance breakdown on $\Omega$. 
		To achieve this, two key ingredients are required. 
		First, we adopt the approach of Killip, Visan, and Zhang [Amer. J. Math. {\bf 138} (2016)] to derive the linear profile decomposition for the linear propagator $e^{it\Delta_\Omega}$ in $\dot{H}^{s_c}(\Omega)$. 
		The second ingredient is the embedding of the nonlinear profiles. More precisely, we need to demonstrate that nonlinear solutions in the limiting geometries, which exhibit global spacetime bounds, can be embedded back into $\Omega$.
		Finally, to rule out the minimal counterexamples, we will establish long-time Strichartz estimates for the exterior domain NLS, along with spatially localized and frequency-localized Morawetz estimates.
		
		\vspace{0.3cm}
		
	\end{abstract}

	\maketitle
	\tableofcontents
	\medskip
	
	\section{Introduction}
	We study the defocusing nonlinear Schr\"odinger equation in the exterior domain $\Omega$  of a smooth compact, strictly convex obstacle in $\mathbb{R}^3$ with Dirichlet boundary condition: 
	\begin{equation}
		\begin{cases}
			iu_t+\Delta_\Omega u=|u|^{\alpha }u,\\
			u(0,x)=u_0(x),\\
			u(t,x)|_{x\in \partial \Omega}=0,
		\end{cases}\label{NLS}
	\end{equation}
	where  $u$ is a complex-valued function defined in  $\mathbb{R} \times \Omega$ and $-\Delta_{\Omega}$ denotes the Dirichlet Laplacian on $\Omega$. The Dirichlet-Laplacian is the unique self-adjoint operator on $L^2(\Omega)$ corresponding to the following quadratic form
	\[
	Q : H_0^1(\Omega) \to [0,\infty) \quad \text{with} \quad Q(f) := \int_{\Omega} \overline{\nabla f(x)} \cdot \nabla f(x) \, dx.
	\]
	We take initial data  $u_0\in \dot H^{s}_D(\Omega)$, where for  $s\ge0$,   the homogeneous Sobolev space is defined by the functional calculus as  the completion of $C_c^{\infty}(\Omega)$ with respect to the norm
	\[
	\|f\|_{\dot{H}^{s}_D(\Omega)} := \|(-\Delta_\Omega)^{s/2} f \|_{L^2(\Omega)}.
	\]
	
	It is easy to find that the  solution  $u$ to equation (\ref{NLS}) with sufficient smooth conditions possesses the mass and energy conservation laws: 
	\[
	M_{\Omega}[u(t)] := \int_{\Omega} |u(t,x)|^2 dx = M_\Omega[u_0],
	\]
	\[
	E_{\Omega}[u(t)] := \frac{1}{2} \int_{\Omega} |\nabla u(t,x)|^2 dx + \frac{1}{\alpha +2} \int_{\Omega} |u(t,x)|^{\alpha +2} dx = E_\Omega[u_0].
	\]
	When posed on the whole   Euclidean space $\mathbb{R}^3$, the Cauchy problem \eqref{NLS} is scale-invariant. More precisely, the scaling transformation
	\[
	u(t,x) \longmapsto \lambda^{\frac{2}{\alpha }} u(\lambda x, \lambda^2 t) \quad \text{for} \quad \lambda > 0,
	\]
	leaves the class of solutions to NLS$_{\mathbb{R} ^3}$ invariant. 
	This transformation also identifies the critical space  $\dot H^{s_c}_x$,  
	where the critical regularity  $s_c$ is given by  $s_c:=\frac{3}{2}-\frac{2}{\alpha }$. 
	We call  \eqref{NLS} mass-critical if  $s_c=0$, energy-critical if   $s_c=1$, inter-critical if   $0<s_c<1$ and energy-supercritical if  $s_c>1$ respectively. 
	Although the obstacle in the domain alters certain aspects of the equation, it does not affect the problem's inherent dimensionality. Therefore,  (\ref{NLS}) maintains the same criticality and is classified as   $\dot H^{s_c}_D(\Omega)$ critical.  
	
	Throughout this paper, we restrict ourselves to the following notion of solution.
	\begin{definition}[Solution]\label{Defsolution}
		A function $ u : I \times \Omega \to \mathbb{C} $ on a non-empty interval $ I \ni 0 $ is called a \emph{solution} to (\ref{NLS}) if it satisfies $u \in C_t \dot{H}^{s_c}_D(K \times \Omega) \cap L^{\frac{5\alpha }{2}}_{t,x}(K \times \Omega)$ for every compact subset $K \subset I$  and obeys the Duhamel formula
		\[
		u(t) = e^{it \Delta_\Omega} u_0 - i \int_0^t e^{i(t-s) \Delta_\Omega} (|u|^\alpha  u)(s) \, ds
		\]
		for each $ t \in I $.  We refer to the interval  $I$ as the lifespan of  $u$. We say that $ u $ is a maximal-lifespan solution if the solution cannot be extended to any strictly larger interval. We say that  $u$ is a global solution if  $I=\mathbb{R} $.  
	\end{definition}
	The assumption that the solution lies in the space  $L_{t,x}^{\frac{5\alpha }{2}}(I\times \Omega)$ locally in time is natural 	since by the Strichartz estimate  (see Proposition \ref{PStrichartz} below), the linear flow always lies in this space. Also, if a solution  $u$ to (\ref{NLS}) is global, with 
	$  \|u\|_{L_{t,x}^{\frac{5\alpha }{2}}(I\times \Omega)} < \infty $, then it \emph{scatters}; that is, there exist unique $ u_\pm \in \dot{H}^{s_c}_D(\Omega) $ such that
	\[
	\lim_{t \to \pm \infty} \left\| u(t) - e^{it \Delta_\Omega} u_\pm \right\|_{\dot{H}^{s_c}_D(\Omega)} = 0.
	\]
	
	The study of NLS in exterior domains was initiated in \cite{BurqGerardTzvetkov2004}. The authors proved a local existence result for the 3d sub-cubic (i.e., $\alpha  < 3$) NLS$_{\Omega}$ equation, assuming that the obstacle is non-trapping. Subsequently,  Anton  \cite{Anton2008} extended these result to  the cubic nonlinearity, while   Planchon-Vega \cite{PlanchonVega2009} extended it to the energy-subcritical NLS$_{\Omega}$ equation in dimension $d=3$. 
	Later, Planchon and Ivanovici \cite{IvanoviciPlanchon2010} established the small data scattering theory for the energy-critical NLS$_\Omega$ equation in dimension $d = 3$.  For 
	NLS outside a smooth, compact, strictly convex obstacle $\Omega$ in  $\mathbb{R} ^3$,   Killip-Visan-Zhang \cite{KillipVisanZhang2016a} proved that for arbitrarily large initial data, the corresponding solutions to the defocusing energy-critical equation scatter in the energy space.
	For related results in the focusing case, see e.g.  \cite{DuyckaertsLandoulsiRoudenko2022JFA, KillipVisanZhang2016c, KYang, XuZhaoZheng}.

	In this paper, we investigate the $\dot H^{s_c}_D(\Omega)$ critical global well-posedness and scattering theory for the defocusing NLS  (\ref{NLS}) in the exterior domain $\Omega$ of a smooth, compact and strictly convex obstacle in $\mathbb{R}^3$.   
	To put the problem in context, let us first recall some earlier results for		the equivalent problem posed in the whole Euclidean space  $\mathbb{R}^d$.
	The study of global well-posedness and scattering theory for nonlinear Schr\"odinger equations 
	\begin{equation}
		iu_t + \Delta u = \pm |u|^{\alpha }u,\qquad (t,x) \in \mathbb{R} \times \mathbb{R}^d  \label{NLS0}
	\end{equation}
	in  $\dot H^{s_c} $ has seen significant advancements in recent years. Due to the presence of conserved quantities at the critical regularity, the mass- and energy-critical equations have been the most widely studied.  For the defocusing energy-critical NLS, it is now known that arbitrary data in  $\dot H^1_x$ lead to solutions that are global and scatter. This was proven first for radial initial data by  Bourgain \cite{Bourgain1999},
	Grillakis \cite{Grillakis2000}, and Tao \cite{Tao2005}
	and later for arbitrary data by Colliander- Keel-Staffilani-Takaoka-Tao,  \cite{Colliander2008}, Ryckman-Visan \cite{RyckmanVisan2007} and  Visan \cite{Visan2007,Visan2012} (For results in the  focusing case, see \cite{Dodson2019ASENS,KenigMerle2006,KillipVisan2010}).
	For the  mass-critical NLS, it has also been established that arbitrary data in  $L^2_x$ lead to solutions that are global and scatter. This was proven through the use of 
	minimal counterexamples, first for radial data in dimensions  $d\ge2$ (see \cite{TaoVisanZhang2007,KillipTaoVisan2009,KillipVisanZhang2008}),	  and later for arbitrary data in all dimensions by Dodson  \cite{Dodson2012,Dodson2015,Dodson2016a,Dodson2016b}.

	Killip-Visan \cite{KillipVisan2012} and Visan \cite{Visan2012} revisited the defocusing energy-critical problem in dimensions $d \in \{3,4\}$ from the perspective of minimal counterexamples, utilizing techniques developed by Dodson \cite{Dodson2012}. In particular, they established a "long-time Strichartz estimate" for almost periodic solutions, which serves to rule out the existence of frequency-cascade solutions. Additionally, they derived a frequency-localized interaction Morawetz inequality (which may in turn be used to preclude the existence of soliton-like solutions). 
	
	Unlike the energy- and mass-critical problems, for any other  $s_c\neq 0,1$, there are no conserved quantities that control the growth in time of the   $\dot H^{s_c}$ norm of the solutions.  It is conjectured that, assuming some \textit{a priori} control of a critical norm, global well-posedness and scattering hold for any $s_c > 0$ and in any spatial dimension:
	\begin{Conjection}\label{CNLS0}
		Let $d \geq 1$, $\alpha \geq \frac{4}{d}$, and $s_c = \frac{d}{2} - \frac{2}{\alpha }$. Assume $u: I \times \mathbb{R}^d \rightarrow \mathbb{C}$ is a maximal-lifespan solution to (\ref{NLS0}) such that 
		\begin{equation}
			u \in L_t^\infty \dot{H}_x^{s_c}(I \times \mathbb{R}^d), \notag
		\end{equation}
		then $u$ is global and scatters as $t \to \pm \infty$.
	\end{Conjection}
	The first work dealing with Conjecture \ref{CNLS0} 
	is attributed to  Kenig and Merle \cite{KenigMerle2010} at the case $d = 3, s_c = \frac{1}{2}$ by using their concentration-compactness method developed in \cite{KenigMerle2006} and the scaling-critical Lin-Strauss Morawetz inequality. Subsequently, Murphy \cite{Murphy2014b} extended the methods of \cite{KenigMerle2010} to higher dimensions, resolving Conjecture \ref{CNLS0} for $d \geq 3$ and $s_c = \frac{1}{2}$.
	In the inter-critical case ($0 < s_c < 1$), Murphy \cite{Murphy2014, Murphy2015} developed a long-time Strichartz estimate in the spirit of \cite{Dodson2012} and proved Conjecture \ref{CNLS0} for  the general data  in the case
	\begin{equation}
		\begin{cases}
			\frac{1}{2}\le s_c\le \frac{3}{4},\qquad &d=3\\
			\frac{1}{2}\le s_c<1,&d=4\\
			\frac{1}{2}<s_c<1,&d=5;
		\end{cases}\notag
	\end{equation}
	and for the radial data in the case  $d=3,s_c\in  (0,\frac{1}{2})\cup (\frac{3}{4},1)$.  Later, Gao-Miao-Yang \cite{GaoMiaoYang2019} resolved Conjecture \ref{CNLS0} for radial initial data in the case $d \geq 4$, $0 < s_c < \frac{1}{2}$; Gao-Zhao \cite{GaoZhao2019} resolved Conjecture \ref{CNLS0} for general initial data in the case $d \geq 5$,  $\frac{1}{2} < s_c < 1$. See also \cite{XieFang2013} for earlier partial results regarding these cases. 
	Recently, Yu \cite{Yu2021}  resolved Conjecture \ref{CNLS0} in the case $d = 2, s_c = \frac{1}{2}$, by first developing a long-time Strichartz estimate in the spirit of   \cite{Dodson2016a} and then utilizing the interaction Morawetz estimate from Planchon-Vega \cite{PlanchonVega2009} to exclude the minimal counterexamples. See Table \ref{table1}. 
	
	In the energy-supercritical case ($s_c > 1$), Killip and Visan \cite{KillipVisan2010} were the first to resolve Conjecture \ref{CNLS0} for $d \ge 5$ under certain conditions on $s_c$. Subsequently, Murphy \cite{Murphy2015} addressed the conjecture for radial initial data in the case $d = 3$ and $s_c \in (1, \frac{3}{2})$.
	By developing long-time Strichartz estimates for the energy-supercritical regime, Miao-Murphy-Zheng \cite{MiaoMurphyZheng2014} and Dodson-Miao-Murphy-Zheng \cite{Dodson2017} resolved the Conjecture \ref{CNLS0} for general initial data when $d = 4$ and $1 < s_c \le \frac{3}{2}$. For the case $d = 4$ and $\frac{3}{2} < s_c < 2$ with radial initial data, see the work of Lu and Zheng \cite{LuZheng2017}.
	More recently, Zhao \cite{Zhao2017AMS} and Li-Li \cite{LiLi2022SIAM} resolved the Conjecture \ref{CNLS0} in the case $d \ge 5$ and $1 < s_c < \frac{d}{2}$. For $d \ge 8$, their results also required $\alpha$ to be an even number. See Table 2.
	
	\begin{table}[h]\label{table1}
		\centering
		\caption{Results for Conjecture \ref{CNLS0} in the sub-critical case: $0<s_c<1$}
		\begin{tabular}{|c|c|c|c|}
			\hline
			& $0 < s_c < \frac{1}{2}$ & $s_c=\frac{1}{2}$& $\frac{1}{2} < s_c < 1 $\\
			\hline
			$d = 1 $& \text{\textcolor{blue}{no results}} & \diagbox{}{} & \diagbox{}{}  \\
			\hline
			$d = 2 $& \text{\textcolor{blue}{no results}} & Yu \cite{Yu2021}& \text{\textcolor{blue}{no results}}  \\
			\hline
			$d=3$ & \textcolor{blue}{radial}, Murphy \cite{Murphy2015}&Kenig-Merle \cite{KenigMerle2010} & \thead{$\frac{1}{2}<s_c\le \frac{3}{4}$,Murphy\cite{Murphy2014} \\\textcolor{blue}{radial},  $\frac{3}{4}<s_c<1$, Murphy\cite{Murphy2015}} \\
			\hline 
			$d\ge4$ & \textcolor{blue}{radial}, Gao-Miao-Yang\cite{GaoMiaoYang2019}& Murphy\cite{Murphy2014b} &Gao-Zhao\cite{GaoZhao2019},Murphy\cite{Murphy2014},Xie-Fang\cite{XieFang2013}\\
			\hline 
		\end{tabular}
	\end{table}
	\begin{table}[h]\label{table2}
		\centering
		\caption{Results for Conjecture \ref{CNLS0} in the super-critical case: $1<s_c<\frac{d}{2}$}
		\begin{tabular}{|c|c|}
			\hline
			$d=3$ & $1<s_c<\frac{3}{2}$, \textcolor{blue}{radial}, Murphy \cite{Murphy2015}\\
			\hline 
			$d=4$ & \thead {  $1<s_c<\frac{3}{2}$, Miao-Murphy-Zheng\cite{MiaoMurphyZheng2014}; $s_c=\frac{3}{2}$, Dodson-Miao-Murphy-Zheng\cite{Dodson2017}; \\  $\frac{3}{2}<s_c<2$, \textcolor{blue}{radial},  Lu-Zheng\cite{LuZheng2017}}\\
			\hline 
			$d\ge5$  & \thead {$1<s_c<\frac{d}{2}$, and \textcolor{blue}{assume  $\alpha $ is even when  $d\ge8$}, \\
				Killip-Visan\cite{KillipVisan2010}, Zhao\cite{Zhao2017AMS}, Li-Li\cite{LiLi2022SIAM}}\\
			\hline
		\end{tabular}
	\end{table}

	Analogous to Conjecture \ref{CNLS0}, it is  conjectured that  for the  NLS  in the exterior domain $\Omega$ of a smooth, compact, strictly convex obstacle in $\mathbb{R}^3$: 
	\begin{Conjection}\label{CNLS}
		Let   $\alpha >\frac{4}{3}$ and $s_c = \frac{3}{2} - \frac{2}{\alpha }$. Assume $u: I \times \Omega \rightarrow \mathbb{C}$ is a maximal-lifespan solution to (\ref{NLS}) such that 
		\begin{equation}
			u \in L_t^\infty \dot{H}_D^{s_c}(I \times \Omega), \label{Ebound} 
		\end{equation}
		then $u$ is global and scatters as $t \to \pm \infty$.
	\end{Conjection}
	Killip-Visan-Zhang \cite{KillipVisanZhang2016a} first resolved  Conjecture \ref{CNLS} in the case $d = 3$ and $s_c = 1$. Since this corresponds to the energy-critical setting, the energy conservation law eliminates the need for the assumption (\ref{Ebound}); it suffices to require the initial data to belong to $\dot H^{1}_D(\Omega)$. 
	In this paper, under the assumption that Conjecture \ref{CNLS0} holds in Euclidean space, we resolve Conjecture \ref{CNLS} in the case  $d = 3$ and $\frac{1}{2} \le s_c < \frac{3}{2}$. Our main result is as follows: 
	\begin{theorem}\label{T1}
		Let  $s_c\in [\frac{1}{2},\frac{3}{2})$.  Assume that Conjection \ref{CNLS0} holds. Then Conjection \ref{CNLS} holds. 
	\end{theorem}
	\begin{remark}
		In Section \ref{S4}, we will embed the solutions in the limit geometries into $\Omega$ via the stability theorem \ref{TStability}. To achieve this, we need to assume that Conjecture \ref{CNLS0} holds true, so that the solutions in the limit geometries satisfy uniform spacetime bounds; then the solutions to NLS$_{\Omega}$ will inherit these spacetime bounds.   These solutions to NLS$_{\Omega}$ will appear again as nonlinear profiles in Proposition \ref{Pps}.
	\end{remark}
	\begin{remark}
		As mentioned earlier, Conjecture \ref{CNLS0} has been resolved for $s_c \in [\frac{1}{2}, \frac{3}{4}]$ and $s_c = 1$. Furthermore, for $s_c \in (\frac{3}{4}, 1) \cup (1, \frac{3}{2})$, Murphy \cite{Murphy2015} addressed Conjecture \ref{CNLS0} in the case of radial initial data. Hence, in Theorem \ref{T1}, we only need to assume that Conjecture \ref{CNLS0} holds for non-radial initial data when $s_c \in (\frac{3}{4}, 1) \cup (1, \frac{3}{2})$.
	\end{remark}
	\subsection{Outline of the proof of  Theorem \ref{T1}}
	We proceed by contradiction and assume that Theorem \ref{T1} is false. Observing that Theorem \ref{TLWP} guarantees the global existence and scattering for sufficiently small initial data. From that we deduce the existence of a critical threshold size. Below this threshold, the theorem holds, but above it, solutions with arbitrarily large scattering size can be found. By employing a limiting argument, we establish the existence of minimal counterexamples, which are blowup solutions precisely at the critical threshold. Due to their minimality, these solutions exhibit compactness properties that ultimately conflict with the dispersive nature of the equation. Consequently, we can exclude their existence   and conclude that Theorem \ref{T1} holds.
	
	A key characteristic of these minimal counterexamples is their almost periodicity modulo the symmetries of the equation. We briefly discuss this property and its immediate implications; for a detailed analysis, the reader is referred to \cite{KillipVisan2013}.
	\begin{definition}
		Let  $s_c>0$. A solution  $u:I\times \Omega\rightarrow \mathbb{C}$ to (\ref{NLS}) is called almost periodic if (\ref{Ebound}) holds and 
		there exist function  $C : \mathbb{R}^+ \to \mathbb{R}^+$ such that for all $t \in I$ and all $\eta > 0$,
		\begin{equation}
			\|(-\Delta _\Omega)^{\frac{s_c}{2}}u(t,x)\|_{L^2_x(\Omega\cap \{x:|x|>C(\eta)\})} + \|(-\Delta _\Omega)^{\frac{s_c}{2}}P^\Omega_{>C(\eta)}u(t,x)\|_{L^2_x(\Omega)}<\eta,\notag
		\end{equation}
		where  $P^{\Omega}_{>N} $ denotes the Littlewood-Paley projections adapted to the Dirichlet Laplacian on $\Omega$ (c.f. (\ref{E11121})).  We call   $C$ the \emph{compactness modulus function}.
	\end{definition}
	\begin{remark}
		Using the equivalence of norms in Lemma \ref{LSquare function estimate}, it is straightforward to deduce that when 
		$\{u(t):t\in I\}$
		is precompact in $\dot H^{s_c}_D(\Omega)$,  then   $u:I\times \Omega\rightarrow \mathbb{C}$  is almost periodic  and 
		there exist functions  $C, c : \mathbb{R}^+ \to \mathbb{R}^+$ such that for all $t \in I$ and all $\eta > 0$, 
		\begin{equation}
			\|(-\Delta _\Omega)^{\frac{s_c}{2}}P^\Omega_{<c(\eta)}u(t,x)\|_{L^2_x(\Omega)} + \|(-\Delta _\Omega)^{\frac{s_c}{2}}P^\Omega_{>C(\eta)}u(t,x)\|_{L^2_x(\Omega)}<\eta.\label{E10101}
		\end{equation}
	\end{remark}
	
	To proceed, we require the following result, which relates the interval length of an almost periodic solution to its Strichartz norms. This result can be established by adapting the proof of \cite[Lemma 5.21]{KillipVisan2013} (the only difference being that we need to use the chain rule (\ref{E12133}) instead of the chain rule in Euclidean space).
	\begin{lemma} \label{Lspace-time bound}
		Let  $s_c\in [\frac{1}{2},\frac{3}{2})$, and suppose $u : I \times \Omega \to \mathbb{C}$ is an almost periodic solution to (\ref{NLS}). Then 
		\[
		|I|\lesssim _u	\|(-\Delta _\Omega)^{\frac{s_c}{2}} u \|^2_{L^2_t L^6_x (I \times\Omega)} \lesssim_u 1 +  |I|.
		\]
	\end{lemma}

	With these preliminaries established, we can now describe the first major step in the proof of Theorem  \ref{T1}.
	\begin{theorem}[Reduction to almost periodic solutions]\label{TReduction}
		Suppose that Theorem \ref{T1} fails for some  $s_c\in [\frac{1}{2},\frac{3}{2})$.  Then there exists a global  solution $u : \mathbb{R} \times\Omega \to \mathbb{C}$ to \eqref{NLS} such that $u \in L_t^{\infty} \dot{H}_D^{s_c}(\mathbb{R}  \times \Omega)$,   whose orbit  $\{u(t):t\in \mathbb{R} \}$ is precompact in  $\dot H^{s_c}_D(\Omega)$ and there exists  $R>0$ such that  
		\begin{equation}
			\int _{\Omega\cap \{|x|\le R\}}|u(t,x)|^{\frac{3\alpha }{2}}dx\gtrsim1 \quad\text{uniformly for }\quad t\in \mathbb{R} .\label{E}
		\end{equation}
	\end{theorem}
	\begin{remark}
		Indeed, our proof shows that Theorem \ref{TReduction} is   valid for all  $s_c \in (0, \frac{3}{2})$. The restriction  $ s_c \geq \frac{1}{2}$ in Theorem \ref{T1} arises from the limitations imposed by the indices in Theorem \ref{TEquivalence}, which make it challenging to exclude almost periodic solutions when  $s_c\in (0,\frac{1}{2})$. See Remark \ref{R128} for more details. 
	\end{remark}
	
	The reduction to almost periodic solutions is now widely regarded as a standard technique in the study of dispersive equations at critical regularity. Keraani \cite{Keraani2006JFA} was the first to prove the existence of minimal blowup solutions, while Kenig-Merle \cite{KenigMerle2006} were the first to use them to establish a global well-posedness result. Since then, this technique has proven to be extremely useful; see \cite{KenigMerle2010,KillipTaoVisan2009,KillipVisan2010,KillipVisan2010AJM,KillipVisan2013,KillipVisan2012,KillipVisanZhang2008,MiaoMurphyZheng2014,Murphy2014,Murphy2014b,Murphy2015} for many more examples of this technique in action (and note that this is by no means an exhaustive list). For a good introduction to these methods, see \cite{KillipVisan2013}.
	
	The proof of Theorem \ref{TReduction} relies on three key components. 
	
	First, the linear and nonlinear profile decompositions are required. For the linear profile decomposition, the case $s_c = 1$ was established in \cite{KillipVisanZhang2016a}, and we will follow the methodology outlined in that work. The main tool used to derive the linear profile decomposition is the inverse Strichartz inequality. This inequality shows that a solution with non-trivial spacetime bounds must concentrate at least one bubble. By repeatedly applying the inverse Strichartz inequality, it can be demonstrated that the linear solution concentrates on multiple bubbles, with the remainder term vanishing after passing to a subsequence.
	
	After obtaining the linear profile decomposition, the next step is to construct the nonlinear profiles. These nonlinear profiles are solutions to NLS$_\Omega$ with initial data corresponding to the linear profiles. Due to the presence of the boundary, suitable scaling and spatial translations lead to the study of NLS in different geometries, which significantly distinguishes our setting from the Euclidean setting. The main challenge is that we cannot guarantee whether a profile with given initial data is entirely contained within the exterior domain. Additionally, the profile may exist at any scale and any possible location. To address this, we adopt the approach from \cite{KillipVisanZhang2016a}, which associates each profile with a specific limiting case. 
	
	Moreover, we consider three scenarios arising from the scaling and spatial translation of $\Omega$. The rescaled domain is denoted as $\Omega_n = \lambda_n^{-1}(\Omega - \{x_n\})$ for the first two cases and $\Omega_n = \lambda_n^{-1} R_n^{-1}(\Omega - \{x_n^*\})$ for the third case, where $x_n^* \in \partial \Omega$, $|x_n - x_n^*| = \operatorname{dist}(x_n, \Omega^c)$, and $R_n \in \operatorname{SO}(3)$ satisfies $R_n e_3 = \frac{x_n - x_n^*}{|x_n - x_n^*|}$. These scenarios are as follows:
	\begin{enumerate}
		\item When $\lambda_n \to \infty$, the rescaled domain $\Omega_n$ approximates $\mathbb{R}^3$.
		\item When $\frac{\operatorname{dist}(x_n, \Omega^c)}{\lambda_n} \to \infty$, the domain $\Omega_n^c$ retreats to infinity.
		\item When $\lambda_n \to 0$ and $\frac{\operatorname{dist}(x_n, \Omega^c)}{\lambda_n} = K > 0$, the domain $\Omega_n$ approximates a half-space.
	\end{enumerate}

	The second ingredient is a stability result for the nonlinear equation (see e.g. Theorem \ref{TStability} below). The third ingredient is a decoupling statement for nonlinear profiles.
	The last two ingredients are closely related, in the sense that the decoupling must hold in a space that is dictated by the stability theory. Most precisely, this means that the decoupling must hold in a space with $s_c$ derivatives. Keraani \cite{Keraani2001} showed how to prove such a decoupling statement in the context of the mass- and  energy-critical NLS; however, these arguments rely on pointwise estimates to bound the difference of nonlinearities and hence fail to be directly applicable in the presence of fractional derivatives. In \cite{KillipVisan2010}, Killip and Visan devised a strategy that is applicable in the energy-supercritical setting, while Murphy \cite{Murphy2014} developed a strategy tailored to the energy-subcritical setting. 
	In particular, by employing a Strichartz square function that provides estimates equivalent to those of $|\nabla|^{s_c}$, they can reduce the problem to a framework where Keraani's arguments can be directly applied.
	In this paper, we adopt the strategies presented in \cite{KillipVisan2010,Murphy2014}.  Specifically, by appropriately selecting the parameters and applying the equivalence theorem (Theorem \ref{TEquivalence}), we reduce the proof of the decoupling for nonlinear profiles to the cases addressed in \cite{KillipVisan2010,Murphy2014}. 
	
	With all the necessary tools in place, we can now apply the standard arguments in \cite{KillipVisan2013} to establish Theorem \ref{TReduction}. Therefore, to complete the proof of Theorem \ref{T1}, it is sufficient to rule out the existence of the solutions described in Theorem \ref{TReduction}. For this purpose, we will utilize versions of the Lin-Strauss Morawetz inequality:  
	\begin{equation}
		\int \int _{I\times \Omega}\frac{|u(t,x)|^{\alpha +2}}{|x|}dxdt\lesssim  \||\nabla |^{1/2}u\|_{L^\infty _tL_x^2(I\times \Omega)}^2, \label{E1242}
	\end{equation}
	which will be applied in Section \ref{S6} to exclude the existence of almost periodic solutions in Theorem \ref{TReduction} for the case $s_c = \frac{1}{2}$. However, when $s_c > \frac{1}{2}$, the estimate (\ref{E1242}) cannot be directly applied because the solutions considered only belong to $\dot H^{s_c}_D(\Omega)$, which means the right-hand side of (\ref{E1242}) might not be finite.
	
	For $s_c > \frac{1}{2}$, it is necessary to suppress the low-frequency components of the solutions to make use of the estimate (\ref{E1242}). In the context of the 3D radial energy-critical NLS, Bourgain \cite{Bourgain1999} achieved this by proving a space-localized version of (\ref{E1242})   (see also \cite{Grillakis2000,TaoVisanZhang2007}). In Section \ref{S6}, we adopt a similar approach to preclude the existence of almost periodic solutions in Theorem \ref{TReduction} for the range $1 < s_c < 3/2$. However, since one of the error terms arising from space localization requires controlling the solution at the $\dot{H}_D^1$ level, a different strategy is needed for the range $\frac{1}{2} < s_c < 1$. 
	To address this, in Section \ref{S1/2-1}, we develop a version of (\ref{E1242}) localized to high frequencies. This high-frequency localized version will   be employed to exclude the existence of almost periodic solutions in Theorem \ref{TReduction} when $\frac{1}{2} < s_c < 1$.
	
	The structure of the paper is as follows:
	
	In Section \ref{S2}, we introduce the necessary notation and foundational materials for the analysis. This includes the equivalence of Sobolev spaces and the product rule for the Dirichlet Laplacian; Littlewood-Paley theory and Bernstein inequalities; Strichartz estimates; local and stability theories for (\ref{NLS}); local smoothing; the convergence of functions related to the Dirichlet Laplacian as the underlying domains converge; and the behavior of the linear propagator in the context of domain convergence.
	
	Section \ref{S3} begins with the proof of the refined and inverse Strichartz inequalities (Proposition \ref{PRefined SZ} and Proposition \ref{inverse-strichartz}). These results establish that linear evolutions with  non-trivial spacetime norms must exhibit a bubble of concentration, which is then used to derive the linear profile decomposition for the propagator $e^{it\Delta_\Omega}$ in $\dot{H}^{s_c}_D(\Omega)$ (see Theorem \ref{linear-profile}).
	
	In Section \ref{S4}, we show that nonlinear solutions in the limiting geometries can be embedded into $\Omega$. Since nonlinear solutions in the limiting geometries admit global spacetime bounds (Here we need to assume that Conjecture \ref{CNLS0} holds true), we deduce that solutions to NLS$_{\Omega}$, whose characteristic length scale and location conform closely with one of these limiting cases, inherit these spacetime bounds. These solutions to NLS$_{\Omega}$ will  reappear as nonlinear profiles in Section \ref{S5}. 
	
	Section \ref{S5} is dedicated to proving the existence of almost periodic solutions (Theorem \ref{TReduction}). The key step involves establishing the Palais-Smale condition (Proposition \ref{Pps}). This is achieved using the profile decomposition developed in Section \ref{S4}, the stability theorem (Theorem \ref{TStability}) from Section \ref{S2}, and techniques from \cite{KillipVisan2010, Murphy2014} to ensure the decoupling of nonlinear profiles.
	
	In Section \ref{S6}, we rule out almost periodic solutions described in Theorem \ref{TReduction} for $1 < s_c < \frac{3}{2}$ and $s_c = \frac{1}{2}$. The proof relies on a space-localized Lin-Strauss Morawetz inequality, following the method of Bourgain \cite{Bourgain1999}.
	
	Finally, in Section \ref{S1/2-1}, we exclude solutions as in Theorem \ref{TReduction} for $\frac{1}{2} < s_c < 1$. The main tool is the long-time Strichartz estimate (Proposition \ref{PLT2}), originally developed by Dodson \cite{Dodson2012} for the mass-critical NLS. Additionally, we establish a frequency-localized Lin-Strauss Morawetz inequality (Proposition \ref{PMorawetz}) to eliminate almost periodic solutions. This approach involves truncating the solution to high frequencies and employing Proposition \ref{PLT2} to handle the error terms introduced by frequency projection.


	\section{Preliminaries}\label{S2}
	\subsection{Notation and useful lemmas}
	We express $ X \lesssim Y $ or $ Y \gtrsim X $ to denote that $ X \leq CY $ for some absolute constant $ C > 0 $, which might change from line to line. If the implicit constant relies on additional variables, this will be shown with subscripts. We employ $ O(Y) $ to represent any quantity $ X $ such that $ |X| \lesssim Y $. The notation $ X \sim Y $ implies that $ X \lesssim Y \lesssim X $. The term $ o(1) $ is used to describe a quantity that converges to zero. We will also use  $s+$ or  $s-$, which means that there exists a small positive number $ \varepsilon  $ such that it is equal to   $s+\varepsilon $ or  $s-\varepsilon $ respectively.

	Throughout this paper, we let $s_c = \frac{3}{2} - \frac{2}{\alpha} \in (0, \frac{3}{2})$. Further restrictions on the range of $s_c$ are imposed only in Section \ref{S6} and Section \ref{S1/2-1}. $ \Omega $ will stand for the exterior domain of a smooth, compact, strictly convex obstacle in $ \mathbb{R}^3 $. Without loss of generality, we assume 
	$0 \in \Omega^c$. The notation $\text{diam} := \text{diam}(\Omega^c)$ is used to denote the diameter of the obstacle, and $d(x) := \text{dist}(x, \Omega^c)$ denotes the distance from a point $x \in \mathbb{R}^3$ to the obstacle. 
	
	We first state the Hardy inequality on the exterior domain.
	\begin{lemma}[Hardy's inequality, \cite{KillipVisanZhang2016b}]
		Let $d\geq3$, $1<p<\infty$ and $0<s<\min\{1+\frac{1}{p},\frac{3}{p}\}$, then for any $f\in C_c^\infty(\Omega)$, we have
		\begin{align*}
			\Big\|\frac{f(x)}{d(x)}\big\|_{L^p(\Omega)}\lesssim\big\|(-\Delta_\Omega)^\frac{s}{2}f\big\|_{L^p(\Omega)},
		\end{align*} 
		where $d(x)=\operatorname{dist}(x,\Omega^c)$.
	\end{lemma}
	We will use the following refined version of Fatou's lemma due to Brezis and Lieb.
	
	\begin{lemma}[Refined Fatou, \cite{BrezisLieb1983}]\label{LRefinedFatou}
		Let $0 < p < \infty$ and assume that $\{f_n\} \subset L^p(\mathbb{R}^d)$ with 
		$\limsup_{n \to \infty} \|f_n\|_p < \infty$. If $f_n \to f$ almost everywhere, then 
		\[
		\int_{\mathbb{R}^d} \left| |f_n|^p - |f_n - f|^p - |f|^p \right| dx \to 0 \quad \text{as} \quad n \to \infty.
		\]
		In particular, $\|f_n\|_{L^p}^p - \|f_n - f\|_{L^p}^p \to \|f\|_{L^p}^p$.
	\end{lemma}
	
	The following fractional difference estimate will be used in the proof of Lemma \ref{Lnonlinearestimate}. 
	\begin{lemma}[Derivatives of differences, \cite{KillipVisan2010}]\label{LDerivatives of differences}
		Let $F(u) = |u|^p  u$ with $p > 0$ and let $0 < s < 1$. Then for $1 < q, q_1, q_2 < \infty$ such that $\frac{1}{q} = \frac{p}{q_1} + \frac{1 }{q_2}$, we have
		\[
		\|\nabla|^s [F(u+v) - F(u)] \|_{L^q(\mathbb{R} ^d)} \lesssim \|\nabla|^s u\|_{L^{q_1}(\mathbb{R} ^d)}^{p } \|v\|_{L^{q_2}(\mathbb{R} ^d)} + \|\nabla|^s v\|_{L^{q_1}(\mathbb{R} ^d)} ^{p }\|u+v\|_{L^{q_2}(\mathbb{R} ^d)}.
		\]
	\end{lemma}
	
	We will  also use the following heat kernel estimate due to Q. S. Zhang \cite{Zhang2003}.
	\begin{lemma}[Heat kernel estimate \cite{Zhang2003}]\label{Lheatkernel}
		Let $\Omega$ denote the exterior of a smooth, compact, convex obstacle in $\mathbb{R}^d$ for $d \geq 3$. Then there exists $c > 0$ such that
		\[
		|e^{t\Delta_\Omega}(x,y)| \lesssim \left( \frac{d(x)}{\sqrt{t} \wedge \text{diam}} \wedge 1 \right) \left( \frac{d(y)}{\sqrt{t} \wedge \text{diam}} \wedge 1 \right) e^{-\frac{c|x - y|^2}{t}} t^{-\frac{d}{2}},
		\]
		uniformly for $x, y \in \Omega$ and  $t\ge0$; recall that   $A\wedge B=\min \{A,B\}$. Moreover, the reverse inequality holds after suitable modification of  $c$ and the implicit constant.
	\end{lemma}
	
	There is a natural family of Sobolev spaces associated with powers of the Dirichlet Laplacian. Our notation for these is as follows.
	
	\begin{definition} 
		For $s \geq 0$ and $1 < p < \infty$, let $\dot{H}^{s,p}_D(\Omega)$ and $H^{s,p}_D(\Omega)$ denote the completions of $C_c^{\infty}(\Omega)$ under the norms
		\[
		\|f\|_{\dot{H}^{s,p}_D(\Omega)} := \|(-\Delta_{\Omega})^{s/2} f\|_{L^p} \quad \text{and} \quad \|f\|_{H^{s,p}_D(\Omega)} := \|(1 - \Delta_{\Omega})^{s/2} f\|_{L^p}.
		\]
		When $p = 2$ we write $\dot{H}^s_D(\Omega)$ and $H^s_D(\Omega)$ for $\dot{H}^{s,2}_D(\Omega)$ and $H^{s,2}_D(\Omega)$, respectively.
	\end{definition}
	
	The following result from \cite{KillipVisanZhang2016c}   establishes a connection between Sobolev spaces defined with respect to the Dirichlet Laplacian and those defined through conventional Fourier multipliers. The constraints on regularity $ s $ are important, as shown by counterexamples in \cite{KillipVisanZhang2016c}.

	\begin{theorem}[Equivalence of Sobolev spaces,\cite{KillipVisanZhang2016c}]\label{TEquivalence}
		Let $ d \geq 3 $ and let $ \Omega $ denote the complement of a compact convex body $ \Omega^c \subset \mathbb{R}^d $ with smooth boundary. Let $ 1 < p < \infty $. If $ 0 \leq s < \min \left\{ 1 + \frac{1}{p}, \frac{d}{p} \right\} $, then
		\[
		\|(-\Delta_{\mathbb{R}^d})^{s/2} f\|_{L^p} \sim_{d,p,s} \|(-\Delta_{\Omega})^{s/2} f\|_{L^p} \quad \text{for all } f \in C_c^\infty(\Omega).
		\]
	\end{theorem}
	
	This result allows us to transfer the $L^p$-product rule for fractional derivatives and the chain rule directly from the Euclidean setting, provided we respect the restrictions on $s$ and $p$.  
	\begin{lemma}\label{LFractional product rule}
		For all $f, g \in C_c^\infty(\Omega)$, we have
		\[
		\|(-\Delta_\Omega)^{s/2} (fg)\|_{L^p(\Omega)} \lesssim \|(-\Delta_\Omega)^{s/2} f\|_{L^{p_1}(\Omega)} \|g\|_{L^{p_2}(\Omega)} + \|f\|_{L^{q_1}(\Omega)} \|(-\Delta_\Omega)^{s/2} g\|_{L^{q_2}(\Omega)}
		\]
		with the exponents satisfying $1 < p, p_1, q_2 < \infty$, $1 < p_2, q_1 \leq \infty$,
		\[
		\frac{1}{p} = \frac{1}{p_1} + \frac{1}{p_2} = \frac{1}{q_1} + \frac{1}{q_2},\quad\text{and}\quad 0 < s < \min \left\{ 1 + \frac{1}{p_1}, 1 + \frac{1}{q_2}, \frac{3}{p_1}, \frac{3}{q_2} \right\}.
		\]
	\end{lemma}
	\begin{lemma}\label{LChainrule}
		Suppose  $G\in C^2(\mathbb{C})$ and  $1<p,p_1,p_2<\infty $ are such that  $\frac{1}{p}=\frac{1}{p_1}+\frac{1}{p_2}$. Then for all  $0<s<\min \left\{ 2,\frac{3}{p_2} \right\}$,
		\begin{equation}
			\|(-\Delta _\Omega)^{\frac{s}{2}}G(u)\|_{L^p(\Omega)}\lesssim   \|G'(u)\|_{L^{p_1}(\Omega)} \|(-\Delta _\Omega)^{\frac{s}{2}}u\|_{L^{p_2}(\Omega)}.\notag   
		\end{equation}
	\end{lemma}
	In particular, in Section \ref{S1/2-1}, we will use the following fractional chain rule:
	\begin{corollary}
		Given  $u\in L_t^{\infty }\dot H^{s_c}_D (I\times \Omega)\cap L_t^{2}\dot H^{s_c,6}_D(I\times \Omega)$,
		\begin{equation}
			\|(-\Delta _\Omega)^{\frac{s_c}{2}}(|u|^{\alpha }u)\|_{L_t^{2}L_x^{\frac{6}{5}}(I\times \Omega)}\lesssim   \|(-\Delta _\Omega)^{\frac{s_c}{2}}u\|_{L_t^{\infty }L_x^{2}}^{\alpha } \|(-\Delta _\Omega)^{\frac{s_c}{2}}u\|_{L_t^{2}L_x^{6}(I\times \Omega)}.\label{E12133}
		\end{equation}
	\end{corollary}
	\begin{proof}
		Using the equivalence theorem \ref{TEquivalence}, the chain rule in Euclidean space, and applying the equivalence theorem \ref{TEquivalence} again, we obtain  
		\begin{equation}
			\|(-\Delta_\Omega)^{\frac{s_c}{2}}(|u|^{\alpha}u)\|_{L_t^{2}L_x^{\frac{6}{5}}(I \times \Omega)} \lesssim \|u\|_{L_t^{2\alpha}L_x^{3\alpha}(I \times \Omega)}^{\alpha} \|(-\Delta_\Omega)^{\frac{s_c}{2}}u\|_{L_t^{\infty}L_x^{2}(I \times \Omega)}. \label{E12131}
		\end{equation}  
		Moreover, by Sobolev embedding and H\"older's inequality, we have  
		\begin{equation}
			\|u\|_{L_t^{2\alpha}L_x^{3\alpha}(I \times \Omega)}^{\alpha} \lesssim \|(-\Delta_\Omega)^{\frac{s_c}{2}}u\|_{L_t^{2\alpha}L_x^{\frac{6\alpha}{3\alpha - 2}}(I \times \Omega)}^{\alpha} \lesssim \|(-\Delta_\Omega)^{\frac{s_c}{2}}u\|_{L_t^{\infty}L_x^{2}(I\times \Omega)}^{\alpha-1} \|(-\Delta_\Omega)^{\frac{s_c}{2}}u\|_{L_t^{2}L_x^{6}(I \times \Omega)}. \label{E12132}
		\end{equation}  
		Substituting (\ref{E12132}) into  (\ref{E12131}), we obtain the desired inequality (\ref{E12133}).
	\end{proof}
	We will also use the local smoothing estimate. The particular version we need is \cite[Lemma 2.13]{KillipVisanZhang2016a}.
	\begin{lemma} \label{LLocalSmoothing}
		Let $u = e^{it\Delta_\Omega} u_0$. Then
		\[
		\int_{\mathbb{R}} \int_\Omega |\nabla u(t, x)|^2 \langle R^{-1} (x-z) \rangle^{-3} dx dt \lesssim R \| u_0 \|_{L^2(\Omega)} \|\nabla u_0 \|_{L^2(\Omega)},
		\]
		uniformly for $z \in \mathbb{R}^3$ and $R > 0$.
	\end{lemma}
	A direct consequence of the local smoothing estimate is the following result, which will be used   to prove Lemma  \ref{LDecoupling of nonlinear profiles}.  
	\begin{corollary}\label{CLocalsmoothing}
		Given $w_0 \in \dot{H}^{s_c}_D(\Omega)$,
		\[
		\|(-\Delta _\Omega)^{\frac{s_c}{2}} e^{it\Delta_\Omega} w_0 \|_{ L_{t,x}^{2}([\tau-T, \tau+T] \times \{|x-z| \leq R\})} \lesssim T^{\frac{2(5\alpha -4)}{10\alpha (s_c+2)}} R^{\frac{15\alpha -4}{10\alpha (s_c+2)}}  \| e^{it\Delta_\Omega} w_0 \|^{\frac{1}{2(s_c+2)}}_{L_{t,x}^{\frac{5\alpha }{2}}(\mathbb{R} \times \Omega)} \| w_0 \|_{\dot{H}^{s_c}_D(\Omega)}^{1-\frac{1}{2(s_c+2)}},
		\]
		uniformly in $w_0$ and the parameters $R, T > 0, \tau \in \mathbb{R}$, and $z \in \mathbb{R}^3$.
	\end{corollary}
	\begin{proof}
		Replacing  $w_0$ by   $e^{i\tau \Delta _\Omega}w_0$, we see that it suffices to treat the case   $\tau=0$. 
		
		Given $N > 0$, using the H\"older, Bernstein, and Strichartz inequalities, as well as the equivalence of Sobolev spaces, we have
		\begin{align*}
			\||\nabla |^{s_c}&e^{it\Delta_\Omega} P^{\Omega}_{<N} w_0\|_{L^2_{t,x}([-T,T] \times \{|x-z| \leq R\})} \notag\\
			&\lesssim T^{\frac{5\alpha -4}{10\alpha }}R^{\frac{3(5\alpha +4)}{40\alpha }} \||\nabla|^{s_c} e^{it\Delta_\Omega} P^{\Omega}_{<N} w_0\|_{L_t^{\frac{5\alpha }{2}}L_x^{\frac{40\alpha }{15\alpha -4}}} \\
			&\lesssim T^{\frac{5\alpha -4}{10\alpha }}R^{\frac{3(5\alpha +4)}{40\alpha }} N^{\frac{s_c}{4}}\||\nabla|^{\frac{3}{4}s_c} e^{it\Delta_\Omega} P^{\Omega}_{<N} w_0\|_{L_t^{\frac{5\alpha }{2}}L_x^{\frac{40\alpha }{15\alpha -4}}}\\
			&\lesssim T^{\frac{5\alpha -4}{10\alpha }}R^{\frac{3(5\alpha +4)}{40\alpha }} N^{\frac{s_c}{4}}  \|e^{it\Delta _\Omega}P^{\Omega}_{<N}w_0\|_{L_{t,x}^{\frac{5\alpha }{2}}}^{\frac{1}{4}}   \||\nabla |^{s_c}e^{it\Delta _\Omega}P^\Omega_{\le N}w_0\|_{L_t^{\frac{5\alpha }{2}}L_x^{\frac{30\alpha }{15\alpha -8}}}^{\frac{3}{4}}\\
			&\lesssim T^{\frac{5\alpha -4}{10\alpha }}R^{\frac{3(5\alpha +4)}{40\alpha }} N^{\frac{s_c}{4}}  \|e^{it\Delta _\Omega}P^{\Omega}_{<N}w_0\|_{L_{t,x}^{\frac{5\alpha }{2}}}^{\frac{1}{4}} \|w_0\|_{\dot H^{s_c}_D(\Omega)}^{\frac{3}{4}} .
		\end{align*}
		
		We estimate the high frequencies using Lemma \ref{LLocalSmoothing} and the Bernstein inequality:
		\begin{align*}
			\||\nabla|^{s_c} &e^{it\Delta_\Omega} P^{\Omega}_{\geq N} w_0\|_{L^2_{t,x}([-T,T] \times \{|x-z| \leq R\})}^2 \notag\\
			&\lesssim R \|P^{\Omega}_{\geq N} |\nabla |^{s_c-1}w_0\|_{L_x^2} \||\nabla|^{s_c} P^{\Omega}_{\geq N} w_0\|_{L_x^2} \lesssim R N^{-1} \|w_0\|_{\dot{H}_D^{s_c}(\Omega)}^2.
		\end{align*}
		
		The desired estimate  in Corollary \ref{CLocalsmoothing} now follows by optimizing in the choice of $N$. 
	\end{proof}
	\subsection{Littlewood-Paley theory on exterior domains}
	Let $ \phi : [0, \infty) \to [0, 1]$ be a smooth, non-negative function satisfying  
	\[
	\phi(\lambda) = 1 \quad \text{for } 0 \leq \lambda \leq 1, \quad \text{and} \quad \phi(\lambda) = 0 \quad \text{for } \lambda \geq 2.
	\]
	For each dyadic number $N \in 2^\mathbb{Z}$, define  
	\[
	\phi_N(\lambda) := \phi(\lambda/N), \quad \psi_N(\lambda) := \phi_N(\lambda) - \phi_{N/2}(\lambda).
	\]
	Observe that the collection $\{\psi_N(\lambda)\}_{N \in 2^\mathbb{Z}}$ forms a partition of unity on $(0, \infty)$.
	
	Using these functions, we define the Littlewood-Paley projections adapted to the Dirichlet Laplacian on $\Omega$ through the functional calculus for self-adjoint operators:  
	\begin{equation}
		P_{\leq N}^\Omega := \phi_N(\sqrt{-\Delta_\Omega}), \quad P_N^\Omega := \psi_N(\sqrt{-\Delta_\Omega}), \quad P_{> N}^\Omega := I - P_{\leq N}^\Omega. \label{E11121}
	\end{equation}
	
	For simplicity, we will frequently denote $f_N := P_N^\Omega f$ and similarly for other projections.
	
	We will also use $P_N^{\mathbb{R}^3}$ and similar notation to refer to the corresponding operators for the standard Laplacian on $\mathbb{R}^3$. Additionally, we will require analogous operators on the half-space $\mathbb{H} = \{x \in \mathbb{R}^3 : x \cdot e_3 > 0\}$, where $e_3 = (0, 0, 1)$. These operators are denoted by $P_N^\mathbb{H}$, and so on.

	Just like their Euclidean counterparts,  the following two basic estimates are well-known.
	
	\begin{lemma}[Bernstein estimates,\cite{KillipVisanZhang2016c}]\label{LBernstein estimates}
		For any $f \in C_c^\infty(\Omega)$, we have
		\[
		\|P_{\leq N}^\Omega f\|_{L^p(\Omega)} + \|P_N^\Omega f\|_{L^p(\Omega)} \lesssim \|f\|_{L^p(\Omega)} \quad \text{for } 1 < p < \infty,
		\]
		\[
		\|P_{\leq N}^\Omega f\|_{L^q(\Omega)} + \|P_N^\Omega f\|_{L^q(\Omega)} \lesssim N^{3\left(\frac{1}{p} - \frac{1}{q}\right)} \|f\|_{L^p(\Omega)} \quad \text{for } 1 \leq p < q \leq \infty,
		\]
		\[
		N^s \|P_N^\Omega f\|_{L^p(\Omega)} \sim \|(-\Delta_\Omega)^{s/2} P_N^\Omega f\|_{L^p(\Omega)} \quad \text{for } 1 < p < \infty \text{ and } s \in \mathbb{R}.
		\]
		Here, the implicit constants depend only on $p$, $q$, and $s$.
	\end{lemma}
	\begin{lemma}[Square function estimate,\cite{KillipVisanZhang2016c}]\label{LSquare function estimate}
		Fix $1 < p < \infty$. For all $f \in C_c^\infty(\Omega)$, 
		\[
		\|f\|_{L^p(\Omega)} \sim \left\|\left( \sum_{N \in 2^\mathbb{Z}} |P_N^\Omega f|^2 \right)^{\frac{1}{2}} \right\|_{L^p(\Omega)}.
		\]
	\end{lemma}

	\subsection{Strichartz estimates, local well-posedness, and the stability result}
	Strichartz estimates for domains exterior to a compact, smooth, strictly convex obstacle were proved by Ivanovici \cite{Ivanovici2010a} with the exception of the endpoint  $L^2_tL^6_x$, see also \cite{BlairSmithSogge2012}.  
	Subsequently, Ivanovici and Lebeau \cite{IvanoviciLebeau2017} proved the dispersive estimate for $d = 3 $.
	\begin{lemma}[Dispersive estimate, \cite{IvanoviciLebeau2017}]\label{LDispersive}
		\begin{equation}
			\| e^{it\Delta_{\Omega}} f \|_{L_x^{\infty}(\Omega)} \lesssim |t|^{-\frac{3}{2}} \|f\|_{L_x^1(\Omega)}.\label{E11122}
		\end{equation}
	\end{lemma}
	For $d \geq 4$, Ivanovici and Lebeau \cite{IvanoviciLebeau2017}  also demonstrated through the construction of explicit counterexamples that the dispersive estimate no longer holds, even for the exterior of the unit ball. 
	However, for $d=5,7$, Li-Xu-Zhang \cite{LiXuZhang2014} established the dispersive estimates for solutions with radially symmetric initial data outside the unit ball.

	Combining the dispersive estimate (\ref{E11122}) with the Theorem of  Keel-Tao\cite{KeelTao1998AJM}, we obtain  the following Strichartz estimates:
	\begin{proposition}[Strichartz estimates \cite{Ivanovici2010a,BlairSmithSogge2012,IvanoviciLebeau2017}]\label{PStrichartz}
		Let $q, \tilde{q} \geq 2$, and $2 \leq r, \tilde{r} \leq \infty$ satisfying 
		\[
		\frac{2}{q} + \frac{3}{r} = \frac{2}{\tilde{q}} + \frac{3}{\tilde{r}}=	\frac{3}{2} .
		\]
		Then, the solution $u$ to $(i\partial_t + \Delta_\Omega)u = F$ on an interval $I \ni 0$ satisfies
		\[
		\|u\|_{L_t^q L_x^r(I \times \Omega)} \lesssim \|u_0\|_{L_x^2(\Omega)} + \|F\|_{L_t^{\tilde{q}'} L_x^{\tilde{r}'}(I \times \Omega)}.
		\tag{2.3}
		\]
	\end{proposition}
	By the Strichartz estimate and the standard contraction mapping principle, we can establish the following local well-posedness result.
	\begin{theorem} \label{TLWP}
		Let $\Omega \subset \mathbb{R}^3$ be the exterior of a smooth compact strictly convex obstacle. There exists $\eta > 0$ such that if $u_0 \in \dot H_D^{s_c}(\Omega)$ obeys
		\begin{equation}
			\|(-\Delta _\Omega)^{\frac{s_c}{2}} e^{it \Delta_\Omega} u_0 \|_{L_t^{\frac{5\alpha }{2}} L_x^{\frac{30\alpha }{15\alpha -8}}(I \times \Omega)} \leq \eta  \label{E10201}
		\end{equation}
		for some time interval $I \ni 0$, then there is a unique strong solution to (\ref{NLS}) on the time interval $I$; moreover,
		\[
		\|(-\Delta _\Omega)^{\frac{s_c}{2}}u\|_{L_t^{\frac{5\alpha }{2}} L_x^{\frac{30\alpha }{15\alpha -8}}(I \times \Omega)} \lesssim \eta. 
		\]
	\end{theorem}
	\begin{remark}
		\ 
		\begin{enumerate}
			\item If $u_0$ has small $\dot{H}^{s_c}_D(\Omega)$ norm, then Proposition \ref{PStrichartz} guarantees that  (\ref{E10201}) holds with $I = \mathbb{R}$. Thus global well-posedness for small data is a corollary of this theorem.
			
			\item For large initial data $u_0$, the existence of some small open interval $I \ni 0$ for which (\ref{E10201}) holds follows from combining the monotone convergence theorem with Proposition \ref{PStrichartz}. In this way, we obtain local well-posedness for all data in $\dot H^{s_c}_D(\Omega)$.
			
			\item The argument below holds equally well for initial data prescribed as $t \to \pm \infty$, thus proving the existence of wave operators.
		\end{enumerate}
	\end{remark}
	\begin{proof}
		Throughout the proof, all space-time norms will be on $I \times \Omega$.  Consider the map  
		\begin{equation}
			\Phi: u \mapsto e^{it\Delta _\Omega}u_0-i\int_0^te^{i(t-s)\Delta _\Omega}(|u|^{\alpha }u)(s)ds.\notag
		\end{equation}
		We will show this is a contraction on the ball
		\[
		B := \left\{ u \in L_t^{\infty} \dot H_D^{s_c} \cap L_t^{ \frac{5\alpha }{2}} \dot  H_D^{s_c, \frac{30\alpha }{15\alpha -8}}  : \| (-\Delta _\Omega)^{\frac{s_c}{2}}u \|_{L_t^{\frac{5\alpha }{2}} L_x^{\frac{30\alpha }{15\alpha -8}}} \leq 2\eta, \right.
		\]
		\[
		\text{and }\left. \| u \|_{L_t^{\infty} \dot H_D^{s_c}} \leq 2 \| u_0 \|_{\dot H_D^{s_c}}, \| u \|_{L_t^{\frac{5\alpha }{2}} L_x^{\frac{5\alpha }{2}}}\leq 2C \eta  \right\}
		\]
		under the metric given by
		\[
		d(u,v) := \| u - v \|_{L_t^{\frac{5\alpha }{2}} L_x^{\frac{30\alpha }{15\alpha -8}}}.
		\]

		To see that $\Phi$ maps the ball $B$ to itself, we use the Strichartz inequality followed by Lemma \ref{LFractional product rule},  (\ref{E10201}), Sobolev embedding, and then   Theorem \ref{TEquivalence}:
		\begin{align}
			&\| (-\Delta _\Omega)^{\frac{s_c}{2}} \Phi(u) \|_{L_t^{\frac{5\alpha }{2}} L_x^{\frac{30\alpha }{15\alpha -8}}}
			\notag\\
			&\leq \| (-\Delta _\Omega)^{\frac{s_c}{2}} e^{it\Delta_{\Omega}} u_0 \|_{L_t^{\frac{5\alpha }{2}} L_x^{\frac{30\alpha }{15\alpha -8}}}  + C \left\| (-\Delta _\Omega)^{\frac{s_c}{2}} \left( |u|^{\alpha } u \right) \right\|_{L_t^{ \frac{5\alpha }{2(\alpha +1)}} L_x^{\frac{30\alpha }{27\alpha -8}}} \notag\\
			&\leq \eta + C \| u \|_{L_t^{\frac{5\alpha }{2}} L_x^{\frac{5\alpha }{2}}}  ^{\alpha }\| (-\Delta _\Omega)^{\frac{s_c}{2}} u \|_{L_t^{\frac{5\alpha }{2}} L_x^{\frac{30\alpha }{15\alpha -8}}}\leq \eta + C \| (-\Delta _\Omega)^{\frac{s_c}{2}} u \|_{L_t^{\frac{5\alpha }{2}} L_x^{\frac{30\alpha }{15\alpha -8}}}^{\alpha +1}\notag\\
			&\le \eta +C(2\eta )^{\alpha +1}\le 2\eta,\notag
		\end{align}
		provided $\eta$ is chosen sufficiently small.
		
		Similar estimates give
		\[
		\|\Phi(u)\|_{L_t^{\frac{5\alpha }{2}} L_x^{\frac{5\alpha }{2}}}
		\leq C\| (-\Delta _\Omega)^{\frac{s_c}{2}} \Phi(u) \|_{L_t^{\frac{5\alpha }{2}} L_x^{\frac{30\alpha }{15\alpha -8}}}\le 2C\eta,
		\]
		and 
		\begin{align}
			\|\Phi(u)\|_{L^\infty _t\dot H^{s_c}_D(\Omega)}&\le  \|u_0\|_{\dot H^{s_c}_D(\Omega)}+C \|(-\Delta _\Omega)^{\frac{s_c}{2}}(|u|^{\alpha }u)\|_{L_t^{\frac{5\alpha }{2(\alpha +1)}}L_x^{\frac{30\alpha }{27\alpha -8}}}\notag\\
			&\le    \|u_0\|_{\dot H^{s_c}_D(\Omega)}+C \|u\|^{\alpha }_{L_t^\frac{5\alpha }{2}L_x^{\frac{5\alpha }{2}}} \|(-\Delta _\Omega)^{\frac{s_c}{2}}u\|  _{L_t^\frac{5\alpha }{2}L_x^{\frac{30\alpha }{15\alpha -8}}}\notag\\
			&\le  \|u_0\|_{\dot H^{s_c}_D(\Omega)} +C(2\eta)^{\alpha +1}\le 2 \|u_0\|_{\dot H^{s_c}_D(\Omega)}, \notag
		\end{align}
		provided   $\eta$ is chosen small enough. 
		This shows that   $\Phi$ maps the ball   $B$ to itself.
		
		Finally, to prove that  $\Phi$ is a contraction, we argue as above:
		\begin{align}
			d(\Phi(u),\Phi(v)) &\leq C  \||u|^{\alpha }u-|v|^{\alpha }v\| _{L_t^{ \frac{5\alpha }{2(\alpha +1)}} L_x^{\frac{30\alpha }{27\alpha -8}}}\notag\\
			&\le Cd(u,v) \left( \| (-\Delta _\Omega)^{\frac{s_c}{2}}u \|_{L_t^\frac{5\alpha }{2}L_x^{\frac{30\alpha }{15\alpha -8}}}^{\alpha }+ \|(-\Delta _\Omega)^{\frac{s_c}{2}}v \|_{L_t^\frac{5\alpha }{2}L_x^{\frac{30\alpha }{15\alpha -8}}}^{\alpha } \right)\notag\\
			&\le 2Cd(u,v)(2\eta )^{\alpha }\le \frac{1}{2}d(u,v),\notag
		\end{align}
		provided   $\eta$ is chosen small enough. 
	\end{proof}
	Below, we present the stability theorem for the Schr\"odinger equation in the exterior domain. Its proof relies on the following nonlinear estimate.
	\begin{lemma}\label{Lnonlinearestimate}
		For any $u, v \in L_t^{\frac{5\alpha }{2}}\dot H^{s_c,\frac{30\alpha }{15\alpha -8}}_D(I\times \Omega)$, the following inequality holds: 
		\begin{align}
			& \|(-\Delta _\Omega)^{\frac{s_c}{2}}\left(|u+v|^{\alpha }(u+v)-|u|^{\alpha }u\right)\| _{L_t^{\frac{5\alpha }{2(\alpha +1)}}L_x^{\frac{30\alpha }{27\alpha -8}}} \notag\\
			&\lesssim \left( \|u\|_{L_t^{\frac{5\alpha }{2}}L_x^{\frac{5\alpha }{2}}} ^{\alpha -1}+ \|v\|_{L_t^{\frac{5\alpha }{2}}L_x^{\frac{5\alpha }{2}}} ^{\alpha -1}  \right)( \| (-\Delta _\Omega)^{\frac{s_c}{2}}u\|_{L_t^{\frac{5\alpha }{2}}L_x^{\frac{30\alpha }{15\alpha -8}}} + \| (-\Delta _\Omega)^{\frac{s_c}{2}}v\|_{L_t^{\frac{5\alpha }{2}}L_x^{\frac{30\alpha }{15\alpha -8}}}   )^2,\label{E1162}      
		\end{align}
	\end{lemma}
	where all the space-time integrals   are  over  $I\times \Omega$. Note that since  $s_c > 0$, we have $\alpha > \frac{4}{3}$. 
	\begin{proof}
		We first consider the case  $s_c<1$. Applying Lemma \ref{LDerivatives of differences} and the equivalence   theorem \ref{TEquivalence}, we obtain 
		\begin{align}
			& \|(-\Delta _\Omega)^{\frac{s_c}{2}}\left(|u+v|^{\alpha }(u+v)-|u|^{\alpha }u\right)\| _{L_t^{\frac{5\alpha }{2(\alpha +1)}}L_x^{\frac{30\alpha }{27\alpha -8}}} \notag\\
			&\lesssim \|v\|^\alpha _{L_t^{\frac{5\alpha }{2}} L_x^{\frac{5\alpha }{2}}} \|(-\Delta _\Omega)^{\frac{s_c}{2}}u\| _{L_t^{\frac{5\alpha }{2}} L_x^{\frac{30\alpha }{15\alpha -8}} } + \|u+v\|_{L_t^{\frac{5\alpha }{2}} L_x^{\frac{5\alpha }{2}} }^\alpha \|(-\Delta _\Omega)^{\frac{s_c}{2}}v\|_{L_t^{\frac{5\alpha }{2}} L_x^{\frac{30\alpha }{15\alpha -8}} }.\notag
		\end{align}
		Further using  Sobolev embedding yields (\ref{E1162}). 
		
		Next, we turn to the case  $s_c>1$. 
		Writing $F(u) = |u|^{\alpha} u$, we have
		\begin{equation}
			|\nabla|^{s_c} \left(|u+v|^{\alpha }(u+v)-|u|^{\alpha }u\right) = |\nabla |^{s_c-1}[F'(u+v)-F'(u)]\nabla u + |\nabla |^{s_c-1}[F'(u+v)\nabla v].\notag
		\end{equation}
		Using the fractional differentiation rule and Sobolev embedding, we obtain
		\begin{align}
			& \||\nabla |^{s_c-1}[F'(u+v)\nabla v]\| _{L_t^{\frac{5\alpha }{2(\alpha +1)}}L_x^{\frac{30\alpha }{27\alpha -8}}} \notag\\
			&\lesssim \||\nabla |^{s_c-1} F'(u+v)\|_{L_t^\frac{5}{2}L_x^{\frac{5\alpha }{2(\alpha -1)}}} \|\nabla v\|_{L_t^{\frac{5\alpha }{2}}L_x^{\frac{15\alpha }{5\alpha +6}}}  + \|u+v\|^\alpha _{L_t^{\frac{5\alpha }{2}}L_x^{\frac{5\alpha }{2}}} \||\nabla |^{s_c}v\|_{L_t^{\frac{5\alpha }{2}}L_x^{\frac{30\alpha }{15\alpha -8}}} \notag\\
			&\lesssim  \|u+v\|_{L_t^{\frac{5\alpha }{2}}L_x^{\frac{5\alpha }{2}}} ^{\alpha -1} \||\nabla |^{s_c}(u+v)\|_{L_t^{\frac{5\alpha }{2}}L_x^{\frac{30\alpha }{15\alpha -8}}} \||\nabla |^{s_c}v\|_{L_t^{\frac{5\alpha }{2}}L_x^{\frac{30\alpha }{15\alpha -8}}}.\label{E1163} 
		\end{align}
		Similarly, using the fractional differentiation rule, Sobolev embedding, and Lemma \ref{LDerivatives of differences}, we have
		\begin{align}
			&\||\nabla |^{s_c-1}[\left(F'(u+v)-F'(u)\right)\nabla u]\|_{L_t^{\frac{5\alpha }{2(\alpha +1)}}L_x^{\frac{30\alpha }{27\alpha -8}}}\notag\\
			&\lesssim \||\nabla |^{s_c-1}\left(F'(u+v)-F'(u)\right) \|_{L_t^{\frac{5\alpha }{2}}L_x^{\frac{30\alpha }{17\alpha -20}}}  \|\nabla u\|_{L_t^{\frac{5\alpha }{2}  }L_x^{\frac{15\alpha }{5\alpha +6}}}\notag\\
			&\qquad + \|F'(u+v)-F'(u)\|_{L_t^{\frac{5}{2}}L_x^{\frac{5}{2}}} \|\nabla |^{s_c}u|\|_{L_t^{\frac{5\alpha }{2}}L_x^{\frac{30\alpha }{15\alpha -8}}}   \notag\\
			&\lesssim  \left(\||\nabla |^{s_c-1}u\|_{L_t^{\frac{5\alpha }{2}}L_x^{\frac{30\alpha }{5\alpha -8}}} \|v\|_{L_t^{\frac{5\alpha }{2}}L_x^{\frac{5\alpha }{2}}} ^{\alpha -1}+ \||\nabla |^{s_c-1}v\|_{L_t^{\frac{5\alpha }{2}}L_x^{\frac{30\alpha }{5\alpha -8}}} \|u+v\|^{\alpha -1}_{L_t^{\frac{5\alpha }{2}}L_x^{\frac{5\alpha }{2}}}    \right) \||\nabla |^{s_c}u\|_{L_t^{\frac{5\alpha }{2}}L_x^{\frac{30\alpha }{15\alpha -8}}}\notag\\
			&\qquad + \left(\|u+v\|_{L_t^{\frac{5\alpha }{2}}L_x^{\frac{5\alpha }{2}}} ^{\alpha -1} + \|v\|_{L_t^{\frac{5\alpha }{2}}L_x^{\frac{5\alpha }{2}}} ^{\alpha -1} \right)  \|v\|_{L_t^{\frac{5\alpha }{2}}L_x^{\frac{5\alpha }{2}}}  \||\nabla ^{s_c}u|\|_{L_t^{\frac{5\alpha }{2}}L_x^{\frac{30\alpha }{15\alpha -8}}}\notag\\
			&\lesssim     \left( \|u\|_{L_t^{\frac{5\alpha }{2}}L_x^{\frac{5\alpha }{2}}} ^{\alpha -1}+ \|v\|_{L_t^{\frac{5\alpha }{2}}L_x^{\frac{5\alpha }{2}}} ^{\alpha -1}  \right)( \||\nabla |^{s_c}u\|_{L_t^{\frac{5\alpha }{2}}L_x^{\frac{30\alpha }{15\alpha -8}}} + \||\nabla |^{s_c}v\|_{L_t^{\frac{5\alpha }{2}}L_x^{\frac{30\alpha }{15\alpha -8}}}   )^2.  \label{E1164}
		\end{align}
		Combining (\ref{E1163}) and (\ref{E1164}), and using the equivalence  theorem \ref{TEquivalence}, we obtain (\ref{E1162}).
	\end{proof}
	Now, we are in position to give the stability result  for the Schr\"odinger equation (\ref{NLS}).  
	\begin{theorem}[Stability result]\label{TStability}
		Let $\Omega$ be the exterior of a smooth compact strictly convex obstacle in $\mathbb{R}^3$. Let $I$ a compact time interval and let $\tilde{u}$ be an approximate solution to (\ref{NLS}) on $I \times \Omega$ in the sense that 
		\begin{equation}
			i\tilde{u}_t = -\Delta_\Omega \tilde{u}  + |\tilde{u}|^{\alpha } \tilde{u} + e\label{E118w3}
		\end{equation}
		for some function $e$. Assume that
		\[
		\|\tilde{u}\|_{L_t^\infty \dot{H}_D^{s_c}(I \times \Omega)} \leq E \quad \text{and} \quad \|\tilde{u}\|_{L_t^{\frac{5\alpha }{2}} L_x^{\frac{5\alpha }{2}} (I \times \Omega)} \leq L
		\]
		for some positive constants $E$ and $L$.  Assume also the smallness conditions
		\[
		\|(-\Delta _\Omega)^{\frac{s_c}{2}}e^{i(t-t_0)\Delta_\Omega} (u_0 - \tilde{u}(t_0))\|_{L_t^{\frac{5\alpha }{2}} L_x^{\frac{30\alpha }{15\alpha -8}} (I\times \Omega)}   \leq \epsilon,
		\]
		\begin{equation}
			\|e\|_{\dot N^{s_c}((I\times \Omega))}:=\inf \left\{ \|(-\Delta _\Omega)^{\frac{s_c}{2}}e\|_{L_t^{q'}L_x^{r'}(I\times \Omega)}: \ \frac{2}{q}+\frac{3}{r}=\frac{3}{2}  \right\} \le \varepsilon .\label{E1241}
		\end{equation}
		for some $0 < \epsilon < \epsilon_1 = \epsilon_1(E, L)$. Then, there exists a unique strong solution $u : I \times \Omega \to \mathbb{C}$ to (\ref{NLS}) with initial data $u_0$ at time $t=t_0$ satisfying
		\[
		\|(-\Delta _\Omega)^{\frac{s_c}{2}}(u - \tilde{u})\|_{L_t^{\frac{5\alpha }{2}} L_x^{\frac{30\alpha }{15\alpha -8}} (I\times \Omega)}  \leq C(E,  L)  \varepsilon,
		\]
		\[
		\|(-\Delta _\Omega)^{\frac{s_c}{2}}u\|_{L_t^{\frac{5\alpha }{2}} L_x^{\frac{30\alpha }{15\alpha -8}}(I\times \Omega) }  \leq C(E, L).
		\]	
	\end{theorem}
	\begin{proof}
		We provide only a brief outline of the proof; the standard proof can be found in \cite{Colliander2008, RyckmanVisan2007, TaoVisan2005}.
		
		Define $w = u - \widetilde{u}$ so that $(i\partial_{t} + \Delta_\Omega) w= |u|^{\alpha} u - |\widetilde{u}|^{\alpha} \widetilde{u} - e$. It then follows from  Lemma \ref{Lnonlinearestimate}, Strichartz estimate, and (\ref{E1241}) that
		\begin{align}
			\|(-\Delta _\Omega)^{\frac{s_c}{2}}w\|_{L_t^{\frac{5\alpha}{2}} L_x^{\frac{30\alpha}{15\alpha - 8}}(I \times \Omega)} &\lesssim \varepsilon + \left( \|\widetilde{u}\|^{\alpha -1}_{L_t^{\frac{5\alpha}{2}} L_x^{\frac{5\alpha}{2}}(I \times \Omega)} + \|w\|_{L_t^{\frac{5\alpha}{2}} L_x^{\frac{5\alpha}{2}}(I \times \Omega)}^{\alpha - 1} \right) \notag\\
			&\qquad \times \left( \|(-\Delta_\Omega)^{\frac{s_c}{2}} \widetilde{u}\|_{L_t^{\frac{5\alpha}{2}} L_x^{\frac{30\alpha}{15\alpha - 8}}(I \times \Omega)} + \|(-\Delta_\Omega)^{\frac{s_c}{2}} w\|_{L_t^{\frac{5\alpha}{2}} L_x^{\frac{30\alpha}{15\alpha - 8}}(I \times \Omega)} \right)^2. \notag
		\end{align}
		
		We first note that the above inequality implies  that there exists $\delta > 0$ such that,  under the additional assumption
		\begin{equation}
			\|(-\Delta_\Omega)^{\frac{s_c}{2}} \widetilde{u}\|_{L_t^{\frac{5\alpha}{2}} L_x^{\frac{30\alpha}{15\alpha - 8}} (I \times \Omega)} \le \delta, \label{E118w1}
		\end{equation}
		we can use the continuity method to obtain
		\begin{equation}
			\|(-\Delta_\Omega)^{\frac{s_c}{2}} w\|_{L_t^{\frac{5\alpha}{2}} L_x^{\frac{30\alpha}{15\alpha - 8}} (I \times \Omega)} \lesssim \varepsilon. \label{E118w2}
		\end{equation}
		This is the so-called "short-time perturbation" (see \cite[Lemma 3.13]{KillipVisan2013}).
		
		For the general case, we   divide the interval $I$ into a finite number of smaller intervals $I_j$, $1 \le j \le n$, such that on each subinterval $I_j$, the $L_t^{\frac{5\alpha}{2}} L_x^{\frac{5\alpha}{2}}$ norm of $\widetilde{u}$ is sufficiently small. Then using equation (\ref{E118w3}), the Strichartz estimate, and the continuity method on each subinterval $I_j$, we know  that (\ref{E118w1}) holds on each $I_j$, thus obtaining that (\ref{E118w2}) holds on each $I_j$. Summing the estimates over all $I_j$, we obtain the desired estimate in Theorem \ref{TStability}.
	\end{proof}
	
	\subsection{Convergence results}
	The region $\Omega$ is not preserved under scaling or translation. In fact, depending on the choice of such operations, the obstacle may shrink to a point, move off to infinity, or even expand to fill an entire half-space. In this subsection, we summarize some results from \cite{KillipVisanZhang2016a} regarding the behavior of functions associated with the Dirichlet Laplacian under these transformations, as well as the convergence of propagators in Strichartz spaces. These results are crucial for the proof of the linear profile decomposition (Proposition \ref{linear-profile}).
	
	Throughout this subsection, we denote the Green's function of the Dirichlet Laplacian in a general open set $\mathcal{O}$ by  
	\begin{align*}
		G_{\mathcal{O}}(x, y; \lambda) := \left( - \Delta_{\mathcal{O}} - \lambda \right)^{-1}(x, y).
	\end{align*}

	\begin{definition}[\cite{KillipVisanZhang2016a}]\label{def-limit}
		Given a sequence $\{\mathcal{O}_n\}_n$ of open subsets of $\mathbb{R}^3$, we define
		\begin{align*}
			\widetilde{\lim} \,  \mathcal{O}_n :  =  \left\{ x \in \mathbb{R}^3 : \liminf\limits_{n \to \infty } \operatorname{dist}  \left(x, \mathcal{O}_n^c  \right) > 0  \right\}.
		\end{align*}
		Writing $\tilde{O} = \widetilde{\lim} \, \mathcal{O}_n$, we say $\mathcal{O}_n \to \mathcal{O}$ if the following two conditions hold: the symmetric difference $\mathcal{O} \triangle \tilde{O}$ is a finite set and
		\begin{align}\label{eq3.1v65}
			G_{\mathcal{O}_n}(x,y;  \lambda ) \to G_{\mathcal{O}} (x,y ;  \lambda )
		\end{align}
		for all $ \lambda  \in (-2 , - 1)$, all $x \in \mathcal{O}$, and uniformly for $y$ in compact subsets of $\mathcal{O} \setminus \{x \}$.
	\end{definition}
	\begin{remark}
		We restrict $\lambda$ to the interval $(-2, -1)$ in (\ref{eq3.1v65}) for simplicity and because it allows us to invoke the maximum principle when verifying this hypothesis. Indeed, Killip-Visan-Zhang \cite[Lemma 3.4]{KillipVisanZhang2016a} proved that this convergence actually holds for all $\lambda \in \mathbb{C} \setminus [0, \infty)$.
	\end{remark}
	
	Given sequences of scaling and translation parameters $N_n \in 2^{\mathbb{Z}}$ and $x_n \in \Omega$, we would like to consider the domains $\Omega_n:=N_n \left( \Omega -  \left\{x_n \right\} \right)$. When $\Omega_n\rightarrow\Omega_\infty$ in the sense of Definition \ref{def-limit}, Killip, Visan and Zhang\cite{KillipVisanZhang2016a} used the maximum principle to prove the convergence of the corresponding Green's functions. Then, by applying the Helffer-Sj\"ostrand formula and using the convergence of the Green's functions, they obtain the following two convergence results: 
	\begin{proposition}\label{convergence-domain}
		Assume $\Omega_n \to \Omega_\infty$   in the sense of Definition \ref{def-limit} and let $\Theta \in C_0^\infty ((0, \infty))$.
		Then, 
		\begin{align}\label{eq3.11v65}
			\left\|  \left( \Theta  \left( - \Delta_{\Omega_n}  \right) - \Theta  \left( - \Delta_{\Omega_\infty}  \right) \right) \delta_y  \right\|_{\dot{H}^{-s_c} ( \mathbb{R}^3 )} \to 0 \qtq{ when} n\to \infty,
		\end{align}
		uniformly for $y$ in compact subsets of $\widetilde{\lim}\, \Omega_n$. Moreover, for any fixed $t\in\R$ and $h\in C_0^\infty(\widetilde{\lim}\,\Omega_n)$, we have
		\begin{align*}
			\lim_{n\to\infty}\big\|e^{it\Delta_{\Omega_n}}h-e^{it\Delta_{\Omega_{\infty}}}h\big\|_{\dot{H}^{-s_c}(\R^3)}=0.
		\end{align*}
	\end{proposition}
	\begin{proposition}\label{P1} Let $\Omega_n\to\Omega_{\infty}$ in the sense of Definition \ref{def-limit}. Then we have
		\begin{align*}
			\big\|(-\Delta_{\Omega_n})^\frac{s_c}{2}f-(-\Delta_{\Omega_\infty})^\frac{s_c}2f\big\|_{L^2(\R^3)}\to0
		\end{align*}
		for all $f\in C_0^\infty(\widetilde{\lim}\,\Omega_n)$.
	\end{proposition}
	\begin{remark}
		Killip, Visan and Zhang \cite{KillipVisanZhang2016a} proved Proposition \ref{convergence-domain} and Proposition \ref{P1} for the case when $s_c=1$. Using their results and interpolation, we can easily extend this to the general case where $s_c\in  (0,\frac{3}{2})$.
	\end{remark}
	
	Next, we state the convergence of the Schr\"odinger propagators within the Strichartz norms. We rescale and translate the domain $\Omega$ to $\Omega_n=N_n*(\{\Omega\}-x_n)$ which depends on the parameters $N_n\in2^\Bbb{Z}$ and $x_n\in\Omega$ conforming to one of the following three scenarios (recall that $d(x_n):=\operatorname{dist}(x_n,\Omega^c)$):
	\begin{align*}
		\begin{cases}
			\text{(i) }N_n\to0\qtq{and}-N_nx_n\to x_\infty\in\R^3,\\
			\text{(ii) }N_nd(x_n)\to\infty,\\
			\text{(iii) }		N_n\to\infty\qtq{and} N_nd(x_n)\to d_\infty>0.
		\end{cases}
	\end{align*} 
	Indeed, in the linear profile decomposition, there are four cases needed to be discussed (see Theorem \ref{linear-profile} below). The first case will not be included in these three scenarios since there is no change of geometry in that case. In Case (i) and (ii), $\Omega_n\to\R^3$ while in Case (iii), $\Omega_n\to\mathbb{H}$. 
	
	After these preparation, we can state the convergence of linear Schr\"odinger propagators. See Theorem 4.1 and Corollary 4.2 in Killip-Visan-Zhang \cite{KillipVisanZhang2016a}.  
	\begin{theorem}\label{convergence-flow}
		Let $\Omega_n$ be as above and let $\Omega_\infty$ be  such that  $\Omega_n\rightarrow\Omega_\infty $.  Then, for all $\phi\in C_0^\infty(\widetilde{\lim}\,\Omega_n)$,
		\begin{align*}
			\lim_{n\to\infty}\big\|e^{it\Delta_{\Omega_n}}\phi-e^{it\Delta_{\Omega_{\infty}}}\phi\big\|_{L_{t,x}^{\frac{5\alpha }{2}}(\R\times\R^3)}=0.
		\end{align*}
	\end{theorem}
	\section{Linear profile decomposition}\label{S3}
	In this section, we prove  a linear profile decomposition for the Schr\"odinger propagator $e^{it\Delta_\Omega}$ for initial data $u_0\in\dot{H}_D^{s_c}(\Omega)$ with $s_c\in(0,\frac{3}{2})$.   
	The case $s_c = 1$ has been established by Killip-Visan-Zhang \cite{KillipVisanZhang2016a}.
	In this section, we use the linear  profile decomposition for $e^{it\Delta_{\R^d}}$ in  $\dot H^{s_c}(\mathbb{R} ^d)$ as a black-box (see e.g. \cite{Shao2009EJDE}), and  extend the result of Killip-Visan-Zhang \cite{KillipVisanZhang2016a} to the general $\dot H^{s_c}_D(\Omega)$ setting.

	Throughout this section, we denote $\Theta:\R^3\to[0,1]$ the smooth function by
	\begin{align*}
		\Theta(x)=\begin{cases}
			0, & |x|\leqslant\frac{1}{4}, \\
			1, & |x|\geqslant\frac{1}{2}.
		\end{cases}
	\end{align*}

	We start with a refined Strichartz estimates.
	\begin{proposition}[Refined Strichartz estimate]\label{PRefined SZ}Let $s_c\in(0,\frac{3}{2})$ and $f\in\dot{H}_D^{s_c}(\Omega)$. Then we have
		\begin{align}\label{refined-strichartz}
			\big\|e^{it\Delta_\Omega}f\big\|_{L_{t,x}^{q_0}(\R\times\Omega)}\lesssim\|f\|_{\dot{H}_D^{s_c}}^{\frac{2}{q_0}}\sup_{N\in2^\Bbb{Z}}\|e^{it\Delta_\Omega}P_N^\Omega f \|_{L_{t,x}^{q_0}(\R\times\Omega)}^{1-\frac{2}{q_0}},
		\end{align}
		where $q_0:=\frac{10}{3-2s_c}=\frac{5\alpha }{2}$. 
	\end{proposition}
	\begin{proof}
		Throughout the proof, all space-time norms are taken over $\R\times\Omega$ and 
		we set $u(t) = e^{it\Delta_\Omega}f$.
		
		We divide the proof of Proposition \ref{PRefined SZ} into two cases. 
		
		\textbf{Case One}.  First suppose $s_c>\frac{1}{4}$, so that $q_0=\frac{10}{3-2s_c}>4$.   By the square function estimate (Lemma~\ref{LSquare function estimate}), Bernstein inequality and Strichartz estimates, we have
		\begin{align*}
			\|u\|_{L_{t,x}^{q_0}}^{q_0} &\lesssim \iint\biggl[\sum_N |u_N|^2\biggr]^{\frac{q_0}{2}}\,dx\,dt   \lesssim \sum_{N_1\leq N_2} \iint\biggl[\sum_N |u_N|^2\biggr]^{\frac{q_0}{2}-2} |u_{N_1}|^2|u_{N_2}|^2\,dx\,dt \\
			& \lesssim \|u\|_{L_{t,x}^{q_0}}^{q_0-4}\sum_{N_1\leq N_2} \|u_{N_1}\|_{L_t^{q_0}L_x^{q_0+}}\|u_{N_2}\|_{L_t^{q_0}L_x^{q_0-}}\prod_{j=1}^2 \|u_{N_j}\|_{L_{t,x}^{q_0}} \\
			& \lesssim \|f\|_{\dot H_D^{s_c}}^{q_0-4} \sup_N \|u_N\|_{L_{t,x}^{q_0}}^2 \sum_{N_1\leq N_2} \bigl(\tfrac{N_1}{N_2}\bigr)^{0+}\prod_{j=1}^2 \|u_{N_j}\|_{L_t^{q_0}\dot H_x^{s_c,r_0}} \\
			& \lesssim \|f\|_{\dot H_D^{s_c}}^{q_0-4}\sup_N\|u_N\|_{L_{t,x}^{q_0}}^2 \sum_{N_1\leq N_2}\bigl(\tfrac{N_1}{N_2}\bigr)^{0+}\|f_{N_1}\|_{\dot H_x^{s_c}}\|f_{N_2}\|_{\dot H_x^{s_c}} \\
			& \lesssim \|f\|_{\dot H_D^{s_c}}^{q_0-2}\sup_N\|e^{it\Delta_\Omega}f_N\|_{L_{t,x}^{q_0}}^2,
		\end{align*}
		where  $r_0=\frac{9+4s_c}{10}$ such that  $(q_0,r_0)$ is admissible pair. Therefore, we complete the proof of the first case.

		\textbf{Case Two}.  Suppose $\frac{1}{4}\leqslant s_c<\frac{3}{2}$, so that $2<q_0\leq4$. Arguing similar to the first case, we observe that
		\begin{align*}
			\|u\|_{L_{t,x}^{q_0}}^{q_0} &\lesssim \iint \biggl[\sum_N |u_N|^2\biggr]^{\frac{q_0}{2}}\,dx\,dt   \lesssim \iint \biggl[\sum_N |u_N|^{\frac{q_0}{2}}\biggr]^2\,dx\,dt \\
			& \lesssim\sum_{N_1\leq N_2} \iint |u_{N_1}|^{\frac{q_0}{2}}|u_{N_2}|^{\frac{q_0}{2}} \,dx \,dt \\
			& \lesssim\sum_{N_1\leq N_2} \|u_{N_1}\|_{L_t^{q_0}L_x^{q_0+}}\|u_{N_2}\|_{L_t^{q_0}L_x^{q_0-}} \prod_{j=1}^2 \|u_{N_j}\|_{L_{t,x}^{q_0}}^{\frac{q_0}{2}-1} \\
			& \lesssim \sup_N \|e^{it\Delta_\Omega}f_N\|_{L_{t,x}^{q_0}}^{q_0-2}\|f\|_{\dot H_D^{s_c}}^2,
		\end{align*}
		giving the desired result in this case.
	\end{proof}
	
	The refined Strichartz estimates above indicate that a linear solution with nontrivial spacetime norms must concentrate in an annular region. The following inverse Strichartz inequality further demonstrates that the linear solution contains at least one bubble near a specific spacetime point.
	
	\begin{proposition}[Inverse Strichartz estimate]\label{inverse-strichartz}
		Let $\{f_n\} \in \dot{H}_D^{s_c}(\Omega)$. Assume that  
		\begin{align}\label{inverse-con}
			\lim_{n\to\infty}\|f_n\|_{\dot{H}_D^{s_c}(\Omega)}=A<\infty,\quad\text{and}\quad \lim_{n\to\infty}\big\|e^{it\Delta_{\Omega}}f_n\big\|_{L_{t,x}^{q_0}(\R\times\Omega)}=\varepsilon>0.
		\end{align}
		Then, there exists a subsequence $\{f_n\}$, along with $\{\phi_n\} \in \dot{H}_D^{s_c}(\Omega)$, $\{N_n\} \subset 2^{\mathbb{Z}}$, and $\{(t_n, x_n)\} \subset \mathbb{R} \times \Omega$, satisfying one of the four scenarios below, such that:  
		\begin{gather}
			\liminf_{n\to\infty}\|\phi_n\|_{\dot{H}_D^{s_c}(\Omega)} \gtrsim \varepsilon^\frac{15}{s_c(4s_c+4)}A^{\frac{4s_c^2+4s_c-15}{2s_c(2s_c+2)}} ,\label{inverse-1}\\
			\liminf_{n\to\infty}\big\{\|f_n\|_{\dot{H}_D^{s_c}(\Omega)}^2-\|f_n-\phi_n\|_{\dot{H}_D^{s_c}(\Omega)}^2-\|\phi_n\|_{\dot{H}_D^{s_c}(\Omega)}^2\big\} \gtrsim \varepsilon^\frac{15}{s_c(2s_c+2)}A^{\frac{4s_c^2+4s_c-15}{s_c(2s_c+2)}} ,\label{inverse-2}\\
			\liminf_{n\to\infty}\left\{\big\|e^{it\Delta_\Omega}f_n\big\|_{L_{t,x}^{q_0}(\R\times\Omega)}^{q_0}-\big\|e^{it\Delta_{\Omega}}(f_n-\phi_n)\big\|_{L_{t,x}^{q_0}(\R\times\Omega)}^{q_0}\right\} \gtrsim \varepsilon^\frac{75}{2s_c(s_c+1)}A^{\frac{20s_c^2+20s_c-75}{2s_c(s_c+1)}} .\label{inverse-3}
		\end{gather}
		
		The four cases are as follows:
		\begin{itemize}
			\item \textbf{Case 1:} $N_n \equiv N_\infty \in 2^{\mathbb{Z}}$ and $x_n \to x_\infty \in \Omega$. Here, we select $\phi \in \dot{H}_D^{s_c}(\Omega)$ and a subsequence such that $e^{it_n\Delta_\Omega}f_n \rightharpoonup \phi$ weakly in $\dot{H}_D^{s_c}(\Omega)$, and define $\phi_n = e^{-it_n\Delta_\Omega}\phi$.
		\end{itemize}
		\begin{itemize}
			\item \textbf{Case 2:} $N_n \to 0$ and $-N_nx_n \to x_\infty \in \mathbb{R}^3$. In this case, we find $\tilde{\phi} \in \dot{H}^{s_c}(\mathbb{R}^3)$ and a subsequence such that  
			\[
			g_n(x) = N_n^{s_c-\frac{3}{2}}(e^{it_n\Delta_\Omega}f_n)(N_n^{-1}x+x_n) \rightharpoonup \tilde{\phi}(x)
			\]
			weakly in $\dot{H}^{s_c}(\mathbb{R}^3)$. We define  
			\[
			\phi_n(x) := N_n^{\frac{3}{2}-s_c}e^{-it_n\Delta_\Omega}[(\chi_n\tilde{\phi})(N_n(x-x_n))],
			\]
			where $\chi_n(x) = \chi(N_n^{-1}x+x_n)$ and $\chi(x) = \Theta\big(\frac{d(x)}{\operatorname{diam}(\Omega^c)}\big)$.
		\end{itemize}
		\begin{itemize}
			\item \textbf{Case 3:} $N_nd(x_n) \to \infty$. In this situation, we take $\tilde{\phi} \in \dot{H}^{s_c}(\mathbb{R}^3)$ and a subsequence such that  
			\[
			g_n(x) = N_n^{s_c-\frac{3}{2}}(e^{it_n\Delta_\Omega}f_n)(N_n^{-1}x+x_n) \rightharpoonup \tilde{\phi}(x)
			\]
			weakly in $\dot{H}^{s_c}(\mathbb{R}^3)$. We then define  
			\[
			\phi_n(x) := N_n^{\frac{3}{2}-s_c}e^{-it_n\Delta_\Omega}[(\chi_n\tilde{\phi})(N_n(x-x_n))],
			\]
			where $\chi_n(x) = 1-\Theta\big(\frac{|x|}{N_nd(x_n)}\big)$.
		\end{itemize}
		\begin{itemize}
			\item \textbf{Case 4:} $N_n \to \infty$ and $N_nd(x_n) \to d_\infty > 0$. Here, we find $\tilde{\phi} \in \dot{H}_D^{s_c}(\mathbb{H})$ and a subsequence such that  
			\[
			g_n(x) = N_n^{s_c-\frac{3}{2}}(e^{it_n\Delta_\Omega}f_n)(N_n^{-1}R_nx+x_n^*) \rightharpoonup \tilde{\phi}(x)
			\]
			weakly in $\dot{H}^{s_c}(\mathbb{R}^3)$. We define  
			\[
			\phi_n(x) = N_n^{\frac{3}{2}-s_c}e^{-it_n\Delta_\Omega}[\tilde{\phi}(N_nR_n^{-1}(\cdot-x_n^*))],
			\]
			where $R_n \in SO(3)$ satisfies $R_ne_3 = \frac{x_n-x_n^*}{|x_n-x_n^*|}$ and $x_n^* \in \partial\Omega$ such that $d(x_n) = |x_n-x_n^*|$.
		\end{itemize}
	\end{proposition}
	
	\begin{proof}
		Using the refined Strichartz estimate \eqref{refined-strichartz} and \eqref{inverse-con}, we see that for each $n$, there exists $N_n$ such that
		\begin{align*}
			\big\|e^{it\Delta_\Omega}P_{N_n}^\Omega f_n\big\|_{L_{t,x}^{q_0}(\R\times\Omega)}&\gtrsim\big\|e^{it\Delta_\Omega}f_n\big\|_{L_{t,x}^{q_0}(\R\times\Omega)}^{\frac{q_0}{q_0-2}}\|f_n\|_{\dot{H}_D^{s_c}(\Omega)}^{-\frac{2}{q_0-2}} \gtrsim\varepsilon^{\frac{q_0}{q_0-2}}A^{-\frac{2}{q_0-2}}.
		\end{align*} 
		By Strichartz, Bernstein and (\ref{inverse-strichartz}),  we obtain
		\begin{align*}
			\big\|e^{it\Delta_\Omega}P_{N_n}^\Omega f_n\big\|_{L_{t,x}^ {q_0}(\R\times\Omega)}\lesssim N_n^{-s_c}A.
		\end{align*} 
		Combining the above two estimates and using H\"older's inequality, we obtain
		\begin{align*}
			\varepsilon^{\frac{q_0}{q_0-2}}A^{-\frac{2}{q_0-2}}\lesssim\big\|e^{it\Delta_\Omega}P_{N_n}^\Omega f_n\big\|_{L_{t.x}^{q_0}(\R\times\Omega)}
			&\lesssim\big\|e^{it\Delta_\Omega}P_{N_n}^\Omega f_n\big\|_{L_{t,x}^\frac{10}{3}(\R\times\Omega)}^{1-\frac{2s_c}{3}}\big\|e^{it\Delta_\Omega}P_{N_n}^\Omega f_n\big\|_{L_{t,x}^\infty(\R\times\Omega)}^{\frac{2s_c}{3}}\\
			&\lesssim N_n^{-s_c(1-\frac{2}{3}s_c)}A^{1-\frac{2s_c}{3}}\big\|e^{it\Delta_\Omega}P_{N_n}^\Omega f_n\big\|_{L_{t,x}^\infty(\R\times\Omega)}^{\frac{2s_c}{3}},
		\end{align*}
		which implies 
		\begin{align}
			\big\|e^{it\Delta_{\Omega}}P_{N_n}^\Omega f_n\big\|_{L_{t,x}^\infty(\R\times\Omega)}\gtrsim N_n^{\frac{3}{2}-s_c}\varepsilon^\frac{15}{s_c(4s_c+4)}A^{\frac{4s_c^2+4s_c-15}{2s_c(2s_c+2)}}.\notag
		\end{align}
		Thus there exist $x_n\in\R$ and $t_n\in\R$ such that
		\begin{align}\label{A}
			\big|(e^{it_n\Delta_\Omega}P_{N_n}^\Omega f_n)(x_n)\big|\gtrsim N_n^{\frac{3}{2}-s_c}\varepsilon^\frac{15}{s_c(4s_c+4)}A^{\frac{4s_c^2+4s_c-15}{2s_c(2s_c+2)}}.
		\end{align}
		Note that the four cases in Proposition \ref{inverse-strichartz} are completely determined by the behavior of $x_n$ and $N_n$. We first claim that
		\begin{align}\label{claim}
			N_nd(x_n)\gtrsim \varepsilon^\frac{15}{s_c(4s_c+4)}A^{-\frac{15}{2s_c(2s_c+2)}}.
		\end{align}
		Indeed, using the heat kernel bound (Lemma \ref{Lheatkernel}),  we have
		\begin{align*}
			\int_{\Omega}|e^{t\Delta_\Omega/N_n^2}(x_n,y)|^2dy&\lesssim N_n^6\int_{\Omega}\big|(N_nd(x_n))(N_n(d(x_n)+N_n|x_n-y|))e^{-cN_n^2|x_n-y|^2}\big|^2dy\\
			&\lesssim(N_nd(x_n))^2(N_n(d(x_n)+1))^2N_n^3.
		\end{align*} 
		Writting
		\begin{align*}
			(e^{it_n\Delta_\Omega}P_{N_n}^\Omega f_n)(x_n)=\int_{\Omega}[e^{\Delta_\Omega/N_n^2}(x_n,y)P^{\Omega}_{\leq 2N_n}e^{-\Delta_{\Omega}/N_n^2}e^{it_n\Delta_\Omega}P_{N_n}^\Omega f_n](y)dy,
		\end{align*}
		using \eqref{A},  and Cauchy-Schwartz gives 
		\begin{align*}
			N_n^{\frac{3}{2}-s_c}\varepsilon^\frac{15}{s_c(4s_c+4)}A^{\frac{4s_c^2+4s_c-15}{2s_c(2s_c+2)}}&\lesssim(N_nd(x_n))(N_nd(x_n)+1)N_n^\frac{3}{2}\|P_{\leq 2N_n}^\Omega e^{-\Delta_\Omega/N_n^2}e^{it_n\Delta_\Omega}P_{N_n}^\Omega f_n\|_{L^2(\Omega)}\\
			&\lesssim (N_nd(x_n))(N_nd(x_n)+1)N_n^{\frac{3}{2}-s_c}A.
		\end{align*}
		Then claim \eqref{claim} follows.
		
		Due to \eqref{claim} and  passing the subsequence, we only need to consider the following four cases:
		\begin{enumerate}
			\item[Case 1.] $N_n\sim 1$ and $N_nd(x_n)\sim1$,
			\item[Case 2.] $N_n\to0$ and $N_nd(x_n)\lesssim1$,
			\item[Case 3.] $N_nd(x_n)\to\infty$ as $n\to\infty$,
			\item[Case 4.] $N_n\to\infty$ and $N_nd(x_n)\sim1$.
		\end{enumerate}
		We will treat these cases in order.
		
		\textbf{Case 1}. After passing through the subsequence, we may assume that 
		\begin{align*}
			N_n\equiv N_\infty\in2^{\Bbb{Z}}\mbox{ and }x_n\to x_\infty\in\Omega.
		\end{align*}
		Let 
		\begin{align*}
			g_n (x ): = N_n^{s_c-\frac{3}{2}} (e^{it_n\Delta _\Omega}f_n)  \left(N_n^{-1} x + x_n \right).
		\end{align*}
		Since $f_n$ is supported in $\Omega$, $g_n$ is supported in $\Omega_n : = N_n ( \Omega - \{x_n\})$. Moreover, we have
		\begin{align*}
			\|g_n \|_{\dot{H}_D^{s_c}( \Omega_n)} = \|f_n \|_{\dot{H}_D^{s_c}(  \Omega)} \lesssim A.
		\end{align*}
		Passing to a further subsequence, we  find a $\tilde{\phi}$ such that $g_n \rightharpoonup \tilde{\phi}$ in $\dot{H}^{s_c}( \R^3 )$ as $n \to \infty$.
		Rescaling this weak convergence, we have 
		\begin{align}\label{B}
			e^{it_n\Delta _\Omega}f_n(x) \rightharpoonup \phi(x) : = N_\infty^{\frac{3}{2}-s_c} \tilde{\phi} (N_\infty (x-x_\infty)) \text{ in } \dot{H}_D^{s_c}(\Omega).
		\end{align}
		Since $\dot{H}_D^{s_c}( \Omega)$ is a weakly closed subset of $\dot{H}^{s_c}(\R^3)$, $\phi \in \dot{H}_D^{s_c}(\Omega)$.
		
		We now proceed to prove that $\phi$ is non-trivial. To this end, let $h := P_{N_\infty}^\Omega \delta_{x_\infty}$. By the Bernstein inequality, we have  
		\begin{align}\label{eq5.7v65}
			\left\| \left(- \Delta_\Omega \right)^{-\frac{s_c}{2}} h \right\|_{L^2(\Omega)} = \left\| \left(- \Delta_\Omega \right)^{-\frac{s_c}{2}} P_{N_\infty}^\Omega \delta_{x_\infty} \right\|_{L^2(\Omega)} 
			\lesssim N_\infty^{\frac{3}{2}-s_c},
		\end{align}  
		which shows that $h \in \dot{H}_D^{-s_c} (\Omega)$. On the other hand, we observe that  
		\begin{align}\label{eq5.8v65}
			\langle \phi, h \rangle &= \lim\limits_{n \to \infty} \langle e^{it_n\Delta_\Omega}f_n, h \rangle 
			= \lim\limits_{n \to \infty} \left\langle e^{it_n\Delta_\Omega}f_n, P_{N_\infty}^\Omega \delta_{x_\infty} \right\rangle \nonumber \\
			&= \lim\limits_{n \to \infty} \left(e^{it_n\Delta_\Omega}P_{N_n}^\Omega f_n \right)(x_n) + \lim\limits_{n \to \infty} \left\langle e^{it_n\Delta_\Omega}f_n, P_{N_\infty}^\Omega \left( \delta_{x_\infty} - \delta_{x_n} \right) \right\rangle.
		\end{align}  
		We first claim that the second term in \eqref{eq5.8v65} vanishes. Indeed, since $d(x_n) \sim 1$, the Bernstein inequality implies  
		\begin{align*}
			\left\| P_{N_\infty}^\Omega e^{it_n\Delta_\Omega}f_n \right\|_{L_x^\infty} \lesssim N_\infty^{\frac{3}{2}-s_c} A, \quad
			\text{and} \quad \left\|\Delta P_{N_\infty}^\Omega e^{it_n\Delta_\Omega}f_n \right\|_{L_x^\infty} \lesssim N_\infty^{\frac{3}{2}+s_c} A.
		\end{align*}	
		Using the fundamental theorem of calculus and the basic elliptic estimate  
		\begin{align}\label{eq5.9v65}
			\| \nabla v \|_{L^\infty(|x| \leq R)} \lesssim R^{-1} \|v\|_{L^\infty(|x| \leq 2R)} + R \|\Delta v\|_{L^\infty(|x| \leq 2R)},
		\end{align}  
		it follows for sufficiently large $n$ that  
		\begin{align}\label{eq5.10v65}
			\left| \left\langle e^{it_n\Delta_\Omega}f_n, P_{N_\infty}^\Omega \left( \delta_{x_\infty} - \delta_{x_n} \right) \right\rangle \right| 
			&\lesssim |x_\infty - x_n| \left\|\nabla P_{N_\infty}^\Omega e^{it_n\Delta_\Omega} f_n \right\|_{L^\infty(|x| \leq R)} \notag\\
			&\lesssim \Big( \frac{N_\infty^{\frac{3}{2}-s_c}}{d(x_n)} + N_\infty^{\frac{3}{2}+s_c} d(x_n) \Big) A |x_\infty - x_n|,
		\end{align}  
		which converges to zero as $n \to \infty$.

		Therefore, it follows from \eqref{A}, \eqref{eq5.7v65}, \eqref{eq5.8v65}, and \eqref{eq5.10v65} that
		\begin{align}\label{eq5.11v65}
			N_\infty^{\frac{3}{2}-s_c} \varepsilon^\frac{15}{s_c(4s_c+4)}A^{\frac{4s_c^2+4s_c-15}{2s_c(2s_c+2)}}  \lesssim |\langle \phi, h \rangle |
			\lesssim \|\phi \|_{\dot{H}_D^{s_c}( \Omega)} \|h \|_{\dot{H}_D^{-s_c} ( \Omega)}
			\lesssim N_\infty^{\frac{3}2-s_c} \|\phi \|_{\dot{H}_D^{s_c}( \Omega)},
		\end{align}
		which gives \eqref{inverse-1}.
		
		Next, since $\dot{H}_D^{s_c}(\Omega)$ is a Hilbert space, \eqref{inverse-2} follows directly from \eqref{inverse-1} and \eqref{B}.
		
		It remains to establish the decoupling for the $L_x^{q_0}$ norm in \eqref{inverse-3}. Note that  
		\begin{align*}
			(i\partial_t)^\frac{s_c}{2}e^{it\Delta_\Omega} = (-\Delta_\Omega)^\frac{s_c}{2}e^{it\Delta_\Omega}.
		\end{align*}  
		Applying H\"older's inequality on a compact domain $K \subset \mathbb{R} \times \mathbb{R}^3$, we obtain  
		\begin{align*}
			\big\|e^{it\Delta_\Omega}e^{it_n\Delta_{\Omega}}f_n\big\|_{H_{t,x}^{\frac{s_c}{2}}(K)} 
			\lesssim \|\langle-\Delta_\Omega\rangle^{\frac{s_c}{2}}e^{i(t+t_n)\Delta_\Omega}f_n\|_{L_{t,x}^2(K)} \lesssim_K A.
		\end{align*}  
		By the Rellich-Kondrachov compactness theorem and a diagonal argument, passing to a subsequence yields  
		\begin{align*}
			e^{it\Delta_\Omega}e^{it_n\Delta_\Omega}f_n \to e^{it\Delta_\Omega}\phi \quad \text{strongly in } L^2_{t,x}(K),
		\end{align*}  
		and  
		\begin{align*}
			e^{it\Delta_\Omega}e^{it_n\Delta_\Omega}f_n \to e^{it\Delta_\Omega}\phi(x) \quad \text{a.e. in } \mathbb{R} \times \mathbb{R}^3.
		\end{align*}
		By the refined Fatou lemma (Lemma \ref{LRefinedFatou}) and a change of variables, we have  
		\begin{align*}
			\lim\limits_{n \to \infty} \left( \|e^{it\Delta_\Omega}f_n \|_{L_{t,x}^{q_0}(\mathbb{R} \times \Omega)}^{q_0} - \|e^{it\Delta_\Omega}(f_n - \phi_n) \|_{L_{t,x}^{q_0}(\mathbb{R} \times \Omega)}^{q_0} \right) = \|e^{it\Delta_\Omega}\phi \|_{L_{t,x}^{q_0}(\mathbb{R} \times \Omega)}^{q_0},
		\end{align*}  
		from which \eqref{inverse-3} will follow once we show that  
		\begin{align}\label{eq5.12v65}
			\|e^{it\Delta_\Omega}\phi \|_{L_{t,x}^{q_0}(\mathbb{R} \times \Omega)} \gtrsim \varepsilon^\frac{15}{s_c(4s_c+4)}A^{\frac{4s_c^2+4s_c-15}{2s_c(2s_c+2)}}.
		\end{align}
		
		To prove \eqref{eq5.12v65}, the Mikhlin multiplier theorem provides the uniform estimate for $|t| \leq N_\infty^{-2}$:  
		\begin{align*}
			\left\|e^{it\Delta_\Omega}P_{\leq 2 N_\infty}^\Omega \right\|_{L_x^{q_0^\prime} \to L_x^{q_0^\prime}} \lesssim 1, \quad \text{with} \quad q_0^\prime = \frac{10}{2s_c+7}.
		\end{align*}  
		Combining this with the Bernstein inequality, we get  
		\begin{align*}
			\|e^{it\Delta_\Omega}h \|_{L_x^{q_0^\prime}} \lesssim \left\|e^{it\Delta_\Omega}P_{\leq 2 N_\infty}^\Omega \right\|_{L_x^{q_0^\prime} \to L_x^{q_0^\prime}} \left\|P_{N_\infty}^\Omega \delta_\infty \right\|_{L_x^{q_0^\prime}} \lesssim N_\infty^{\frac{9-6s_c}{10}}.
		\end{align*}  
		This, together with \eqref{eq5.11v65}, implies  
		\begin{align*}
			N_\infty^{\frac{3}{2}-s_c} \varepsilon^\frac{15}{s_c(4s_c+4)}A^{\frac{4s_c^2+4s_c-15}{2s_c(2s_c+2)}} 
			\lesssim |\langle\phi, h\rangle| = |\langle e^{it\Delta_\Omega}\phi, e^{it\Delta_\Omega}h \rangle| 
			\lesssim N_\infty^{\frac{9-6s_c}{10}} \|e^{it\Delta_\Omega}\phi \|_{L_x^{q_0}(\mathbb{R} \times \Omega)},
		\end{align*}  
		uniformly for $|t| \leq N_\infty^{-2}$. Integrating over $t$ then establishes \eqref{eq5.12v65}.

		\textbf{Case 2}. As $N_n \to 0$, the condition $N_n d(x_n) \lesssim 1$ ensures that the sequence $\{N_n x_n\}_{n \geq 1}$ is bounded. Hence, up to a subsequence, we assume $-N_n x_n \to x_\infty \in \mathbb{R}^3$ as $n \to \infty$. 
		Similar to Case 1, we define $\Omega_n := N_n (\Omega - \{x_n\})$. Since $N_n \to 0$, the rescaled obstacles $\Omega_n^c$ shrink to $x_\infty$ as $n \to \infty$.
		Because $f_n$ is bounded in $\dot{H}_D^{s_c}(\Omega)$, the sequence $g_n$ remains bounded in $\dot{H}_D^{s_c}(\Omega_n) \subset \dot{H}^{s_c}(\mathbb{R}^3)$. Passing to a subsequence, we find $\tilde{\phi}$ such that $g_n \rightharpoonup \tilde{\phi}$ in $\dot{H}^{s_c}(\mathbb{R}^3)$ as $n \to \infty$.
		
		Next, we claim that  
		\begin{align}\label{eq5.13v65}
			\chi_n \tilde{\phi} \to \tilde{\phi}, \quad \text{or equivalently,} \quad \left(1 - \chi\left(N_n^{-1}x + x_n\right)\right)\tilde{\phi}(x) \to 0 \text{ in } \dot{H}^{s_c}(\mathbb{R}^3).
		\end{align}
		To show this, let  
		\begin{align*}
			B_n := \left\{x \in \mathbb{R}^3 : \operatorname{dist}(x, \Omega_n^c) \leq \operatorname{diam}(\Omega_n^c) \right\}.
		\end{align*}
		The set $B_n$ contains $\supp(1 - \chi_n)$ and $\supp(\nabla \chi_n)$. Since $N_n \to 0$, the measure of $B_n$ tends to zero as $n \to \infty$. Thus, \eqref{eq5.13v65} follows from H\"older's inequality, Sobolev embedding, the fractional chain rule, and the monotone convergence theorem.
		
		With \eqref{eq5.13v65} established, the proofs of \eqref{inverse-1} and \eqref{inverse-2} proceed analogously to their counterparts in Case 1. First, we prove \eqref{inverse-1}. Define $h := P_1^{\mathbb{R}^3}\delta_0$. Then,  
		\begin{align*}
			\left\langle \tilde{\phi}, h \right\rangle = \lim\limits_{n \to \infty} \langle g_n, h \rangle 
			= \lim\limits_{n \to \infty} \left\langle g_n, P_1^{\Omega_n}\delta_0 \right\rangle + \lim\limits_{n \to \infty} \left\langle g_n, \left(P_1^{\mathbb{R}^3} - P_1^{\Omega_n}\right)\delta_0 \right\rangle.
		\end{align*}
		By Proposition \ref{convergence-domain} and the uniform boundedness of $\|g_n\|_{\dot{H}^{s_c}(\mathbb{R}^3)}$, the second term vanishes. Hence, using the definition of $g_n$ and a change of variables, we find  
		\begin{align}\label{estimate-pair}
			\left|\left\langle \tilde{\phi}, h \right\rangle\right| 
			&= \left|\lim\limits_{n \to \infty} \left\langle g_n, P_1^{\Omega_n}\delta_0 \right\rangle\right| 
			= \left|\lim\limits_{n \to \infty} \left\langle f_n, N_n^{s_c+\frac{3}{2}}\left(P_1^{\Omega_n}\delta_0\right)(N_n(x-x_n)) \right\rangle\right| \notag \\
			&= \left|\lim\limits_{n \to \infty} \left\langle f_n, N_n^{s_c-\frac{3}{2}}P_{N_n}^\Omega\delta_{x_n} \right\rangle\right| 
			\gtrsim \varepsilon^\frac{15}{s_c(4s_c+4)}A^{\frac{4s_c^2+4s_c-15}{2s_c(2s_c+2)}},
		\end{align}
		where the final inequality follows from \eqref{A}. Thus, arguing as in \eqref{eq5.11v65}, we obtain  
		\begin{align*}
			\|\tilde{\phi}\|_{\dot{H}^{s_c}(\mathbb{R}^3)} \gtrsim \varepsilon^\frac{15}{s_c(4s_c+4)}A^{\frac{4s_c^2+4s_c-15}{2s_c(2s_c+2)}},
		\end{align*}
		which, combined with \eqref{eq5.13v65}, establishes \eqref{inverse-1}.

		To establish the decoupling estimate in $\dot{H}_D^{s_c}(\Omega)$, we write  
		\begin{align*}
			&\quad \|f_n\|_{\dot{H}_D^{s_c}(\Omega)}^2 - \|f_n - \phi_n\|_{\dot{H}_D^{s_c}(\Omega)}^2  = 2\langle f_n, \phi_n \rangle_{\dot{H}_D^{s_c}(\Omega)} - \|\phi_n\|_{\dot{H}_D^{s_c}(\Omega)}^2 \\
			&= 2\left\langle N_n^{s_c-\frac{3}{2}} f_n (N_n^{-1} x + x_n), \tilde{\phi}(x) \chi(x) \right\rangle_{\dot{H}_D^{s_c}(\Omega_n)} - \|\chi_n \tilde{\phi}\|_{\dot{H}_D^{s_c}(\Omega_n)}^2 \\
			&= 2\left\langle g_n, \tilde{\phi} \right\rangle_{\dot{H}^{s_c}(\mathbb{R}^3)} - 2\left\langle g_n, \tilde{\phi}(1 - \chi_n) \right\rangle_{\dot{H}^{s_c}(\mathbb{R}^3)} - \|\chi_n \tilde{\phi}\|_{\dot{H}_D^{s_c}(\Omega_n)}^2.
		\end{align*}
		Using the weak convergence of $g_n$ to $\tilde{\phi}$, \eqref{eq5.13v65}, and \eqref{inverse-1}, we deduce  
		\begin{align*}
			\lim\limits_{n \to \infty} \left(\|f_n\|_{\dot{H}_D^{s_c}(\Omega)}^2 - \|f_n - \phi_n\|_{\dot{H}_D^{s_c}(\Omega)}^2\right) 
			= \|\tilde{\phi}\|_{\dot{H}^{s_c}(\mathbb{R}^3)}^2 \gtrsim \varepsilon^\frac{15}{s_c(2s_c+2)} A^{\frac{4s_c^2+4s_c-15}{s_c(2s_c+2)}}.
		\end{align*}
		This verifies \eqref{inverse-2}.
		
		Next, we establish the decoupling for the $L_{t,x}^{q_0}(\mathbb{R} \times \Omega)$ norm by proving  
		\begin{align}\label{eq5.15v65}
			\liminf\limits_{n \to \infty} \left(\|e^{it\Delta_\Omega}f_n\|_{L_{t,x}^{q_0}}^{q_0} - \|e^{it\Delta_\Omega}(f_n - \phi_n)\|_{L_{t,x}^{q_0}}^{q_0}\right) 
			= \|e^{it\Delta_\Omega}\tilde{\phi}\|_{L_{t,x}^{q_0}}^{q_0}.
		\end{align}
		From this, \eqref{inverse-3} follows by establishing the lower bound  
		\begin{align}\label{eq5.16v65}
			\|e^{it\Delta_\Omega}\tilde{\phi}\|_{L_x^{q_0}}^{q_0} \gtrsim \left(\varepsilon^\frac{15}{s_c(4s_c+4)} A^{\frac{4s_c^2+4s_c-15}{2s_c(2s_c+2)}}\right)^{q_0}.
		\end{align}
		The proof of \eqref{eq5.16v65} is similar to that in Case 1 and is omitted here.
		
		It remains to verify \eqref{eq5.15v65}. Two key observations are required:  
		\begin{align}\label{eq5.17v65}
			e^{it\Delta_{\Omega_n}}(g_n - \chi_n \tilde{\phi}) \to 0 \quad \text{a.e. in } \mathbb{R} \times \mathbb{R}^3,
		\end{align}
		and  
		\begin{align}\label{eq5.18v65}
			\|e^{it\Delta_{\Omega_n}}\chi_n \tilde{\phi} - e^{it\Delta_{\mathbb{R}^3}}\tilde{\phi}\|_{L_{t,x}^{q_0}(\mathbb{R} \times \mathbb{R}^3)} \to 0.
		\end{align}
		For \eqref{eq5.17v65}, combining the definition of $\tilde{\phi}$ with \eqref{eq5.13v65}, we find  
		\begin{align*}
			g_n - \chi_n \tilde{\phi} \rightharpoonup 0 \quad \text{weakly in } \dot{H}^{s_c}(\mathbb{R}^3).
		\end{align*}
		Using Lemma \ref{L:compact} and the fact that $(i\partial_t)^{s_c/2}e^{it\Delta_{\Omega_n}} = (-\Delta_\Omega)^{s_c/2}e^{it\Delta_{\Omega_n}}$, we conclude \eqref{eq5.17v65} by passing to a subsequence. 
		For \eqref{eq5.18v65}, we apply \eqref{eq5.13v65}, the Strichartz inequality, and Theorem \ref{convergence-flow} to deduce the result.
		
		Combining \eqref{eq5.17v65} and \eqref{eq5.18v65}, and passing to a subsequence if necessary, we obtain  
		\begin{align*}
			e^{it\Delta_{\Omega_n}}g_n - e^{it\Delta_{\mathbb{R}^3}}\tilde{\phi} \to 0 \quad \text{a.e. in } \mathbb{R} \times \mathbb{R}^3.
		\end{align*}
		By the refined Fatou lemma (Lemma \ref{LRefinedFatou}), we have  
		\begin{align*}
			\liminf\limits_{n \to \infty} \left(\|e^{it\Delta_{\Omega_n}}g_n\|_{L_{t,x}^{q_0}}^{q_0} - \|e^{it\Delta_{\Omega_n}}g_n - e^{it\Delta_{\mathbb{R}^3}}\tilde{\phi}\|_{L_{t,x}^{q_0}}^{q_0}\right) 
			= \|e^{it\Delta_{\mathbb{R}^3}}\tilde{\phi}\|_{L_{t,x}^{q_0}}^{q_0}.
		\end{align*}
		Combining this with \eqref{eq5.18v65}, \eqref{eq5.13v65}, and a rescaling argument, we conclude \eqref{eq5.15v65}.
		
		\textbf{Case 3}. The proof of this case closely follows the argument in \textit{Case 2}. The main difference lies in the geometry of the two cases, which affects the application of Proposition \ref{convergence-domain} and the analogue of \eqref{eq5.13v65}. Since these key results have already been established for all cases, it suffices to show  
		\begin{align}\label{eq5.19v65}
			\chi_n \tilde{\phi} \to \tilde{\phi}, \quad \text{or equivalently,} \quad \Theta\left(\frac{|x|}{\operatorname{dist}(0, \Omega_n^c)}\right)\tilde{\phi}(x) \to 0 \text{ in } \dot{H}^{s_c}(\mathbb{R}^3).
		\end{align}
		To prove this, define  
		\begin{align*}
			B_n := \left\{x \in \mathbb{R}^3 : |x| \geq \frac{1}{4} \operatorname{dist}(0, \Omega_n^c) \right\}.
		\end{align*}
		Using H\"older's inequality and Sobolev embedding, we estimate  
		\begin{align*}
			\left\|\Theta\left(\frac{|x|}{\operatorname{dist}(0, \Omega_n^c)}\right)\tilde{\phi}(x)\right\|_{\dot{H}^{s_c}(\mathbb{R}^3)} 
			\lesssim \left\|(-\Delta)^\frac{s_c}{2}\tilde{\phi}\right\|_{L^2(B_n)} + \left\|\tilde{\phi}\right\|_{L^\frac{6}{3-2s_c}(B_n)}.
		\end{align*}
		As the measure of $B_n$ shrinks to zero, the right-hand side converges to $0$ by the monotone convergence theorem.
		
		\medskip
		
		\textbf{Case 4}. By passing to a subsequence, we assume $N_n d(x_n) \to d_\infty > 0$. By the weak sequential compactness of bounded sequences in $\dot{H}^{s_c}(\mathbb{R}^3)$, there exists a subsequence and $\tilde{\phi} \in \dot{H}^{s_c}(\mathbb{R}^3)$ such that $g_n \rightharpoonup \tilde{\phi}$ in $\dot{H}^{s_c}(\mathbb{R}^3)$.
		
		Using the characterization of Sobolev spaces,  
		\begin{align*}
			\dot{H}_D^{s_c}(\mathbb{H}) = \left\{g \in \dot{H}^{s_c}(\mathbb{R}^3) : \int_{\mathbb{R}^3} g(x) \psi(x) dx = 0 \text{ for all } \psi \in C_c^\infty(-\mathbb{H}) \right\},
		\end{align*}
		we conclude that $\tilde{\phi} \in \dot{H}_D^{s_c}(\mathbb{H})$ because for any compact set $K$ in the half-space, $K \subset \Omega_n^c$ for sufficiently large $n$, where  
		\begin{align*}
			\Omega_n := N_n R_n^{-1}(\Omega - \{x_n^*\}) \supset \supp(g_n).
		\end{align*}
		As $\tilde{\phi} \in \dot{H}_D^{s_c}(\mathbb{H})$, it follows that  
		\begin{align*}
			x \in \mathbb{H} \Longleftrightarrow N_n^{-1}R_nx + x_n^* \in \mathbb{H}_n := \left\{y : \left(x_n - x_n^*\right) \cdot \left(y - x_n^*\right) > 0\right\} \subset \Omega,
		\end{align*}
		where $\partial \mathbb{H}_n$ represents the tangent plane to $\partial \Omega$ at $x_n^*$. This inclusion yields  
		\begin{align}\label{eq5.20v65}
			\|\tilde{\phi}\|_{\dot{H}_D^{s_c}(\mathbb{H})} = \|\phi_n\|_{\dot{H}_D^{s_c}(\mathbb{H}_n)} = \|\phi_n\|_{\dot{H}_D^{s_c}(\Omega)}.
		\end{align}
		
		To establish \eqref{inverse-1}, we need a lower bound for $\|\tilde{\phi}\|_{\dot{H}_D^{s_c}(\mathbb{H})}$. Let $h := P_1^{\mathbb{H}}\delta_{d_\infty e_3}$. Using the Bernstein inequality, we have  
		\begin{align}\label{eq5.21v65}
			\left\| \left(-\Delta_{\mathbb{H}}\right)^{-\frac{s_c}{2}} h \right\|_{L^2(\Omega)} \lesssim 1,
		\end{align}
		which implies $h \in \dot{H}_D^{-s_c}(\mathbb{H})$. Now, define $\tilde{x}_n := N_nR_n^{-1}(x_n - x_n^*)$. By assumption, $\tilde{x}_n \to d_\infty e_3$. Using Proposition \ref{convergence-domain}, we compute  
		\begin{align*}
			\langle \tilde{\phi}, h \rangle &= \lim\limits_{n \to \infty} \Big(\langle g_n, P_1^{\Omega_n} \delta_{\tilde{x}_n} \rangle + \langle g_n, (P_1^{\mathbb{H}} - P_1^{\Omega_n})\delta_{d_\infty e_3} \rangle + \langle g_n, P_1^{\Omega_n}(\delta_{d_\infty e_3} - \delta_{\tilde{x}_n}) \rangle\Big) \\
			&= \lim\limits_{n \to \infty} \Big(N_n^{s_c - \frac{3}{2}}(e^{it_n\Delta_\Omega}P_{N_n}^\Omega f_n)(x_n) + \langle g_n, P_1^{\Omega_n}(\delta_{d_\infty e_3} - \delta_{\tilde{x}_n}) \rangle\Big).
		\end{align*}
		Following the argument in \eqref{eq5.10v65} and applying \eqref{eq5.9v65} to $v(x) = \left(P_1^{\Omega_n}g_n\right)(x + \tilde{x}_n)$ with $R = \frac{1}{2}N_n d(x_n)$, we obtain  
		\begin{align*}
			\left| \left\langle g_n, P_1^{\Omega_n} \left(\delta_{d_\infty e_3} - \delta_{\tilde{x}_n}\right) \right\rangle \right| \lesssim A\left(d_\infty^{-1} + d_\infty\right)\left|d_\infty e_3 - \tilde{x}_n\right| \to 0 \quad \text{as } n \to \infty.
		\end{align*}
		Thus, we conclude  
		\begin{align*}
			\left|\left\langle \tilde{\phi}, h \right\rangle\right| \gtrsim \varepsilon^\frac{15}{s_c(2s_c+2)}A^{\frac{4s_c^2+4s_c-15}{s_c(2s_c+2)}},
		\end{align*}
		which, together with \eqref{eq5.20v65} and \eqref{eq5.21v65}, proves \eqref{inverse-1}.
		
		Finally, following the same reasoning as in Case 2, we establish \eqref{inverse-2}. This completes the proof of Proposition \ref{inverse-strichartz}.

		To establish the linear profile decomposition for the Schr\"odinger flow $e^{it\Delta_\Omega}$, we require the following two weak convergence results.  
		
		\begin{lemma}[Weak convergence]\label{weak-convergence}
			Assume that $\Omega_n \equiv \Omega$ or $\{\Omega_n\}$ conforms to  one of the last three cases in Proposition \ref{inverse-strichartz}. Let $f \in C_0^\infty(\widetilde{\lim}\,\Omega_n)$ and $\{(t_n, x_n)\}_{n \geq 1} \subset \mathbb{R} \times \mathbb{R}^3$. Assuming either $|t_n| \to \infty$ or $|x_n| \to \infty$, then 
			\begin{align}\label{weak}
				e^{it_n\Delta_{\Omega_n}}f(x + x_n) \rightharpoonup 0
			\end{align}
			weakly in $\dot{H}^{s_c}(\mathbb{R}^3)$ as $n \to \infty$.
		\end{lemma}
		
		\begin{proof}
			Killip-Visan-Zhang \cite[Lemma 5.4]{KillipVisanZhang2016a} demonstrated that $\{e^{it_n\Delta_{\Omega_n}}f(x + x_n)\}_{n=1}^\infty$ converges weakly to zero in $\dot{H}^{1}(\mathbb{R}^3)$ as $n \to \infty$. Noting that $\{e^{it_n\Delta_{\Omega_n}}f(x + x_n)\}_{n=1}^\infty$ is also bounded in $\dot{H}^{s_c}(\mathbb{R}^3)$, we deduce it converges to zero in $\dot{H}^{s_c}(\mathbb{R}^3)$ as well.
		\end{proof}
	\end{proof}
	
	\begin{lemma}[Weak convergence]\label{L:compact} Assume $\Omega_n\equiv\Omega$ or $\{\Omega_n\}$ conforms to one of the last three scenarios considered in  Proposition~\ref{inverse-strichartz}. Let $f_n\in \dot H_D^{s_c}(\Omega_n)$ be such that $f_n\rightharpoonup 0$ weakly in $\dot H^{s_c}(\R^3)$ and let $t_n\to t_\infty\in \R$. Then  
		\begin{align*}
			e^{it_n\Delta_{\Omega_n}} f_n\rightharpoonup 0 \quad\text{weakly in}\quad  \dot{H}^{s_c}(\R^3).
		\end{align*}
	\end{lemma}
	
	\begin{proof} Given any $\phi\in C_c^{\infty}(\R^3)$,
		\begin{align*}
			\big|\langle \big(e^{it_n\Delta_{\Omega_n}}-e^{it_\infty\Delta_{\Omega_n}}\big)f_n, \phi\rangle_{\dot H^{s_c}(\R^3)}\big|
			\lesssim |t_n-t_\infty|^{\frac{s_c}2} \|(-\Delta_{\Omega_n})^{\frac{s_c}2}f_n\|_{L^2} \|\phi\|_{\dot{H}^{2s_c}},
		\end{align*}
		which converges to zero as $n\to \infty$.  To obtain the last inequality above, we have used the spectral theorem together with the elementary inequality $|e^{it_n\lambda}-e^{it_\infty\lambda}|\lesssim |t_n-t_\infty|^{s_c/2}\lambda^{s_c/2}$ for $\lambda\geq 0$.  Thus, we are left to prove
		\begin{align*}
			\int_{\R^3} |\nabla|^{s_c} \bigl[e^{it_\infty\Delta_{\Omega_n}} f_n\bigr](x) |\nabla|^{s_c} \bar\phi(x)dx
			= \int_{\R^3}e^{it_\infty\Delta_{\Omega_n}}f_n(x) (-\Delta)^{s_c}\bar\phi(x)dx\to0\quad\text{as}\quad n\rightarrow\infty 
		\end{align*}
		for all $\phi\in C_0^\infty(\R^3)$. 
		As  $\{e^{it_\infty\Delta_{\Omega_n}} f_n\}_{n=1}^{\infty }$ is uniformly bounded in  $\dot H^{s_c}(\mathbb{R} ^3)$, 
		it suffices to show (using the fact that the measure of $\Omega_n\triangle(\widetilde{\lim}\,\Omega_n)$ converges to zero)
		\begin{align}\label{9:38am}
			\int_{\R^3} e^{it_\infty\Delta_{\Omega_n}} f_n (x) \bar\phi(x)\, dx \to 0 \qtq{as} n\to \infty
		\end{align}
		for all $\phi\in C_c^\infty(\widetilde{\lim} \Omega_n)$.  To prove (\ref{9:38am}), we write
		\begin{align*}
			\langle e^{it_\infty\Delta_{\Omega_n}} f_n, \phi \rangle =\langle f_n, [e^{-it_\infty\Delta_{\Omega_n}} -e^{-it_\infty\Delta_{\Omega_\infty}}]\phi \rangle
			+ \langle f_n,e^{-it_\infty\Delta_{\Omega_\infty}}\phi \rangle,
		\end{align*}
		where $\Omega_\infty$ denotes the limit of $\Omega_n$.  The first term converges to zero by Proposition~\ref{convergence-domain}.  As $f_n\rightharpoonup 0$ in $\dot H^{s_c}(\R^3)$, to see that the second term converges to zero, we merely need to prove that $e^{-it_\infty\Delta_{\Omega_\infty}}\phi\in \dot H^{-s_c}(\R^3)$ for all $\phi\in C_0^\infty(\widetilde{\lim}\,\Omega_n)$.  This in fact follows from the Mikhlin multiplier theorem and Bernstein's inequality:
		\begin{align*}
			\|e^{-it_\infty\Delta_{\Omega_\infty}}\phi\|_{\dot H^{-s_c}(\R^3)}
			&\lesssim\|e^{-it_\infty\Delta_{\Omega_\infty}}P_{\leq 1}^{\Omega_\infty} \phi\|_{L^{\frac6{2s_c+3}}(\R^3)}+\sum_{N\geq 1}\|e^{-it_\infty\Delta_{\Omega_\infty}}P_N^{\Omega_\infty}\phi\|_{L^{\frac6{2s_c+3}}(\R^3)}\\
			&\lesssim \|\phi\|_{L^{\frac6{2s_c+3}}(\mathbb{R} ^3)} + \|(-\Delta_{\Omega_\infty})^2\phi\|_{L^{\frac6{2s_c+3}}(\mathbb{R} ^3)}\lesssim_\phi 1.
		\end{align*}
		This   completes the proof of the lemma.
	\end{proof}
	Now, we are in position to give the linear profile decomposition for the Schr\"odinger propagator $e^{it\Delta_\Omega}$ in $\dot{H}_D^{s_c}(\Omega)$. Indeed, this follows from   the application of Proposition \ref{refined-strichartz} and \ref{inverse-strichartz}.

	\begin{theorem}[$\dot{H}_D^{s_c}(\Omega)$ linear profile decomposition]\label{linear-profile}
		Let $\{f_n\}_{n\geq1}$ be a bounded sequence in $\dot{H}_D^{s_c}(\Omega)$. Passing to a subsequence, there exist $J^*\in\{0,1,\cdots,\infty\}$, $\{\phi_n^j\}_{j=1}^{J^*}\subset\dot{H}_D^{s_c}(\Omega)$, $\{\lambda_n^j\}_{j=1}^{J^*}\subset(0,\infty)$, and $\{(t_n^j, x_n^j)\}_{j=1}^{J^*}\subset\mathbb{R}\times\Omega$, such that for each $j$, one of the following cases holds:
		\begin{itemize}
			\item  \textbf{Case 1.} $\lambda_n^j\equiv\lambda_\infty^j$, $x_n^j=x_\infty^j$ and there exists a $\phi^j\in\dot{H}_D^{s_c}(\Omega)$ such that
			\begin{align*}
				\phi_n^j=e^{it_n^j(\lambda_n^j)^2\Delta_{\Omega}}\phi^j.
			\end{align*}
			We define $[G_n^jf](x):=(\lambda_n^j)^{s_c-\frac{3}{2}}f\big(\frac{x-x_n^j}{\lambda_n^j}\big)$		and
			$\Omega_n^j=(\lambda_n^j)^{-1}(\Omega-\{x_n^j\})$.
		\end{itemize}
		\begin{itemize}
			\item \textbf{Case 2. } $\lambda_n^j\to\infty$, $-\frac{x_n^j}{\lambda_n^j}\to x_\infty^j\in\R^3$. There exists a $\phi^j\in\dot{H}^{s_c}(\R^3)$ such that
			\begin{align*}
				\phi_n^j(x)=G_n^j\big(e^{it_n^j\Delta_{\Omega_{n}^j}}(\chi_n^j\phi^j)\big)(x)\qtq{with}[G_n^jf](x):=(\lambda_n^j)^{s_c-\frac{3}{2}}f\big(\frac{x-x_n^j}{\lambda_n^j}\big),
			\end{align*}
			\begin{equation}
				\Omega_n^j=(\lambda_n^j)^{-1}(\Omega-\{x_n^j\}),\qquad	\chi_n^j(x)=\chi(\lambda_n^jx+x_n^j)\qtq{and}\chi(x)=\Theta\big(\frac{d(x)}{\operatorname{diam}(\Omega^c)}\big).\notag
			\end{equation}
		\end{itemize}
		\begin{itemize}
			\item \textbf{Case 3.}  $\frac{d(x_n^j)}{\lambda_n^j}\to\infty$ and there exists a $\phi^j\in\dot{H}^{s_c}(\R^3)$ such that
			\begin{align*}
				\phi_n^j(x):=G_n^j\big(e^{it_n^j\Delta_{\Omega_{n}^j}}(\chi_n^j\phi^j)\big)(x)\qtq{with}[G_n^jf](x):=(\lambda_n^j)^{s_c-\frac{3}{2}}f\big(\frac{x-x_n^j}{\lambda_n^j}\big),
			\end{align*}
			where
			\begin{equation}
				\Omega_n^j=(\lambda_n^j)^{-1}(\Omega-\{x_n^j\}),\quad\text{and}\quad   	\chi_n^j(x):=1-\Theta\big(\frac{\lambda_n^j|x|}{d(x_n^j)}\big).\notag
			\end{equation}
		\end{itemize}
		\begin{itemize}
			\item  \textbf{Case 4.}   $\lambda_n^j\to0$, $\frac{d(x_n^j)}{\lambda_n^j}\to\infty$ and there exists a $\phi^j\in\dot{H}^{s_c}(\mathbb{H})$ such that
			\begin{align*}
				\phi_n^j(x):=G_n^j\big(e^{it_n^j\Delta_{\Omega_n^j}}\phi^j\big)(x)\quad\text{with}\quad  [G_n^jf](x)=(\lambda_n^j)^{s_c-\frac{3}{2}}f\big(\frac{(R_n^j)^{-1}(x-(x_n^j)^*)}{\lambda_n^j}\big),
			\end{align*}
			$\Omega_n^j=(\lambda_n^j)^{-1}(R_n^j)^{}(\Omega-\{(x_n^j)^*\})$, $(x_n^j)^*\in\partial\Omega$ is chosen by $d(x_n)=|x_n^j-(x_n^j)^*|$ and $R_n^j\in \operatorname{SO}(3)$ satisfies $R_n^je_3=\frac{x_n^j-(x_n^j)^*}{|x_n^j-(x_n^j)^*|}.$
		\end{itemize}
		
		Moreover, for any finite $0 \leq J \leq J^*$, we have the profile decomposition  
		\begin{align*}
			f_n = \sum_{j=1}^J \phi_n^j + W_n^J,
		\end{align*}
		where:
		\begin{itemize}
			\item For all $n$ and $J \geq 1$, $W_n^J \in \dot{H}_D^{s_c}(\Omega)$, and
			\begin{align}\label{profile-1}
				\lim_{J \to J^*} \limsup_{n \to \infty} \|e^{it\Delta_\Omega}W_n^J\|_{L_{t,x}^{q_0}(\mathbb{R} \times \Omega)} = 0.
			\end{align}
			\item For any $J \geq 1$, we have the decoupling property:
			\begin{align}\label{profile-2}
				\lim_{n \to \infty} \left(\|f_n\|_{\dot{H}_D^{s_c}(\Omega)}^2 - \sum_{j=1}^J \|\phi_n^j\|_{\dot{H}_D^{s_c}(\Omega)}^2 - \|W_n^J\|_{\dot{H}_D^{s_c}(\Omega)}^2\right) = 0.
			\end{align}
			\item For any $1 \leq J \leq J^*$,
			\begin{align}\label{profile-3}
				e^{it_n^J\Delta_{\Omega_n^J}}(G_n^J)^{-1}W_n^J \rightharpoonup 0 \quad \text{weakly in } \dot{H}_D^{s_c}(\mathbb{R}^3).
			\end{align}
			\item For all $j \neq k$, we have asymptotic orthogonality:
			\begin{align}\label{profile-4}
				\lim_{n \to \infty} \left(\frac{\lambda_n^j}{\lambda_n^k} + \frac{\lambda_n^k}{\lambda_n^j} + \frac{|x_n^j - x_n^k|^2}{\lambda_n^j\lambda_n^k} + \frac{|t_n^j(\lambda_n^j)^2 - t_n^k(\lambda_n^k)^2|}{\lambda_n^j\lambda_n^k}\right) = \infty.
			\end{align}
		\end{itemize}
		Finally, we may assume for each $j$ that either $t_n^j \equiv 0$ or $|t_n^j| \to \infty$.
	\end{theorem}
	
	\begin{proof}
		We employ an induction argument to complete the proof by extracting one bubble at a time. 
		
		Initially, we set $W_n^0 := f_n$. Suppose that for some $J \geq 0$, we have a decomposition satisfying \eqref{profile-2} and \eqref{profile-3}. Passing to a subsequence if needed, define  
		\begin{align*}
			A_J := \lim\limits_{n \to \infty} \left\|W_n^J\right\|_{\dot{H}_D^{s_c}(\Omega)} \quad \text{and} \quad 
			\epsilon_J := \lim\limits_{n \to \infty} \left\|e^{it\Delta_{\Omega}}W_n^J\right\|_{L_{t,x}^{q_0}(\mathbb{R} \times \Omega)}.
		\end{align*}
		
		If $\epsilon_J = 0$, the induction terminates, and we set $J^* = J$. Otherwise, we apply the inverse Strichartz inequality (see Proposition \ref{inverse-strichartz}) to $W_n^J$. After passing to a subsequence, we obtain $\{\phi_n^{J+1}\} \subseteq \dot{H}_D^{s_c}(\Omega)$, $\{\lambda_n^{J+1}\} \subseteq 2^{\mathbb{Z}}$, and $\{x_n^{J+1}\} \subseteq \Omega$, which correspond to one of the four cases described in the theorem. The parameters provided by Proposition \ref{inverse-strichartz} are renamed as follows:
		\[
		\lambda_n^{J+1} := N_n^{-1} \quad \text{and} \quad t_n^{J+1} := -N_n^2 t_n.
		\]
		
		The profile $\tilde{\phi}^{J+1}$ is defined as a weak limit:
		\begin{align*}
			\tilde{\phi}^{J+1} = w\lim_{n \to \infty}(G_n^{J+1})^{-1}\left[e^{-it_n^{J+1}(\lambda_n^{J+1})^2\Delta_\Omega}W_n^J\right] = w\lim_{n \to \infty} e^{-it_n^{J+1}\Delta_{\Omega_n^{J+1}}}\left[\left(G_n^{J+1}\right)^{-1}W_n^J\right],
		\end{align*}
		where $G_n^{J+1}$ is defined in the theorem.
		In Cases 2, 3, and 4, we set $\phi^{J+1} := \tilde{\phi}^{J+1}$. For Case 1, we define:
		\[
		\phi^{J+1}(x) := G_\infty^{J+1}\tilde{\phi}^{J+1}(x) := \left(\lambda_\infty^{J+1}\right)^{s_c-\frac{3}{2}} \tilde{\phi}^{J+1}\left(\frac{x - x_\infty^{J+1}}{\lambda_\infty^{J+1}}\right).
		\]
		Finally, $\phi_n^{J+1}$ is constructed as stated in the theorem.

		For Case 1, $\phi_n^{J+1}$ can be expressed as
		\[
		\phi_n^{J+1} = e^{it_n^{J+1}(\lambda_n^{J+1})^2\Delta_{\Omega}}\tilde{\phi}^{J+1} = G_\infty^{J+1}e^{it_n^{J+1}\Delta_{\Omega_{\infty}^{J+1}}}\tilde{\phi}^{J+1},
		\]
		where $\Omega_\infty^{J+1} := \left(\lambda_\infty^{J+1}\right)^{-1}\left(\Omega - \left\{x_\infty^{J+1}\right\}\right)$. 
		
		In all four cases, we observe that
		\begin{align}\label{weakly-con-profile}
			\lim\limits_{n \to \infty} \left\| e^{-it_n^{J+1}\Delta_{\Omega_n^{J+1}}}\left(G_n^{J+1}\right)^{-1}\phi_n^{J+1} - \tilde{\phi}^{J+1} \right\|_{\dot{H}^{s_c}(\mathbb{R}^3)} = 0;
		\end{align}
		see also \eqref{eq5.13v65} and \eqref{eq5.19v65} for Cases 2 and 3.
		
		Next, define $W_n^{J+1} := W_n^J - \phi_n^{J+1}$. By \eqref{weakly-con-profile} and the construction of $\tilde{\phi}^{J+1}$ in each case, we have
		\[
		e^{-it_n^{J+1}\Delta_{\Omega_n^{J+1}}}\left(G_n^{J+1}\right)^{-1}W_n^{J+1} \rightharpoonup 0 \quad \text{in } \dot{H}^{s_c}(\mathbb{R}^3) \quad \text{as } n \to \infty,
		\]
		which establishes \eqref{profile-3} at the level $J+1$.
		
		Moreover, from \eqref{inverse-2}, we deduce
		\[
		\lim\limits_{n \to \infty} \left(\left\|W_n^J\right\|_{\dot{H}_D^{s_c}(\Omega)}^2 - \left\|\phi_n^{J+1}\right\|_{\dot{H}_D^{s_c}(\Omega)}^2 - \left\|W_n^{J+1}\right\|_{\dot{H}_D^{s_c}(\Omega)}^2\right) = 0.
		\]
		This, combined with the inductive hypothesis, verifies \eqref{profile-2} at the level $J+1$.

		From Proposition \ref{inverse-strichartz}, passing to a further subsequence, we obtain
		\begin{align}\label{eq5.31v65}
			\begin{split}
				A_{J+1}^2 = \lim\limits_{n \to \infty}\left\|W_n^{J+1}  \right\|_{\dot{H}_D^{s_c}(\Omega)}^2\leqslant A_J^2
				\left(1-C\left(\frac{\epsilon_J}{A_J}\right)^\frac{15 }{s_c(2s_c+2)}   \right) \le A_J^2, \\
				\epsilon_{J+1}^{q_0}=\lim\limits_{n \to\infty}  \left\|e^{it\Delta_\Omega}W_n^{J+1}\right\|_{L_{t,x}^{q_0}( \R\times\Omega)}^{q_0}
				\le \epsilon_J^{\frac{10}{3-2s_c}}
				\left( 1-C\left( \frac{\epsilon_J}{A_J}  \right)^\frac{75}{s_c(2s_c+2)(3-2s_c)}\right).
			\end{split}
		\end{align}

		If $\epsilon_{J+1} = 0$, we terminate the process and set $J^* = J+1$; in this case, \eqref{profile-1} holds automatically. If $\epsilon_{J+1} > 0$, we proceed with the induction. Should the process continue indefinitely, we set $J^* = \infty$. In this scenario, \eqref{eq5.31v65} ensures that $\epsilon_J \xrightarrow{J \to \infty} 0$, which establishes (\ref{profile-1}).
		
		Next, we confirm the asymptotic orthogonality condition \eqref{profile-4} by contradiction. Suppose \eqref{profile-4} does not hold for some pair $(j, k)$. Without loss of generality, assume $j < k$ and that \eqref{profile-4} is valid for all pairs $(j, l)$ where $j < l < k$.
		
		Passing to a subsequence, we let
		\begin{equation}
			\frac{\lambda_n^j}{ \lambda_n^k} \to \lambda_0 \in (0, \infty), \quad \frac{x_n^j - x_n^k}{ \sqrt{\lambda_n^j \lambda_n^k} } \to x_0, \quad\text{and}\quad \frac{t_n^j(\lambda_n^j)^2-t_n^k(\lambda_n^k)^2}{\lambda_n^j\lambda_n^k}\to t_0\qtq{as}n\to\infty.\label{condition-profile} 
		\end{equation}
		From the inductive relation
		\begin{align*}
			W_n^{k-1}= W_n^j-\sum\limits_{l = j+1}^{k - 1} \phi_n^l
		\end{align*}
		and the definition of $\tilde{\phi}^k$, we obtain
		\begin{align*}
			\tilde{\phi}^k&=w\lim_{n\to\infty}  e^{-it_n^k\Delta_{\Omega_{n}^{k}}}\left[\left(G_n^k \right)^{-1} W_n^{k-1}\right]\\&=  w\lim_{n\to\infty}e^{-it_n^k\Delta_{\Omega_{n}^{k}}}\left[\left(G_n^k \right)^{-1} W_n^{j}\right] -  \sum\limits_{l = j+1}^{k-1}  w\lim_{n\to\infty}
			e^{-it_n^k\Delta_{\Omega_{n}^{k}}}\left[\left(G_n^k \right)^{-1} \phi_n^l\right]\\&=:A_1+A_2.
		\end{align*}
		Next, we claim that the weak limits in $A_1$ and $A_2$ are zero, which would be a contradiction to  $\tilde{\phi}^k\neq0$. 
		
		Rewriting $A_1$ as follows:
		\begin{align*}
			e^{-it_n^k\Delta_{\Omega_n^k}}\left[\left(G_n^k\right)^{-1}W_n^j\right]
			&=e^{-it_n^k\Delta_{\Omega_n^k}}\left(G_n^k\right)^{-1}G_n^je^{it_n^j\Delta_{\Omega_n^j}}\left[e^{-it_n^j\Delta_{\Omega_n^j}}\left(G_n^j\right)^{-1}W_n^j\right]\\
			&=\left(G_n^k\right)^{-1}G_n^je^{i\big(t_n^j-t_n^k\tfrac{(\lambda_n^k)^2}{(\lambda_n^j)^2}\big)\Delta_{{\Omega_n^j}}}\left[e^{-it_n^j\Delta_{\Omega_n^j}}\left(G_n^j\right)^{-1}W_n^j\right].
		\end{align*}
		
		Note that by \eqref{condition-profile}, we have 
		\begin{align}
			t_n^j - t_n^k \frac{(\lambda_n^k)^2}{(\lambda_n^j)^2} = \frac{t_n^j (\lambda_n^j)^2 - t_n^k (\lambda_n^k)^2}{\lambda_n^j \lambda_n^k} \cdot \frac{\lambda_n^k}{\lambda_n^j} \to \frac{t_0}{\lambda_0}.\label{E11131}
		\end{align}
		Using this, along with (\ref{profile-3}), Lemma \ref{L:compact}, and the fact that the adjoints of the unitary operators $(G_n^k)^{-1}G_n^{j}$ converge strongly, we deduce that $A_1 = 0.$
		
		To complete the proof of \eqref{profile-4}, it remains to verify that $A_2 = 0$. For all $j < l < k$, we express
		\begin{align*}
			e^{-it_n^k{\Delta_{\Omega_n^k}}}\left[\left(G_n^k\right)^{-1}\phi_n^l\right]
			= \left(G_n^k\right)^{-1}G_n^j e^{i\big(t_n^j - t_n^k \tfrac{(\lambda_n^k)^2}{(\lambda_n^j)^2}\big)\Delta_{\Omega_n^j}}\left[e^{-it_n^j\Delta_{\Omega_n^j}}(G_n^j)^{-1}\phi_n^l\right].
		\end{align*}
		
		By (\ref{E11131}) and Lemma \ref{L:compact}, it suffices to show
		\begin{align*}
			e^{-it_n^j\Delta_{\Omega_n^j}}\left[\left(G_n^j\right)^{-1}\phi_n^l\right] \rightharpoonup 0 \quad \text{weakly in } \dot{H}^{s_c}(\mathbb{R}^3).
		\end{align*}
		
		By density, this reduces to proving the following:
		for all $\phi \in C_0^\infty  \left( \widetilde{\lim} \,  \Omega_n^l \right)$,
		\begin{align}\label{eq5.35v65}
			I_n : = e^{-it_n^j\Delta_{\Omega_n^j}}(G_n^j)^{-1}G_n^le^{it_n^l\Delta_{\Omega_n^l}}\phi\rightharpoonup 0 \qtq{weakly in} \dot H^{s_c}(\R^3)\qtq{as}n\to\infty.
		\end{align}
		Depending on which cases $j$ and $l$ fall into, we can rewrite $I_n$ as follows:
		\begin{itemize}
			\item Case (a): If both $j$ and $l$ conform to Case 1, 2, or 3, then
			\begin{align*}
				I_n=\bigg(\frac{\lambda_n^j}{\lambda_n^l}\bigg)^{\frac{3}{2}-s_c}\bigg[e^{i\big(t_n^l-t_n^j\big(\frac{\lambda_n^j}
					{\lambda_n^l}\big)^2\big)\Delta_{\Omega_n^l}}\phi\bigg]\bigg(\frac{\lambda_n^j x+x_n^j- x_n^l}{\lambda_n^l}\bigg).
			\end{align*}
		\end{itemize}
		\begin{itemize}
			\item Case (b): If $j$ conforms to Case 1, 2, or 3 and $l$ to Case 4, then
			\begin{align*}
				I_n=\bigg(\frac{\lambda_n^j}{\lambda_n^l}\bigg)^{\frac{3}{2}-s_c}\bigg[e^{i\big(t_n^l-t_n^j\big(\frac{\lambda_n^j}
					{\lambda_n^l}\big)^2\big) \Delta_{\Omega_n^l}}\phi\bigg]\bigg(\frac{(R_n^l)^{-1}(\lambda_n^j x+x_n^j-(x_n^l)^*)}{\lambda_n^l}\bigg).
			\end{align*}
		\end{itemize}
		\begin{itemize}
			\item Case (c): If $j$ conforms to Case 4 and $l$ to Case 1, 2, or 3, then
			\begin{align*}
				I_n=\bigg(\frac{\lambda_n^j}{\lambda_n^l}\bigg)^{\frac{3}{2}-s_c}\bigg[e^{i\big(t_n^l-t_n^j\big(\frac{\lambda_n^j}
					{\lambda_n^l}\big)^2\big) \Delta_{\Omega_n^l}}\phi\bigg]\bigg(\frac{R_n^j\lambda_n^j x+(x_n^j)^*-x_n^l}{\lambda_n^l}\bigg).
			\end{align*}
		\end{itemize}
		\begin{itemize}
			\item Case (d): If both $j$ and $l$ conform to Case 4, then
			\begin{align*}
				I_n=\bigg(\frac{\lambda_n^j}{\lambda_n^l}\bigg)^{\frac{3}{2}-s_c}\bigg[e^{i\big(t_n^l-t_n^j\big(\frac{\lambda_n^j}
					{\lambda_n^l}\big)^2\big) \Delta_{\Omega_n^l}}\phi\bigg]\bigg(\frac{(R_n^l)^{-1}(R_n^j\lambda_n^j x+(x_n^j)^*-(x_n^l)^*)}{\lambda_n^l}\bigg).
			\end{align*}
		\end{itemize}
		
		We first prove \eqref{eq5.35v65} in the case where the scaling parameters are not comparable, i.e.,
		\begin{align}\label{A2}
			\lim\limits_{n \to \infty} \left( \frac{\lambda_n^j}{\lambda_n^l} + \frac{\lambda_n^l}{\lambda_n^j} \right) = \infty.
		\end{align}
		
		In this scenario, we handle all four cases simultaneously. Using the Cauchy-Schwarz inequality and \eqref{A2}, for any $\psi \in C_c^\infty(\mathbb{R}^3)$, we have
		\begin{align*}
			\left| \langle I_n, \psi \rangle_{\dot{H}^{s_c}(\mathbb{R}^3)} \right| 
			&\lesssim \min \left( \|(-\Delta)^{s_c} I_n \|_{L^2(\mathbb{R}^3)} \|\psi \|_{L^2(\mathbb{R}^3)}, \|I_n \|_{L^2(\mathbb{R}^3)} \|(-\Delta)^{s_c} \psi \|_{L^2(\mathbb{R}^3)} \right) \\
			&\lesssim \min \left( \left(\frac{\lambda_n^j}{\lambda_n^l}\right)^{s_c} \|(-\Delta)^{s_c} \phi \|_{L^2(\mathbb{R}^3)} \|\psi \|_{L^2(\mathbb{R}^3)}, \left(\frac{\lambda_n^l}{\lambda_n^j}\right)^{s_c} \|\phi \|_{L^2(\mathbb{R}^3)} \|(-\Delta)^{s_c} \psi \|_{L^2(\mathbb{R}^3)} \right),
		\end{align*}
		which tends to zero as $n \to \infty$. 		Therefore, in this case, $A_2 = 0$, leading to the desired contradiction.
		
		Now, we may assume
		\begin{align*}
			\lim_{n \to \infty} \frac{\lambda_n^j}{\lambda_n^l} = \lambda_0 \in (0, \infty).
		\end{align*}
		Proceeding as in the previous case, we further assume that the time parameters diverge, i.e.,
		\begin{align}\label{A3}
			\lim_{n \to \infty} \frac{|t_n^j (\lambda_n^j)^2 - t_n^l (\lambda_n^l)^2|}{\lambda_n^j \lambda_n^l} = \infty.
		\end{align}
		Under this assumption, we have
		\begin{align*}
			\left| t_n^l - t_n^j \frac{(\lambda_n^j)^2}{(\lambda_n^l)^2} \right| = \frac{|t_n^l (\lambda_n^l)^2 - t_n^j (\lambda_n^j)^2|}{\lambda_n^j \lambda_n^l} \cdot \frac{\lambda_n^j}{\lambda_n^l} \to \infty
		\end{align*}
		as $n \to \infty$.
		
		First, we address Case (a). By \eqref{A3} and Lemma \ref{weak-convergence}, we obtain
		\begin{align*}
			\lambda_0^{\frac{3}{2}-s_c}\left(e^{i\big(t_n^l - t_n^j\big(\frac{\lambda_n^j}{\lambda_n^l}\big)^2\big)\Delta_{\Omega_n^l}}\phi\right)(\lambda_0 x + (\lambda_n^l)^{-1}(x_n^j - x_n^l)) \rightharpoonup 0 \quad \text{weakly in } \dot{H}^{s_c}(\mathbb{R}^3),
		\end{align*}
		which implies \eqref{eq5.35v65}.
		
		For Cases (b), (c), and (d), the proof proceeds similarly since $\operatorname{SO}(3)$ is a compact group. Indeed, by passing to a subsequence, we may assume that $R_n^j \to R_0$ and $R_n^l \to R_1$, placing us in a situation analogous to Case (a).
		
		Finally, consider the case where
		\begin{equation}
			\frac{\lambda_n^j}{\lambda_n^l} \to \lambda_0, \quad \frac{t_n^l(\lambda_n^l)^2 - t_n^j(\lambda_n^j)^2}{\lambda_n^j\lambda_n^l} \to t_0, \quad \text{but} \quad \frac{|x_n^j - x_n^l|^2}{\lambda_n^j\lambda_n^l} \to \infty.
		\end{equation}
		In this case, we also have $t_n^l - t_n^j \frac{(\lambda_n^j)^2}{(\lambda_n^l)^2} \to \lambda_0 t_0$. Thus, for Case (a), it suffices to show that
		\begin{equation}
			\lambda_0^{\frac{3}{2}-s_c} e^{it_0 \lambda_0 \Delta_{\Omega_n^l}}\phi(\lambda_0 x + y_n) \rightharpoonup 0 \quad \text{weakly in } \dot{H}^{s_c}(\mathbb{R}^3), \label{E1181}
		\end{equation}
		where
		\begin{align*}
			y_n := \frac{x_n^j - x_n^l}{\lambda_n^l} = \frac{x_n^j - x_n^l}{(\lambda_n^l\lambda_n^j)^{\frac{1}{2}}} \cdot \sqrt{\frac{\lambda_n^j}{\lambda_n^l}} \to \infty \quad \text{as } n \to \infty.
		\end{align*}
		The desired weak convergence \eqref{E1181} follows from Lemma \ref{weak-convergence}.
		
		In Case (b), since $\operatorname{SO}(3)$ is compact, the argument is similar if we can establish
		\begin{equation}
			\frac{|x_n^j - (x_n^l)^*|}{\lambda_n^l} \to \infty \quad \text{as } n \to \infty. \label{E1182}
		\end{equation}
		In fact, this follows from the triangle inequality:
		\begin{align*}
			\frac{|x_n^j - (x_n^l)^*|}{\lambda_n^l} \geq \frac{|x_n^j - x_n^l|}{\lambda_n^l} - \frac{|x_n^l - (x_n^l)^*|}{\lambda_n^l} \geq \frac{|x_n^j - x_n^l|}{\lambda_n^l} - 2d_\infty^l \to \infty.
		\end{align*}
		Case (c) is symmetric to Case (b), so the result for Case (c) follows immediately.

		Now, we handle  case (d). For sufficiently large $n$, we have
		\begin{align*}
			\frac{|(x_n^j)^*-(x_n^l)^*|}{\lambda_n^l}&\geq\frac{|x_n^j-x_n^l|}{\lambda_n^l}-\frac{|x_n^j-(x_n^j)^*|}{\lambda_n^l}-\frac{|x_n^l-(x_n^l)^*|}{\lambda_n^l}\\
			&\geq\frac{|x_n^j-x_n^l|}{\sqrt{\lambda_n^l\lambda_n^j}}\cdot\sqrt{\frac{\lambda_n^j}{\lambda_n^l}}-\frac{d(x_n^j)\lambda_n^j}{\lambda_n^j\lambda_n^l}-\frac{d(x_n^l)}{\lambda_n^l} \notag\\
			&\ge \frac{1}{2}\sqrt{\lambda_0}\frac{|x_n^j-x_n^l|}{\sqrt{\lambda_n^l\lambda_n^j}}-2\lambda_0d_\infty ^j-2d_\infty ^l \rightarrow\infty  \quad\text{as }\quad n\rightarrow\infty .\notag
		\end{align*}
		The desired weak convergence follows again from Lemma  \ref{weak-convergence}. 
	\end{proof}
	\section{Embedding of nonlinear profiles}\label{S4}
	In Section \ref{S5}, we will utilize the linear profile decomposition established in the previous section to prove Theorem \ref{TReduction}. The key challenge lies in deriving a Palais-Smale condition for minimizing sequences of blowup solutions to (\ref{NLS}). This task primarily involves proving a nonlinear profile decomposition for solutions to NLS$_\Omega$. A critical aspect of this process is addressing the scenario where the nonlinear profiles correspond to solutions of the $\dot H^{s_c}$-critical equation in \emph{distinct} limiting geometries. 
	To handle this, we embed these nonlinear profiles, associated with different limiting geometries, back into $\Omega$, following the approach in \cite{KillipVisanZhang2016a}. As nonlinear solutions in the limiting geometries possess global spacetime bounds, we infer that the solutions to NLS$_\Omega$ corresponding to Cases 2, 3, and 4 in Theorem \ref{linear-profile} inherit these spacetime bounds. These solutions to NLS$_{\Omega}$ will reappear as nonlinear profiles in Proposition \ref{Pps}. 
	This section presents three theorems: Theorems \ref{Tembbedding1}, \ref{Tembedding2}, and \ref{Embed3}, which correspond to Cases 2, 3, and 4 of Theorem \ref{linear-profile}, respectively.
	
	As in the previous section, we denote $\Theta:\R^3\to[0,1]$ the smooth function such that
	\begin{align*}
		\Theta(x)=\begin{cases}
			0,&|x|\leq\frac{1}{4},\\
			1,&|x|\geq\frac{1}{2}.
		\end{cases}
	\end{align*}
	
	Our first result in this section consider the scenario when the rescaled obstacles   $\Omega_n^{c}$ are shrinking to a point (i.e. Case 2 in Theorem \ref{linear-profile}).
	\begin{theorem}[Embedding nonlinear profiles for shrinking obstacles]\label{Tembbedding1}
		Let $\{\lambda_n\}\subset2^{\Bbb Z}$ be such that $\lambda_n\to\infty$. Let $\{t_n\}\subset\R$ be such that $t_n\equiv0$ or $|t_n|\to\infty$. Suppose that $\{x_n\}\subset\Omega$ satisfies
		$-\lambda_n^{-1}x_n\to x_\infty\in\R^3$. Let $\phi\in\dot{H}^{s_c}(\R^3)$ and
		\begin{align*}
			\phi_n(x):=\lambda_n^{s_c-\frac{3}{2}}e^{it_n\lambda_n^2\Delta_\Omega}\left[(\chi_n\phi)\left(\frac{x-x_n}{\lambda_n}\right)\right],
		\end{align*}
		where $\chi_n(x)=\chi(\lambda_n x+x_n)$ with  $\chi (x)=\Theta (\frac{d(x)}{\text{diam }\Omega^c})$. Then for $n$ sufficiently large, there exists a global solution $v_n$ to (\ref{NLS}) with initial data $v_n(0)=\phi_n$ such that
		\begin{align*}
			\|v_n\|_{L_{t,x}^{\frac{5\alpha }{2}}(\R\times\Omega)}\lesssim1,
		\end{align*}
		with the implicit constant depending only on $\|\phi\|_{\dot{H}^{s_c}}$. Moreover, for any $\varepsilon>0$, there exists $N_\varepsilon\in\N$ and $\psi_\varepsilon\in C_0^\infty(\R\times\R^3)$ such that for all  $n\ge N_\varepsilon $
		\begin{align}
			\big\|(-\Delta _\Omega)^{\frac{s_c}{2}}[v_n(t-\lambda_n^2t_n,x+x_n)-\lambda_n^{s_c-\frac{3}{2}}\psi_\varepsilon(\lambda_n^{-2}t,\lambda_n^{-1}x)]\big\|_{ L_t^{\frac{5\alpha }{2}}L_x^{\frac{30\alpha }{15\alpha -8}}(\R\times\R^3)}<\varepsilon.\label{approximate-1}
		\end{align}
	\end{theorem}
	\begin{proof}
		Our proof follows the idea of \cite[Theorem 6.1]{KillipVisanZhang2016a}. For the first step, we will construct the global solution to $\dot{H}^{s_c}$-critical NLS in the limiting geometry of $\Omega_n$.
		
		\textbf{Step 1}: Constructing the global solution to  NLS$_{\mathbb{R} ^3}$.
		
		Let $\theta=\frac{1}{100(\alpha +1)}$.  The construction of the global solution on $\R^3$ depends on the choice of time parameter $t_n$. If $t_n\equiv0$, let $w_n$ and $w_\infty$ be the solutions to  NLS$_{\mathbb{R} ^3}$ with initial data $w_n(0)=\phi_{\le\lambda_n^\theta}$ and $w_\infty(0)=\phi$.  
		
		Otherwise,	if $t_n\to\pm\infty$, let $w_n$ be the solutions to NLS$_{\mathbb{R} ^3}$ such that
		\begin{align*}
			\big\|w_n(t)-e^{it\Delta}\phi_{\le\lambda_n^\theta}\big\|_{\dot{H}^{s_c}(\R^3)}\to0,\qtq{as} t\to\pm\infty.
		\end{align*}
		Similarly, we denote $w_\infty$ by the solution  to  NLS$_{\mathbb{R} ^3}$  such that 
		\begin{equation}	\big\|w_\infty(t)-e^{it\Delta}\phi\big\|_{\dot{H}^{s_c}(\R^3)}\to0,\qtq{as}t\to\pm\infty.\label{E11101}
		\end{equation}
		By \cite{Murphy2014} and assumption made in Theorem \ref{T1}, both $w_n(t)$ and $w_\infty(t)$ are global solutions and satisfy
		\begin{equation}
			\|w_n\|_{ L_t^{\frac{5\alpha }{2}}L_x^{\frac{5\alpha }{2}}(\R\times\R^3)}+\|w_\infty\|_{ L_t^{\frac{5\alpha }{2}}L_x^{\frac{5\alpha }{2}}(\R\times\R^3)}\lesssim1.\label{E11102}
		\end{equation}
		Moreover, by the perturbation theory in \cite{Murphy2014}, 
		\begin{align}
			\lim_{n\to\infty}\big\|w_n(t)-w_\infty(t)\big\|_{ L_t^{\frac{5\alpha }{2}}L_x^{\frac{5\alpha }{2}}(\R\times\R^3)}=0.\label{perturb}
		\end{align}
		From the Bernstein inequality, we have
		\begin{align*}
			\|\phi_{\le \lambda_n^\theta}\|_{\dot{H}^s(\R^3)}\lesssim\lambda_n^{\theta(s-s_c)},\qtq{for any }s\geqslant s_c.
		\end{align*}
		The persistence of regularity yields that
		\begin{align*}
			\big\||\nabla|^{s}w_n\big\|_{\dot S^{s_c}(\R\times\R^3)}\lesssim\lambda_n^{\theta s} \qtq{for any}s\geqslant0,
		\end{align*}
		which together with the   Gagliardo-Nirenberg inequality 
		\[
		\|f\|_{L^\infty(\R^3)}\lesssim \|f\|_{\dot{H}^{s_c}(\R^3)}^\frac{1}{2}\|f\|_{\dot{H}^{3-s_c}(\R^3)}^\frac{1}{2} 
		\]
		implies that 
		\begin{align}\label{key-1}
			\big\||\nabla|^{s}w_n\big\|_{L_{t,x}^\infty(\R\times\R^3)}\lesssim\lambda_n^{\theta(s+\frac{3}{2}-s_c)},\quad\text{for all} \quad s\ge0.
		\end{align}
		Finally, using the structure of the NLS$_{\R^3}$, we have
		\begin{align}\label{key-2}
			\|\partial_tw_n\|_{L_{t,x}^\infty(\R\times\R^3)}\lesssim\|\Delta w_n\|_{L_{t,x}^\infty(\R\times\R^3)}+\|w_n\|_{L_{t,x}^\infty(\R\times\R^3)}^{\alpha+1}\lesssim\lambda_n^{\theta(\frac{7}{2}-s_c)}.
		\end{align}
		
		\textbf{Step 2}. Constructing the approximate solution to (\ref{NLS}).
		
		As discussed in Case 2 of Proposition \ref{inverse-strichartz}, we let $\Omega_n=\lambda_n^{-1}(\Omega-\{x_n\})$. One may want to embed $w_n(t)$ to $\Omega_n$ by taking $\tilde{v}_n(t)=\chi_nw_n(t)$ directly. However, this is not a approximation of (\ref{NLS}). Instead, we take
		\begin{align*}
			z_n(t):=i\int_{0}^{t}e^{i(t-\tau)\Delta_{\Omega_{n}}}(\Delta_{\Omega_{n}}\chi_n)w_n(\tau,-\lambda_n^{-1}x_n)d\tau.
		\end{align*}
		This can allow us to  control the reflected waves near the boundary. Moreover, we have the following properties.
		\begin{lemma}\label{zn}
			For all $T>0$, we have
			\begin{gather}\label{embed-lem-1}
				\limsup_{n\to\infty}\|(-\Delta _\Omega)^{\frac{s_c}{2}}z_n\|_{L_{t}^{\frac{5\alpha }{2} } L_{x}^{\frac{30\alpha }{15\alpha -8}}([-T,T]\times\Omega_{n})}=0,\\
				\big\|(-\Delta_{\Omega_{n}})^\frac{s}{2}z_n\big\|_{L_{t}^\infty L_{x}^2([-T,T]\times\Omega_{n})}\lesssim\lambda_n^{s-\frac{3}{2}+\theta(\frac{7}{2}-s_c)}(T+\lambda_n^{-2\theta})\qtq{for all}0\le s<\frac{3}{2}.\label{embed-lem-2}
			\end{gather}
		\end{lemma}
		\begin{proof}
			Integrating by parts, we write
			\begin{align*}
				z_n(t)&=-\int_{0}^{t}\big(e^{it\Delta_{\Omega_{n}}}\partial_\tau e^{-i\tau\Delta_{\Omega_{n}}}\chi_n\big)w_n(\tau,-\lambda_n^{-1}x_n)d\tau\\
				&=-\chi_nw_n(t,-\lambda_n^{-1}x_n)+e^{it\Delta_{\Omega_{n}}}\big(\chi_nw_n(0,-\lambda_n^{-1}x_n)\big)\\
				&\hspace{3ex}+\int_{0}^{t}\big(e^{i(t-\tau)\Delta_{\Omega_{n}}}\chi_n\big)\partial_\tau w_n(\tau,-\lambda_n^{-1}x_n)d\tau.
			\end{align*}
			By the Strichartz estimate, the equivalence of Sobolev norms, \eqref{key-1} and \eqref{key-2}, we have
			\begin{align*}
				&\big\|(-\Delta_{\Omega_{n}})^\frac{s}{2}z_n\big\|_{L_t^\infty L_x^2([-T,T]\times\Omega_{n})}\notag\\
				&\lesssim\big\|(-\Delta)^\frac{s}{2}\chi_nw_n(t,-\lambda_n^{-1}x_n)\big\|_{L_t^\infty L_x^2([-T,T]\times\Omega_{n})} +\big\|(-\Delta_{\Omega_{n}})^\frac{s}{2}\chi_nw_n(0,-\lambda_n^{-1}x_n)\big\|_{L^2([-T,T]\times\Omega_{n})}\\
				&\hspace{3ex}+\big\|(-\Delta)^\frac{s}{2}\chi_n\partial_tw_n(t,-\lambda_n^{-1}x_n)\big\|_{L_t^1L_x^2([-T,T]\times\Omega_{n})}\\
				&\lesssim\lambda_n^{s-\frac{3}{2}+\theta (\frac{3}{2}-s_c)}+T\lambda_n^{s-\frac32+\theta( \frac{7}{2}-s_c)}.
			\end{align*}
			This proves \eqref{embed-lem-2}.  By a similar argument, we can prove (\ref{embed-lem-1}). This completes the proof of   lemma \ref{zn}.
		\end{proof}
		We are now prepared to construct the approximate solution  
		\begin{align*}
			\tilde{v}_n(t,x) := 
			\begin{cases}
				\lambda_n^{s_c-\frac{3}{2}}(\chi_n w_n + z_n)(\lambda_n^{-2}t, \lambda_n^{-1}(x-x_n)), & |t| \leqslant \lambda_n^2 T, \\
				e^{i(t-\lambda_n^2T)\Delta_{\Omega}} \tilde{v}_n(\lambda_n^2T,x), & t > \lambda_n^2 T, \\
				e^{i(t+\lambda_n^2T)\Delta_{\Omega}} \tilde{v}_n(-\lambda_n^2T,x), & t < -\lambda_n^2 T,
			\end{cases}
		\end{align*}
		where $T > 0$ is a parameter to be determined later.
		
		We first observe that $\tilde{v}_n$ has a finite scattering norm. Indeed, this follows from Lemma \ref{zn}, the Strichartz estimate, and a change of variables:
		\begin{align}
			\|\tilde{v}_n\|_{L_{t,x}^{\frac{5\alpha }{2}}(\R\times\Omega)}&\lesssim\big\|\chi_nw_n+z_n\big\|_{L_{t,x}^{\frac{5\alpha }{2}}([-T,T]\times \Omega)}+\|(\chi_nw_n+z_n)(\pm T)\|_{\dot{H}_D^{s_c}(\Omega_{n})}\notag\\
			&\lesssim\|w_n\|_{L_{t,x}^{\frac{5\alpha }{2}}(\R\times\R^3)}+\|z_n\|_{L_{t,x}^{\frac{5\alpha }{2}}([-T,T]\times \Omega)}+\|\chi_n\|_{L_x^\infty(\Omega_{n})}\big\||\nabla|^{s_c}w_n\big\|_{L_t^\infty L_x^2(\R\times\R^3)}\notag\\
			&\hspace{3ex}	+\big\||\nabla|^{s_c}\chi_n\big\|_{L^{\frac{3}{s_c}}}\|w_n\|_{L_t^\infty L_x^{\frac{6}{3-2s_c}}(\R\times\R^3)}+\big\|(-\Delta_{\Omega_{n}})^\frac{s_c}{2}z_n\big\|_{L_t^\infty L_x^2([-T,T]\times \Omega)}\notag\\
			&\lesssim 1+ \|z_n\|_{L_{t,x}^{\frac{5\alpha }{2}}([-T,T]\times \Omega)}++\big\|(-\Delta_{\Omega_{n}})^\frac{s_c}{2}z_n\big\|_{L_t^\infty L_x^2([-T,T]\times \Omega)}<+\infty . \label{step-2}
		\end{align}
		
		\textbf{Step 3.}   {Asymptotic agreement of the initial data.}
		
		In this step, we aim to show that
		\begin{align}\label{step-3}
			\lim_{T\to\infty} \limsup_{n\to\infty} \big\|e^{it\Delta_{\Omega}}\big(\tilde{v}_n(\lambda_n^2t_n) - \phi_n\big)\big\|_{L_t^{\frac{5\alpha}{2}}\dot{H}_D^{s_c,\frac{30\alpha}{15\alpha-8}}(\mathbb{R}\times\Omega)} = 0.
		\end{align}
		
		We first consider the case when $t_n \equiv 0$. Using H\"older's inequality, the Strichartz estimate, and a change of variables, we obtain
		\begin{align*}
			&\big\|e^{it\Delta_{\Omega}}\big(\tilde{v}_n(0) - \phi_n\big)\big\|_{L_t^{\frac{5\alpha}{2}}\dot{H}_D^{s_c,\frac{30\alpha}{15\alpha-8}}(\mathbb{R}\times\Omega)}   \lesssim \|\tilde{v}_n(0) - \phi_n\|_{\dot{H}_D^{s_c}(\Omega)} 
			\lesssim \|\chi_n \phi_{\le \lambda_n^\theta} - \chi_n \phi\|_{\dot{H}_D^{s_c}(\Omega)} \\
			&\quad \lesssim \big\||\nabla|^{s_c}\chi_n\big\|_{L_x^{\frac{3}{s_c}}(\Omega)} \|\phi_{\le \lambda_n^\theta} - \phi\|_{L_x^{\frac{6}{3-2s_c}}(\Omega)} + \|\chi_n\|_{L_x^\infty(\Omega)} \big\||\nabla|^{s_c}(\phi_{\le \lambda_n^\theta} - \phi)\big\|_{L_x^2(\Omega)} \to 0, \quad \text{as } n \to \infty.
		\end{align*}
		
		Next, we address the case when $|t_n| \to \infty$. By symmetry, it suffices to consider $t_n \to +\infty$, as the case $t_n \to -\infty$ can be treated analogously. Since $T > 0$ is fixed, for sufficiently large $n$, we have $t_n > T$, which implies
		\begin{align*}
			\tilde{v}_n(\lambda_n^2t_n, x) &= e^{i(t_n - T)\lambda_n^2\Delta_{\Omega}} \tilde{v}_n(\lambda_n^2T, x) \\
			&= e^{i(t_n - T)\lambda_n^2\Delta_{\Omega}} \left[\lambda_n^{s_c - \frac{3}{2}} (\chi_n w_n + z_n)\big(T, \frac{x - x_n}{\lambda_n}\big)\right].
		\end{align*}
		Applying a change of variables, H\"older's inequality, and the Strichartz estimate, we obtain
		\begin{align*}
			& \big\|(-\Delta_\Omega)^\frac{s_c}{2}e^{it\Delta_{\Omega}}\left[\tilde{v}_n(\lambda_n^2t_n)-\phi_n\right]\big\|_{L_{t}^{\frac{5\alpha}{2}}L_x^{\frac{30\alpha}{15\alpha-8}}(\R\times\Omega)}\\
			&= \big\|(-\Delta_{\Omega_n})^\frac{s_c}{2}\left[e^{i(t-T)\Delta_{\Omega_{n}}}(\chi_nw_n+z_n)(T)-e^{it\Delta_{\Omega_{n}}}(\chi_n\phi)\right]\big\|_{L_{t}^{\frac{5\alpha}{2}}L_x^{\frac{30\alpha}{15\alpha-8}}(\R\times\Omega_n)}\\
			&\lesssim\big\|(-\Delta_{\Omega_n})^\frac{s_c}2z_n(T)\big\|_{L^2_x}+\big\|(-\Delta_{\Omega_{n}})^\frac{s_c}2\big(\chi_n(w_n-w_\infty)(T)\big)\big\|_{L_x^2}\\
			&\hspace{2ex}+\big\|(-\Delta_{\Omega_{n}})^\frac{s_c}2[e^{i(t-T)\Delta_{\Omega_{n}}}(\chi_nw_\infty)(T)-e^{it\Delta_{\Omega_{n}}}(\chi_n\phi)]\big\|_{L_{t}^{\frac{5\alpha}{2}}L_x^{\frac{30\alpha}{15\alpha-8}}(\R\times\Omega_n)}.
		\end{align*}
		Using \eqref{perturb} and \eqref{embed-lem-2}, we have
		\begin{align*}
			&\big\|(-\Delta_{\Omega_n})^\frac{s_c}2z_n(T)\big\|_{L_x^2}+\big\|(-\Delta_{\Omega_{n}})^\frac{s_c}2\big(\chi_n(w_n-w_\infty)(T)\big)\big\|_{L_x^2}\\	&\lesssim\lambda_n^{s_c-\frac{3}{2}+\theta(\frac{7}{2}-s_c)}(T+\lambda_n^{-2\theta})+\big\|(-\Delta_{\Omega_{n}})^\frac{s_c}2)\chi_n\big\|_{L_x^\frac{3}{s_c}}\|w_n-w_\infty\|_{L_t^\infty L_x^{\frac{6}{3-2s_c}}}\\
			&\hspace{3ex}+\|\chi_n\|_{L^\infty}\big\|(-\Delta_{\Omega_{n}})^\frac{s_c}2(w_n-w_\infty)\|_{L_t^\infty L_x^2}\to0\qtq{as}n\to\infty.
		\end{align*}
		Thus, we are left to verify that 
		\begin{align*}
			\lim_{T\to\infty}\limsup_{n\to\infty}\big\|(-\Delta_{\Omega_{n}})^{\frac{s_c}2}\left[e^{i(t-T)\Delta_{\Omega_{n}}}(\chi_nw_\infty)(T)-e^{it\Delta_{\Omega_{n}}}(\chi_n\phi)\right]\big\|_{L_t^\frac{5\alpha}{2}L_x^{\frac{30\alpha}{15\alpha-8}}(\R\times\Omega_{n})}=0.
		\end{align*}
		By the  triangle inequality and the Strichartz estimate,  
		\begin{align*}
			&\hspace{3ex}	\big\|(-\Delta_{\Omega_{n}})^\frac{s_c}2e^{i(t-T)\Delta_{\Omega_{n}}}\big(\chi_nw_\infty(T)\big)-e^{it\Delta_{\Omega_{n}}}(\chi_n\phi)\big\|_{L_t^\frac{5\alpha}{2}L_x^{\frac{30\alpha}{15\alpha-8}}(\R\times \Omega_n)}\\
			&\lesssim\big\|(-\Delta_{\Omega_{n}})^\frac{s_c}2\big(\chi_nw_\infty(T)\big)-\chi_n(-\Delta)^\frac{s_c}{2}w_\infty(T)\big\|_{L_x^2}\\
			&\hspace{3ex}+\big\|[e^{i(t-T)\Delta_{\Omega_{n}}}-e^{i(t-T)\Delta}][\chi_n(-\Delta)^\frac{s_c}2w_\infty(T)]\big\|_{L_t^\frac{5\alpha}{2}L_x^{\frac{30\alpha}{15\alpha-8}}(\R\times\Omega_{n})}\\
			&\hspace{3ex}+\big\|e^{-iT\Delta}[\chi_n(-\Delta)^\frac{s_c}{2}w_\infty(T)]-\chi_n(-\Delta)^\frac{s_c}{2}\phi\big\|_{L_x^2}\\
			&\hspace{3ex}+\big\| [e^{it\Delta _{\Omega_n}}-e^{it\Delta }][\chi_n(-\Delta)^\frac{s_c}{2}\phi]\big\|_{L_t^{\frac{5\alpha}{2}}L_x^{\frac{30\alpha}{15\alpha-8}}(\R\times\Omega_{n})}\\
			&\hspace{3ex}+\big\|(-\Delta_{\Omega_{n}})^\frac{s_c}2(\chi_n\phi)-\chi_n(-\Delta)^\frac{s_c}{2}\phi\big\|_{L_x^2}\\
			&\stackrel{\triangle}{=}I_1+I_2+I_3+I_4+I_5.
		\end{align*}
		The fact that $I_2$ and $I_4$ converge to zero as $n \to \infty$ follows directly from Theorem \ref{convergence-flow} and the density of $C_c^\infty$ functions supported in $\mathbb{R}^3$ minus a point within $L^2_x$. 
		
		Next, we estimate $I_1$, $I_3$, and $I_5$. Using the triangle inequality, Proposition \ref{P1}, and the monotone convergence theorem, for any $f \in \dot{H}^{s_c}(\mathbb{R}^3)$, we obtain
		\begin{align}
			&\hspace{2ex} \big\|\big(-\Delta_{\Omega_{n}}\big)^\frac{s_c}{2}(\chi_n f) - \chi_n (-\Delta)^\frac{s_c}{2} f \big\|_{L^2_x} \notag \\
			&\lesssim \big\|(1 - \chi_n)(-\Delta)^\frac{s_c}{2}f\big\|_{L^2_x} + \big\|(-\Delta)^\frac{s_c}{2}\big((1 - \chi_n)f\big)\big\|_{L^2_x} \notag \\
			&\hspace{3ex} + \big\|(-\Delta_{\Omega_{n}})^\frac{s_c}{2}(\chi_n f) - (-\Delta)^\frac{s_c}{2}(\chi_n f)\big\|_{L^2_x} \to 0 \quad \text{as } n \to \infty. \notag
		\end{align}
		This completes the proof for $I_5$, and thus for $I_1$ as well.
		
		Finally, for the term $I_3$, we apply (\ref{E11101}) along with the monotone convergence theorem to find
		\begin{align*}
			I_3 &\lesssim \big\|(1 - \chi_n)(-\Delta)^\frac{s_c}{2}w_\infty(T)\big\|_{L^2_x} + \big\|(1 - \chi_n)(-\Delta)^\frac{s_c}{2}\big\|_{L^2_x} \\
			&\hspace{3ex} + \big\|e^{-iT\Delta}(-\Delta)^\frac{s_c}{2}w_\infty(T) - (-\Delta)^\frac{s_c}{2}\phi\big\|_{L^2_x} \to 0,
		\end{align*}
		first taking $n \to \infty$, and then $T \to \infty$.
		
		\textbf{Step 4}. 	We demonstrate that $\tilde{v}_n$ serves as an approximate solution to \eqref{NLS} in the sense that 
		\begin{align*}
			i\partial_t\tilde{v}_n + \Delta_{\Omega}\tilde{v}_n = |\tilde{v}_n|^{\alpha}\tilde{v}_n + e_n,
		\end{align*}
		where $e_n$ satisfies the smallness condition
		\begin{equation}
			\lim_{T \to \infty} \limsup_{n \to \infty} \big\|e_n\big\|_{\dot{N}^{s_c}(\mathbb{R} \times \Omega)} = 0. \label{E1110x1}
		\end{equation}
		
		First, consider the case of a large time scale $t > \lambda_n^2 T$. By symmetry, the case $t < -\lambda_n^2 T$ can be handled similarly. Using the equivalence of Sobolev spaces, Strichartz estimates, and H\"older's inequality, we obtain
		\begin{align*}
			&\big\|(-\Delta _\Omega)^{\frac{s_c}{2}}e_n\big\|_{ L_t^{\frac{5\alpha }{2}}L_x^{\frac{30\alpha }{27\alpha -8}}(\{t>\lambda_n^2 T\}\times\Omega)}\lesssim\big\|(-\Delta_{\Omega})^\frac{s_c}{2}(|\tilde{v}_n|^{\alpha}\tilde{v}_n)\big\|_{ L_t^{\frac{5\alpha }{2}}L_x^{\frac{30\alpha }{27\alpha -8}}(\{t>\lambda_n^2 T\}\times\Omega)}\\
			&\lesssim\big\|(-\Delta_{\Omega})^\frac{s_c}{2}\tilde{v}_n\big\|_{L_t^{\frac{5\alpha }{2}}L_x^{ \frac{30\alpha }{15\alpha -8}}(\{t>\lambda_n^2T\}\times\Omega)}\|\tilde{v}_n\|_{L_{t,x}^{\frac{5\alpha}{2}}(\{t>\lambda_n^2T\}\times\Omega)}^\alpha\\
			&\lesssim\big\|(-\Delta_{\Omega})^\frac{s_c}{2}[\chi_nw_n(T)+z_n(T)]\big\|_{L_x^2}\|\tilde{v}_n\|_{L_{t,x}^{\frac{5\alpha}{2}}(\{t>\lambda_n^2T\}\times\Omega)}^\alpha\\
			&\lesssim\big(1+\lambda_n^{s_c-\frac{3}{2}+\theta(\frac{7}{2}-s_c)}(T+\lambda_n^{-2\theta})\big)\|\tilde{v}_n\|_{L_{t,x}^{\frac{5\alpha}{2}}(\{t>\lambda_n^2T\}\times\Omega)}^\alpha.
		\end{align*} 
		Therefore, to establish (\ref{E1110x1}),  it suffices to prove that
		\begin{align}\label{convergence-6.1}
			\lim_{T\to\infty}\limsup_{n\to\infty}\big\|e^{i(t-\lambda_n^2T)\Delta_{\Omega}}\tilde{v}_n(\lambda_n^2T)\big\|_{L_{t,x}^\frac{5\alpha}{2}(\{t>\lambda_n^2T\}\times\Omega)}=0.
		\end{align}
		We now prove (\ref{convergence-6.1}).  By  the spacetime bounds (\ref{E11102}), the global solution   $w_\infty $ scatters. Let  $\phi_+$ denote the forward asymptotic state, that is,
		\begin{align}\label{scattering}
			\big\|w_\infty-e^{it\Delta}\phi_+\big\|_{\dot{H}^{s_c}(\R^3)}\to0,\qtq{as}t\to\pm\infty.
		\end{align}
		It then follows from Strichartz estimate, H\"older's inequality and change of variables that 
		\begin{align*}
			& \big\|e^{i(t-\lambda_n^2T)\Delta_{\Omega}}\tilde{v}_n(\lambda_n^2T)\big\|_{L_{t,x}^\frac{5\alpha}{2}(\{t>\lambda_n^2T\}\times\Omega)} \lesssim\big\|e^{it\Delta_{\Omega_n}}(\chi_nw_n(T)+z_n(T))\big\|_{L_{t,x}^\frac{5\alpha}{2}([0,\infty)\times\Omega_n)}\\
			&\lesssim \big\|(-\Delta_{\Omega_n})^{\frac{s_c}2}z_n(T)\big\|_{L_x^2}+\big\|(-\Delta_{\Omega_n})^{\frac{s_c}2}[\chi_n(w_n(T)-w_\infty(T))]\big\|_{L_x^2}\\
			&\quad+\big\|(-\Delta_{\Omega_n})^{\frac{s_c}2}[\chi_n(w_{\infty}(T)-e^{iT\Delta}w_+)]\big\|_{L_x^2}+\big\|e^{it\Delta_{\Omega_n}}[\chi_ne^{iT\Delta}w_+]\big\|_{L_{t,x}^{\frac{5\alpha}{2}}([0,\infty)\times\Omega_n)}\\
			&\lesssim \lambda_n^{s_c-\frac{3}2+\theta(\frac72-s_c)}(T+\lambda_n^{-2\theta})+\big\|w_n(T)-w_\infty(T)\big\|_{\dot H^{s_c}}+\big\|w_\infty(T)-e^{iT\Delta}w_+\big\|_{\dot H^{s_c}}\\
			&\quad+\big\|[e^{it\Delta_{\Omega_n}}-e^{it\Delta}][\chi_ne^{iT\Delta}w_+]\big\|_{L_{t,x}^{\frac{5\alpha}{2}}([0,\infty)\times\R^3)} +\big\|(-\Delta)^{\frac{s_c}2} [(1-\chi_n)e^{iT\Delta}w_+]\big\|_{L_x^2}\\
			&\quad+\big\|e^{it\Delta}w_+\big\|_{L_{t,x}^{\frac{5\alpha}{2}}((T,\infty)\times\R^3)},
		\end{align*}
		which converges to zero by first letting  $n\rightarrow\infty $
		and then $T\to\infty$  by (\ref{embed-lem-2}),  \eqref{scattering},  Theorem \ref{convergence-flow},  and the monotone convergence theorem. 
		
		Now, we consider the case that $|t_n|\leq \lambda_n^2T$. For these values of time, by the direct calculus we have
		\begin{align*}
			e_n(t,x)&=[(i\partial_t+\Delta_\Omega )\tilde v_n- |\tilde v_n|^\alpha\tilde v_n](t,x)\\
			&=-\lambda_n^{s_c-\frac72}[\Delta\chi_n](\lambda_n^{-1}(x-x_n))w_n(\lambda_n^{-2}t,-\lambda_n^{-1}x_n)+\lambda_n^{s_c-\frac72}[\Delta\chi_n w_n](\lambda_n^{-2}t,\lambda_n^{-1}(x-x_n))\\
			&\quad+2\lambda_n^{s_c-\frac72}(\nabla\chi_n\cdot\nabla w_n)(\lambda_n^{-2}t, \lambda_n^{-1}(x-x_n))\\
			&\quad+\lambda_n^{s_c-\frac72}[\chi_n|w_n|^\alpha w_n-|\chi_nw_n+z_n|^\alpha(\chi_nw_n+z_n)](\lambda_n^{-2}t,\lambda_n^{-1}(x-x_n)).
		\end{align*}
		By a change of variables and the equivalence of Sobolev norms Theorem \ref{TEquivalence}, we obtain
		\begin{align*}
			\big\|(-\Delta_{\Omega})^\frac{s_c}2e_n\big\|_{  \dot N^{s_c}(\R\times\Omega)}\notag &\lesssim\big\|(-\Delta)^\frac{s_c}2[\Delta\chi_n(w_n(t,x)-w_n(t,\lambda_n^{-1}x_n))]\big\|_{L_t^{2}L_x^{\frac{6}{5}}([-T,T]\times\Omega_{n})}\\
			&\hspace{3ex}+\big\|(-\Delta)^\frac{s_c}{2}\big(\nabla\chi_n\nabla w_n\big)\big\|_{L_t^{2}L_x^{\frac{6}{5}}([-T,T]\times\Omega_{n})}\\
			&\hspace{3ex}+\big\|(-\Delta)^\frac{s_c}{2}\big[(\chi_n-\chi_n^{\alpha+1})|w_n|^{\alpha}w_n\|_{L_t^{2}L_x^{\frac{6}{5}}([-T,T]\times\Omega_{n})}\\
			&\hspace{3ex}+ \|(-\Delta )^{s_c} [|\chi_n w_n+z_n|^{\alpha }(\chi_n w_n z_n)-|\chi_n w_n|^{\alpha }\chi_n w_n]\|_{L_t^{\frac{5\alpha }{2}}L_x^{\frac{30\alpha }{15\alpha -8}}([-T,T]\times \Omega_n)}  \notag\\
			&\stackrel{\triangle}{=}J_1+J_2+J_3+J_4.
		\end{align*}
		Using H\"older, the fundamental theorem of calculus, and 
		\eqref{key-1}, we estimate
		\begin{align*}
			J_1&\lesssim T^{\frac{1}{2}}\big\|(-\Delta)^\frac{s_c}{2}(w_n(t,x)-w_n(t,-\lambda_n^{-1}x_n))\big\|_{L_{t,x}^\infty}\|\Delta \chi_n\|_{L^\frac{6}{5}}\\
			&\hspace{3ex}+T^\frac{1}{2}\|w_n-w_n(t,-\lambda_n^{-1}x_n)\|_{L_{t,x}^\infty(\mathbb{R} \times \text{supp}\Delta \chi_n)}\big\|(-\Delta)^{\frac{s_c}{2}}(\Delta\chi_n)\big\|_{L_x^\frac{6}{5}}\\
			&\lesssim  T^{\frac{1}{2}}\lambda_n^{-\frac{1}{2}+\frac{3}{2}\theta }+T^{\frac{1}{2}}\lambda_n^{-1+\theta (\frac{5}{2}-s_c)}\lambda_n^{s_c-\frac{1}{2}}\rightarrow0\quad\text{as}\quad n\rightarrow\infty .
		\end{align*}
		By a similar argument, we can show that  $J_2\rightarrow0$ as  $n\rightarrow\infty $ and we omit the details.

		Next,  we turn our attention to $J_3$.  
		By Lemma \ref{LFractional product rule},  H\"older's inequality and (\ref{key-1}), we have
		\begin{align*}
			J_3&\lesssim\big\||\nabla|^{s_c}\chi_n\big\|_{L_x^{\frac{6}{5}}}\|w_n\|_{L_t^\infty L_x^{\infty }}^{\alpha+1} 
			+\big\|\chi_n-\chi_n^{\alpha+1}\big\|_{L_x^{\frac{6}{5}}}\|w_n\|_{L_t^\infty L_x^{\infty}}^\alpha\big\||\nabla|^{s_c}w_n\big\|_{L_t^\infty L_x^{\infty}}\\
			&\lesssim\lambda_n^ {s_c-\frac{5}{2}+\theta (\alpha +1)(\frac{3}{2}-s_c)}+\lambda_n^{-\frac{5}{2}+\theta \alpha (\frac{3}{2}-s_c)+\frac{3}{2}\theta }\rightarrow0\quad\text{as} \quad n\rightarrow\infty .\notag
		\end{align*}
		
		Finally, we consider  $J_4$.  By Lemma \ref{Lnonlinearestimate}, 
		\begin{align}
			J_4&\lesssim  \left(\|\chi_n w_n\|^{\alpha -1}_{L_t^{\frac{5\alpha }{2}}L_x^{\frac{5\alpha }{2}}([-T,T]\times \Omega_n)}+ \|z_n\|_{L_t^{\frac{5\alpha }{2}}L_x^{\frac{5\alpha }{2}}([-T,T]\times \Omega_n)}^{\alpha -1} \right)\notag\\
			&\qquad\times \left(\||\nabla |^{s_c}(\chi_n w_n)\|_{L_t^{\frac{5\alpha }{2}}L_x^{\frac{30\alpha }{15\alpha -8}}([-T,T]\times \Omega_n) }+ \||\nabla |^{s_c}z_n\|_{L_t^{\frac{5\alpha }{2}}L_x^{\frac{30\alpha }{15\alpha -8}}([-T,T]\times \Omega_n) }\right)^2.\label{E1110x2}
		\end{align}
		Using the fractional product rule and (\ref{E11102}), we have 
		\begin{align}
			&\||\nabla |^{s_c}(\chi_n w_n)\|_{L_t^{\frac{5\alpha }{2}}L_x^{\frac{30\alpha }{15\alpha -8}}([-T,T]\times \Omega_n) } \lesssim \||\nabla |^{s_c}\chi_n\|_{L_x^{\frac{30\alpha }{15\alpha -8}}} \|w_n\|_{L^\infty _tL^\infty _x}+ \|\chi_n\|_{L_x^{\frac{30\alpha }{15\alpha -8}}} \||\nabla |^{s_c}w_n\|    _{L^\infty _tL^\infty _x}\notag\\
			&\lesssim  T^{\frac{2}{5\alpha }}\lambda_n^{s_c-\frac{15\alpha -8}{30\alpha }\times 3+\theta (\frac{3}{2}-s_c)}+T^{\frac{2}{5\alpha }}\lambda_n^{-\frac{15\alpha -8}{30\alpha }\times 3+\frac{3}{2}\theta }= T^{\frac{2}{5\alpha }}\lambda_n^{\frac{3(2s_c-3)}{10}+\theta (\frac{3}{2}-s_c)}+T^{\frac{2}{5\alpha }}\lambda_n^{-\frac{3}{2}+\frac{4}{5\alpha }+\frac{3}{2}\theta },\notag
		\end{align}
		which converges to  $0$ as  $n\rightarrow\infty $.   This together with (\ref{E11102}), Lemma \ref{zn} and (\ref{E1110x2}) gives  $J_4\rightarrow0$ as  $n\rightarrow\infty $. This completes the proof of (\ref{E1110x1}).

		\textbf{Step 5.}  Constructing  $v_n$ and approximation by   $C_c^{\infty }$ functions. 
		
		By (\ref{step-2}), \eqref{step-3}, and applying the stability Theorem \ref{TStability}, we conclude that for sufficiently large $n$ and $T$, there exists a global solution $v_n$ to \eqref{NLS} with initial data $v_n(0) = \phi_n$. Moreover, this solution has a finite scattering norm and satisfies 
		\begin{align}\label{approximate-2}
			\lim_{T \to \infty} \limsup_{n \to \infty} \big\|v_n(t - \lambda_n^2 t_n) - \tilde{v}_n(t)\big\|_{L_t^{\frac{5\alpha}{2}} L_x^{\frac{5\alpha}{2}}(\mathbb{R} \times \Omega)} = 0.
		\end{align}
		Thus, to prove Theorem \ref{Tembbedding1}, it suffices to establish the approximation \eqref{approximate-1}.
		
		This result follows from a standard argument; see, for example, \cite{KillipVisan2013,KillipVisanZhang2016a}. Here, we provide only a brief outline of the proof. First, by a density argument, we select $\psi_\varepsilon \in C_0^\infty(\mathbb{R} \times \mathbb{R}^3)$ such that
		\begin{equation}
			\|(-\Delta_\Omega)^{\frac{s_c}{2}}(w_\infty - \psi_\varepsilon)\|_{L_t^{\frac{5\alpha}{2}} L_x^{\frac{30\alpha}{15\alpha - 8}}(\mathbb{R} \times \mathbb{R}^3)} < \varepsilon. \label{E1110w1}
		\end{equation}
		Then, employing a change of variables and the triangle inequality, we derive
		\begin{align}
			&\hspace{3ex} \big\|(-\Delta _\Omega)^{\frac{s_c}{2}}[v_n(t - \lambda_n^2 t_n, x + x_n) - \lambda_n^{s_c - \frac{3}{2}} \psi_\varepsilon(\lambda_n^{-2}t, \lambda_n^{-1}x)]\big\|_{L_t^{\frac{5\alpha }{2}}L_x^{\frac{30\alpha }{15\alpha -8}}(\mathbb{R} \times \mathbb{R}^3)} \notag\\
			&\lesssim \big\|(-\Delta _\Omega)^{\frac{s_c}{2}}(w_\infty - \psi_\varepsilon)\big\|_{\dot{X}^{s_c}(\mathbb{R} \times \mathbb{R}^3)} + \big\|v_n(t - \lambda_n^2 t_n) - \tilde{v}_n(t)\big\|_{L_t^{\frac{5\alpha }{2}}L_x^{\frac{30\alpha }{15\alpha -8}}(\mathbb{R} \times \mathbb{R}^3)}  \label{E11132}\\
			&\hspace{3ex} + \big\|(-\Delta _\Omega)^{\frac{s_c}{2}}[\tilde{v}_n(t, x) - \lambda_n^{s_c - \frac{3}{2}} w_\infty(\lambda_n^{-2}t, \lambda_n^{-1}(x - x_n))]\big\|_{L_t^{\frac{5\alpha }{2}}L_x^{\frac{30\alpha }{15\alpha -8}}(\mathbb{R} \times \mathbb{R}^3)}. \label{E11133}
		\end{align}
		Clearly, by \eqref{approximate-2} and (\ref{E1110w1}), we have $(\ref{E11132}) \lesssim \varepsilon$. For (\ref{E11133}), note that by (\ref{perturb}), for sufficiently large $n$, $w_n$ approximates $w_\infty$ and $\chi_n(x) \rightarrow 1$. As $\widetilde{v}_n$ is constructed through $w_n$, $\chi_n$, and $z_n$,, we can use Lemma \ref{zn}, the triangle inequality, the Strichartz estimate, and Theorem \ref{convergence-flow} to show that for sufficiently large $n$, (\ref{E11133}) is also  small, which yields (\ref{approximate-1}).
	\end{proof}
	Next, we concerns the scenario when the rescaled obstacles $\Omega_n^c$
	(where $\Omega_n = \lambda_n^{- 1}  \left( \Omega - \left\{ x_n \right\} \right)$) are retreating to infinity, which corresponds to Case 3 of Theorem \ref{linear-profile}.
	\begin{theorem}[Embedding of nonlinear profiles for retreating obstacles]\label{Tembedding2}
		Let $\{t_n\}\subset\R$ be such that $t_n\equiv0$ or $|t_n|\to+\infty$. Let $\{x_n\}\subset\Omega$ and $\{\lambda_n\}\subset2^{\Bbb Z}$ satisfy that $\frac{d(x_n)}{\lambda_n}\to\infty$. Suppose that $\phi\in\dot{H}^{s_c}(\R^3)$ and  
		\begin{align*}
			\phi_n(x)=\lambda_n^{s_c-\frac{3}{2}}e^{i\lambda_n^2t_n\DeltaO}\left[(\chi_n\phi)\left(\frac{x-x_n}{\lambda_n}\right)\right]
		\end{align*} 
		with $\cn(x)=1-\Theta(\lambda_n|x|/d(x_n))$. Then for sufficiently large $n$, there exists a global solution $v_n$ to $\eqref{NLS}$ with initial data $v_n(0)=\pn$, which satisfies
		\begin{equation}
			\|v_n\|_{L_{t,x}^{\frac{5\alpha}{2}}(\R\times\Omega)}\lesssim_{\|\phi\|_{\Hsc}}1.\label{E11145}
		\end{equation}
		Furthermore, for every $\varepsilon>0$, there exist $N_\varepsilon>0$ and $\psie\in C_0^\infty(\R\times\R^3)$ such that for $n\geq N_\varepsilon$, we get
		\begin{align}\label{Embed-2}
			\norm (-\Delta _\Omega)^{\frac{s_c}{2}}[v_n(t-\lamn^2t_n,x+x_n)-\lamn^{s_c-\frac{3}{2}}\psie(\lamn^{-2}t,\lamn^{-1}x)]\norm_{ L_t^{\frac{5\alpha }{2}}L_x^{\frac{30\alpha }{15\alpha -8}}(\R\times\R^3)}<\varepsilon.
		\end{align}
	\end{theorem}
	\begin{proof}
		Similar to the proof of Theorem \ref{Tembbedding1}, we also divide the proof of Theorem \ref{Tembedding2} into five steps. For the sake of simpleness, we denote $-\Delta_{\R^3}=-\Delta$.
		
		\textbf{Step 1}.  Constructing the global solution to NLS$_{\mathbb{R}^3}$.
		
		Let $\theta = \frac{1}{100(\alpha + 1)}$. Following the proof of Theorem \ref{Tembbedding1}, if $t_n \equiv 0$, we define $w_n$ and $w_\infty$ as solutions to NLS$_{\mathbb{R}^3}$ with initial data $w_n(0) = P_{\leq d(x_n)^{\theta} \lambda_n^{-\theta}} \phi$ and $w_\infty(0) = \phi$. If $t_n \to \pm \infty$, we let $w_n$ and $w_\infty$ be solutions to NLS$_{\mathbb{R}^3}$ such that
		\begin{equation}
			\begin{cases}
				\|w_n(t) - e^{it\Delta} P_{\leq d(x_n)^{\theta} \lambda_n^{-\theta}} \phi\|_{\dot{H}_D^{s_c}(\mathbb{R}^3)} \to 0,\\
				\|w_\infty(t) - e^{it\Delta} \phi\|_{\dot{H}_D^{s_c}(\mathbb{R}^3)} \to 0.
			\end{cases}\notag
		\end{equation}
		By the assumptions in Theorem \ref{T1}, we deduce that $w_n$ and $w_\infty$ are global solutions with uniformly bounded Strichartz norms. Moreover, using arguments similar to those in the proof of Theorem \ref{Tembbedding1} and invoking Theorem \ref{TStability}, we establish that $w_n$ and $w_\infty$ satisfy the following properties:
		\begin{equation}
			\begin{cases}
				\|w_n\|_{L_t^{\frac{5\alpha }{2}}L_x^{\frac{5\alpha }{2}}(\R\times\R^3)}+\|w_\infty\|_{L_t^{\frac{5\alpha }{2}}L_x^{\frac{5\alpha }{2}}(\R\times\R^3)}\lesssim1,\\
				\||\nabla |^{s_c}(w_n-w_\infty)\|_{L_t^{\frac{5\alpha }{2}}L_x^{\frac{30\alpha }{15\alpha -8}}(\R\times\R^3)}\to0\qtq{as}t\to\pm\infty,\\
				\norm|\nabla|^{s}w_n\norm_{L_t^{\frac{5\alpha }{2}}L_x^{\frac{5\alpha }{2}}(\R\times\R^3)}\lesssim\(\frac{d(x_n)}{\lamn}\)^{\theta s},\qtq{for all }s\geq0.
			\end{cases}\label{E11141}
		\end{equation}
		\textbf{Step 2.} Constructing the approximate solution to \eqref{NLS}.
		
		Fix $T>0$ to be chosen later. We define
		\begin{align*}
			\tilde{v}_n(t,x)\stackrel{\triangle}{=}\begin{cases}
				\lamn^{s_c-\frac{3}{2}}\big(\cn w_n\big)(\lamn^{-2}t,\lamn^{-1}(x-x_n)), & |t|\leq\lamn^2T,\\
				e^{i(t-\lamn^2T)\DeltaO}\tilde{v}_n(\lamn^2T,x), &t>\lamn^2T,\\
				e^{i(t+\lamn^2T)\DeltaO}\tilde{v}_n(-\lamn^2T,x), &t<-\lamn^2T.
			\end{cases}
		\end{align*}
		Similar to (\ref{step-2}), using (\ref{E11141}), it is easy to see that 
		$\tilde{v}_n$  has finite scattering norm.

		\textbf{Step 3.} Agreement of the initial data:
		\begin{align}\label{step-3-embed2}
			\lim_{T\to\infty}\limsup_{n\to\infty}\norm e^{it\DeltaO}\big(\tilde{v}_n(\lambda_n^2 t_n)-\pn\big)\norm_{L_t^{\frac{5\alpha }{2}}L_x^{\frac{30\alpha }{15\alpha -8}}(\R\times\Omega)}=0.
		\end{align}
		By the same argument as used in the proof of Step 3 in Theorem \ref{Tembbedding1}, we can prove (\ref{step-3-embed2}) in the cases of $t_n \equiv 0$ and $|t_n| \rightarrow \infty$ by applying a change of variables, the Strichartz estimate, and using (\ref{E11141}).

		\textbf{Step 4.} Proving that $\tilde{v}_n$ is the approximate solution to \eqref{NLS} in the sense  that 
		\begin{align}\label{step4-embed2}
			\lim_{T\to\infty}\limsup_{n\to\infty}\norm (i\partial_t+\DeltaO)\tilde{v}_n-|\tilde{v}_n|^\alpha\tilde{v}_n\norm_{\dot N^{s_c}(\R\times\Omega)}=0.
		\end{align}
		
		Similar to \eqref{convergence-6.1}, it sufficies to prove
		\begin{align}\label{convergence-6.2}
			\lim_{T\to\infty}\limsup_{n\to\infty}\norm e^{i(t-\lamn^2T)\DeltaO}\vn(\lamn^2 T)\norm_{L_{t,x}^{\frac{5\alpha}{2}}(\{t>\lamn^2T\}\times\Omega)}=0.
		\end{align}
		Let $w_+$ be the asymptotic state of $w_\infty$. 
		Then by Strichartz estimates and the change of variables, we get
		\begin{align*}
			&\hspace{3ex}\norm e^{i(t-\lamn^2T)\DeltaO}\vn(\lamn^2T)\norm_{L_{t,x}^\frac{5\alpha}{2}(\{t>\lamn^2T\}\times\Omega)} =\norm e^{it\DeltaOn}(\cn w_n(T))\norm_{L_{t,x}^{\frac{5\alpha}{2}}((0,\infty)\times\Omega)}\\
			&\lesssim\norm e^{it\Delta_{\Omega_n}}[\chi_ne^{iT\Delta}w_+]\norm_{L_{t,x}^{\frac{5\alpha}{2}}((0,\infty)\times\Omega_n)}+\norm\cn[w_\infty(T)-e^{iT\Delta}w_+]\norm_{\dot H^{s_c}(\R^3)} +\norm \cn[w_\infty (T)-w_n(T)]\norm_{\Hsc(\R^3)}\\
			&\lesssim\norm\big(e^{it\Delta_{\Omega_n}}-e^{it\Delta}\big)[\cn e^{iT\Delta}w_+]\norm_{L_{t,x}^{\frac{5\alpha}{2}}((0,\infty)\times\R^3)}+\norm(1-\cn)e^{iT\Delta}w_+\norm_{\Hsc(\R^3)}\\
			&\quad +\norm e^{it\Delta}w_+\norm_{L_{t,x}^{\frac{5\alpha}{2}}((T,\infty)\times\R^3)}+\|w_\infty(T) -e^{iT\Delta}w_+\|_{\Hsc(\R^3)}+\|w_\infty(T)-w_n(T)\|_{\Hsc(\R^3)},
		\end{align*}
		which converges to zero by first letting $n\to\infty$ and then $T\to\infty $ in view of Theorem \ref{convergence-flow}, \eqref{E11141} and the monotone convergence theorem.
		
		Finally, we consider the intermediate time scale $|t|\leq \lamn^2T$.  We compute
		\begin{align*}
			[(i\partial_t+\Delta_\Omega )\tilde v_n- |\tilde v_n|^\alpha\tilde v_n](t,x)
			&=\lambda_n^{s_c-\frac72}[\Delta\chi_n w_n](\lambda_n^{-2}t,\lambda_n^{-1}(x-x_n))\\
			&\quad+2\lambda_n^{s_c-\frac72}(\nabla\chi_n\cdot\nabla w_n)(\lambda_n^{-2}t, \lambda_n^{-1}(x-x_n))\\
			&\quad+\lambda_n^{s_c-\frac72}[(\chi_n-\chi_n^{\alpha+1})|w_n|^\alpha w_n](\lambda_n^{-2}t,\lambda_n^{-1}(x-x_n)).
		\end{align*}
		Note that the cut-off function $\chi_n\sim1_{|x|\sim\frac{d(x_n)}{\lamn}}$  and $\frac{d(x_n)}{\lamn}\to\infty$ as $n\to\infty$. Therefore, we can modified the proof in step 4 of Theorem \ref{Tembedding2} with minor change to obtain (\ref{step4-embed2}).

		\textbf{Step 5.} Constructing  $v_n$ and approximation by   $C_c^{\infty }$ functions.

		By \eqref{step-3-embed2}, \eqref{step4-embed2} and invoking the stability Theorem \ref{TStability}, for sufficiently large $n$ we   obtain a global solution  $v_n$ to \eqref{NLS} with initial data $v_n(0)=\pn$. Moreover, it satisfies
		\begin{equation}
			\|v_n\|_{L_{t,x}^\frac{5\alpha}{2}(\R\times\Omega)}\lesssim1,\quad\text{and}\quad 
			\lim_{T\to\infty}\limsup_{n\to\infty}\norm v_n(t-\lamn^2t_n)-\vn(t)\norm_{\dot H_D^{s_c}(\Omega)}=0.\notag
		\end{equation}
		Finially, by the same argument as that used to derive (\ref{approximate-1}), we can obtain 	 the convergence \eqref{Embed-2} and omit the details.
		This completes the proof of Theorem \ref{Tembedding2}.
	\end{proof}

	At last, we treat the case that the obstacle expands to fill the half-space, i.e. Case 4 in Theorem \ref{linear-profile}.
	\begin{theorem}[Embedding the nonlinear profiles: the half-space case]\label{Embed3}
		Let $\{t_n\}\subset\R$ be such that $t_n\equiv0$ and $|t_n|\to\infty$. Let $\{\lamn\}\subset2^{\Bbb Z}$ and $\{x_n\}\subset\Omega$ be such that
		\begin{align*}
			\lamn\to0,\qtq{and}\frac{d(x_n)}{\lamn}\to d_\infty>0.
		\end{align*}
		Let $x_n^*\in \partial \Omega$ be such that $|x_n-x_n^*|=d(x_n)$ and  $R_n\in \operatorname{SO}(3)$ be such that
		$R_ne_3=\frac{x_n-x_n^*}{|x_n-x_n^*|}$. Finally, let $\phi\in\dot{H}_D^{s_c}(\mathbb{H})$, we define
		\begin{align*}
			\pn(x)=\lamn^{s_c-\frac{3}{2}}e^{i\lamn^2t_n\DeltaO}\phi\(\frac{R_n^{-1}(x_n-x_n^*)}{\lamn}\).
		\end{align*}
		Then for $n$ sufficiently large, there exists a global solution $v_n$ to \eqref{NLS} with initial data $v_n(0)=\pn$, which   also satisfies
		\begin{align*}
			\|v_n\|_{L_{t,x}^\frac{5\alpha}{2}(\RO)}\lesssim1.
		\end{align*} 
		Furthermore, for every $\varepsilon>0$, there exists  $N_\varepsilon\in\N$ and $\psie\in C_0^\infty(\R\times\mathbb{H})$ so that for every $n\geq N_\varepsilon$, we have
		\begin{align}\label{approximate-embed3}
			\norm (-\Delta _\Omega)^{\frac{s_c}{2}}[v_n(t-\lamn^2t_n,R_nx+x_n^*)-\lamn^{s_c-\frac{3}{2}}\psie(\lamn^{-2}t,\lamn^{-1}x)]\norm_{L_t^{\frac{5\alpha }{2}}L_x^{\frac{30\alpha }{15\alpha -8}}(\RRT)}<\varepsilon.
		\end{align}
	\end{theorem}
	\begin{proof}
		Again, we divide the proof   of this theorem  into five main steps.
		
		\textbf{Step 1}.  Construction of the global solution to NLS$_{\mathbb{R}^3}$.
		
		Let $\theta \ll 1$. When $t_n \equiv 0$, define $U_n$ and $U_\infty$ as solutions to NLS$_{\mathbb{H}}$ with initial data $U_n(0) = \phi_{\lambda_n^{-\theta}}$ and $U_\infty(0) = \phi$. If $|t_n| \to +\infty$, we set $U_n$ and $U_\infty$ to be solutions to NLS$_{\mathbb{R}^3}$ satisfying
		\begin{equation}
			\|U_n(t) - e^{it\Delta_{\mathbb{H}}} \phi_{\leq \lambda_n^{-\theta}}\|_{\dot{H}_D^{s_c}(\mathbb{H})} \to 0
			\quad \text{and} \quad
			\|U_\infty(t) - e^{it\Delta_{\mathbb{H}}} \phi\|_{\dot{H}_D^{s_c}(\mathbb{H})} \to 0, \quad \text{as} \quad t \to \pm\infty. \label{m12}
		\end{equation}
		In all cases, the assumption in Theorem \ref{T1} ensures that
		\begin{align*}
			\|U_n\|_{L_t^{\frac{5\alpha}{2}}L_x^{\frac{5\alpha}{2}}(\mathbb{R} \times \mathbb{H})} + \|U_\infty\|_{L_t^{\frac{5\alpha}{2}}L_x^{\frac{5\alpha}{2}}(\mathbb{R} \times \mathbb{H})} \lesssim 1.
		\end{align*}
		Moreover, the solution to NLS$_{\mathbb{H}}$ can be extended to a solution of NLS$_{\mathbb{R}^3}$ by reflecting across the boundary $\partial\mathbb{H}$. Using similar arguments as in the proofs of the previous embedding theorems, along with the stability theorem and persistence of regularity, we obtain
		\begin{equation}
			\begin{cases}
				\lim_{n\to\infty}\|U_n-U_\infty\|_{L_t^{\frac{5\alpha }{2}}L_x^{\frac{5\alpha }{2}}(\R\times\mathbb{H})}=0,\\
				\norm(-\Delta_{\mathbb{H}})^\frac{s}{2}U_n\norm_{L_t^\infty L_x^2(\R\times\mathbb{H})}\lesssim\lamn^{\theta(s-1)}.
			\end{cases}\label{difference-half}
		\end{equation}

		\textbf{Step 2}. Construction of the approximate solution to \eqref{NLS}.
		
		Let $\Omega_n := \lambda_n^{-1} R_n^{-1} (\Omega - \{x_n^*\})$, and let $T > 0$ be chosen later. On the intermediate time scale $|t| < \lambda_n^2 T$, we embed $U_n$ into a corresponding neighborhood in $\mathbb{H}$ by employing a boundary-straightening diffeomorphism $\Psi_n$ of size $L_n := \lambda_n^{-2\theta}$ in a neighborhood of zero in $\Omega_n$.
		To achieve this, we define a smooth function $\psi_n$ on the set $|x^\perp| \leq L_n$ such that $x^\perp \mapsto (x^\perp, -\psi_n(x^\perp))$ parametrizes $\partial\Omega_n$. Here, we write $x \in \mathbb{R}^3$ as $x = (x^\perp, x_3)$. 
		By our choice of $R_n$, the unit normal to $\partial\Omega_n$ at zero is $e_3$. Moreover, the curvatures of $\partial\Omega_n$ are $O(\lambda_n)$. Thus, $\psi_n$ satisfies the following properties:
		\begin{align}\label{psin}
			\begin{cases}
				\psi_n(0) = 0, \quad \nabla\psi_n(0) = 0, \quad |\nabla\psi_n(x^\perp)| \lesssim \lambda_n^{1-2\theta}, \\
				|\partial^{\alpha}\psi_n(x^\perp)| \lesssim \lambda_n^{|\alpha| - 1} \quad \text{for all } |\alpha| \geq 2.
			\end{cases}
		\end{align}
		We then define the map $\Psi_n: \Omega_n \cap \{|x^\perp| \leq L_n\} \to \mathbb{H}$ and a cutoff $\chi_n: \mathbb{R}^3 \to [0,1]$ as follows:
		\begin{align*}
			\Psi_n(x) := (x^\perp, x_3 + \psi_n(x^\perp)) \quad \text{and} \quad \chi_n(x) := 1 - \Theta\bigl(\tfrac{x}{L_n}\bigr).
		\end{align*}
		On the domain of $\Psi_n$, which contains $\operatorname{supp} \chi_n$, we have:
		\begin{align}\label{detpsin}
			|\det(\partial \Psi_n)| \sim 1 \quad \text{and} \quad |\partial\Psi_n| \lesssim 1.
		\end{align}

		Now, we are in position to define the approximate solution. Let $\tilde U_n:=\chi_nU_n$ and define
		\begin{align*}
			\tilde v_n(t,x):=\begin{cases} \lamn^{s_c-\frac32}[\tilde
				U_n(\lamn^{-2}t)\circ\Psi_n](\lambda_n^{-1}R_n^{-1}(x-x_n^*)), &|t|\le \lamn^2 T, \\
				e^{i(t-\lamn^2 T)\Delta_\Omega}\vn(\lambda_n^2 T,x), &t>\lamn^2 T,\\
				e^{i(t+\lamn^2 T)\Delta_\Omega}\vn(-\lambda_n^2T,x), &t<-\lamn^2 T .
			\end{cases}
		\end{align*}
		We first prove that $\tilde v_n$ has finite scattering size.  Indeed, by the Strichartz inequality, a change of variables, and \eqref{detpsin},
		\begin{align}\label{tildevn4}
			\|\tilde v_n\|_{L_{t,x}^{\frac{5\alpha}{2}}(\R\times\Omega)}
			&\lesssim \|\widetilde{U}_n\circ\Psi_n\|_{L_{t,x}^{\frac{5\alpha}{2}}(\R\times\On)}+\|\tilde U_n(\pm T)\circ\Psi_n\|_{\dot H_D^{s_c}(\On)}\notag\\
			&\lesssim \|\tilde U_n\|_{L_{t,x}^{\frac{5\alpha}{2}}(\R\times\mathbb{H})} + \|\tilde U_n(\pm T)\|_{\dot H^{s_c}_D(\mathbb{H})}\lesssim 1.
		\end{align}
		
		\textbf{Step 3}. Asymptotic agreement with the initial data:
		\begin{align}\label{step3-embed3}
			\lim_{T\to\infty}\limsup_{n\to \infty}\|(-\Delta_\Omega)^{\frac{s_c}2}e^{it\Delta_\Omega}[\tilde v_n(\lambda_n^2 t_n)-\phi_n]\|_{\isca(\R\times\Omega)}=0.
		\end{align}
		
		First, we consider the case that $t_n\equiv0$.  By Strichartz and a change of variables,
		\begin{align*}
			&\hspace{3ex}\norm  (-\DeltaO)^{\frac {s_c}2} e^{it\Delta_\Omega}(\vn(0)-\phi_n)\norm_{\isca(\R\times\Omega)} \lesssim \norm(\chi_n\phi_{\le \lambda_n^{-\theta}})\circ\Psi_n-\phi\|_{\dot H^{s_c}_D(\On)}\\
			&\lesssim \norm(-\Delta)^\frac{s_c}{2}\big((\chi_n\phi_{>\lambda_n^{-\theta}})\circ\Psi_n\big)\|_{L^2_x}+\|(-\Delta)^\frac{s_c}{2}[(\chi_n\phi)\circ\Psi_n-\chi_n\phi]\norm_{L^2_x}+\norm(-\Delta)^\frac{s_c}{2}\big((1-\chi_n)\phi\big)\norm_{L^2_x}.
		\end{align*}
		As $\lambda_n \to 0$, we have $\| \phi_{>\lambda_n^{-\theta}} \|_{\dot{H}^{s_c}} \to 0$ as $n \to \infty$. Thus, using \eqref{detpsin}, the first term converges to $0$. For the second term, since $\Psi_n(x) \to x$ in $C^1$, approximating $\phi$ by functions in $C_0^\infty(\mathbb{H})$, we conclude that the second term also converges to $0$. Finally, the last term approaches $0$ by the dominated convergence theorem and the fact that $L_n = \lambda_n^{-2\theta} \to \infty$.
		
		It remains to prove \eqref{step3-embed3} when $t_n \to +\infty$. The case $t_n \to -\infty$ can be handled similarly. 
		
		Since $T > 0$ is fixed, for sufficiently large $n$, we have $t_n > T$, so that
		\begin{align*}
			\tilde{v}_n(\lambda_n^2 t_n, x) &= e^{i(t_n - T)\lambda_n^2\Delta_\Omega}[\lambda_n^{s_c - \frac{3}{2}}(\tilde{U}_n(T) \circ \Psi_n)(\lambda_n^{-1}R_n^{-1}(x - x_n^*))].
		\end{align*}
		
		A change of variables then yields that
		\begin{align}
			&\hspace{3ex}\norm(-\Delta_\Omega)^{\frac{s_c}2} e^{it\DeltaO}(\vn(\lamn^2 t_n)-\phi_n)\norm_{\isca(\R\times\Omega)}\notag\\
			&\lesssim \norm(-\Delta_{\On})^{\frac {s_c}2}(\tilde U_n(T)\circ\Psi_n-U_\infty(T))\norm_{L^2_x}\label{nn13}\\
			&\quad+\norm(-\Delta_{\Omega_n})^{\frac {s_c}2}\big(e^{i(t-T)\Delta_{\Omega_n}}U_\infty(T)-e^{it\Delta_{\Omega_n}}\phi\big)\|_{\isca(\R\times\Omega_n)}.\label{nn12}
		\end{align}
		By the  triangle inequality,
		\begin{align}
			\eqref{nn13}
			&\lesssim\norm(-\Delta_{\Omega_n})^{\frac{s_c}2}\big((\chi_nU_\infty(T))\circ\Psi_n-U_\infty(T)\big)\|_{L^2_x} +\norm(-\Delta_{\Omega_n})^{\frac {s_c}2}\big((\chi_n(U_n(T)-U_\infty(T)))\circ\Psi_n\big)\|_{L^2_x},\notag
		\end{align}
		which converges to zero as $n\to \infty$ by  the fact that  $\Psi_n(x)\to x$ in $C^1$ and  (\ref{difference-half}). For the second term,   by the Strichartz estimate, Proposition \ref{P1}, 
		Theorem~\ref{convergence-flow}, and \eqref{m12}, we see that
		\begin{align*}
			\eqref{nn12}
			&\lesssim \norm e^{i(t-T)\Delta_{\Omega_n}}(-\Delta_{\mathbb{H}})^{\frac{s_c}2}U_\infty(T)-e^{it\Delta_{\Omega_n}}(-\Delta_{\mathbb{H}})^{\frac{s_c}2}\phi\norm_{\isca(\R\times\Omega_n)}\\
			&\quad +\norm\big((-\Delta_{\Omega_n})^{\frac {s_c}2}-(-\Delta_{\mathbb{H}})^{\frac{s_c}2}\big)U_\infty(T)\|_{L^2_x}+\norm\big((-\Delta_{\Omega_n})^{\frac {s_c}2}-(-\Delta_{\mathbb{H}})^{\frac {s_c}2}\big)\phi\|_{L^2_x}\\
			&\lesssim\norm\big(e^{i(t-T)\Delta_{\Omega_n}}-e^{i(t-T)\Delta_{\mathbb{H}}}\big)(-\Delta_{\mathbb{H}})^{\frac {s_c}2}U_\infty(T)\|_{\isca(\R\times\Omega_n)}\\
			&\quad+\norm\big(e^{it\Delta_{\Omega_n}}-e^{it\Delta_{\mathbb{H}}}\big)(-\Delta_{\mathbb{H}})^   {\frac{s_c}2}\phi\|_{\isca(\R\times\Omega_n)}\\
			&\quad+\norm e^{-iT\Delta_{\mathbb{H}}}U_\infty(T)-\phi\|_{\dot H^{s_c}_D(\mathbb{H})}+o(1),
		\end{align*}
		and that this converges to zero by first taking $n\to \infty$ and then $T\to \infty$.
		
		\textbf{Step 4}. Proving that $\vn$ is approximate solution to \eqref{NLS} in the following sense
		\begin{align}
			\label{nn14}
			\lim_{T\to\infty}\limsup_{n\to\infty}\norm(i\partial_t+\Delta_\Omega)\tilde v_n-|\tilde v_n|^\alpha\tilde v_n\norm_{\dot N^{s_c}(\R\times\Omega)}=0.
		\end{align}
		We first control the contribution of $|t|\ge \lambda_n^2T$. 
		By the same argument as that used in step 4 of Theorem \ref{Tembbedding1},    this reduces to proving
		\begin{align}\label{nn15}
			\lim_{T\to\infty}\limsup_{n\to\infty}\|e^{i(t-\lambda_n^2T)\Delta_{\Omega}}\tilde v_n(\lambda_n^2 T)\|_{\scaa(\{t>\lamn^2T\}\times\Omega)}=0.
		\end{align}

		Let $U_+$ denote the scattering state of $U_\infty$ in the forward-time direction. By the Strichartz estimate, Theorem \ref{convergence-flow}, and the monotone convergence theorem, we obtain
		\begin{align*}
			&  \norm e^{i(t-\lambda_n^2 T)\Delta_{\Omega}}\tilde{v}_n(\lambda_n^2T)\norm_{\scaa((\lambda_n^2 T, \infty) \times \Omega)} = \norm e^{i(t-T)\Delta_{\Omega_n}}(\tilde{U}_n(T) \circ \Psi_n)\|_{\scaa((T, \infty) \times \Omega_n)} \\
			&\lesssim \norm\big(e^{i(t-T)\Delta_{\Omega_n}} - e^{i(t-T)\Delta_{\mathbb{H}}}\big)(e^{iT\Delta_{\mathbb{H}}}U_+)\|_{\scaa((0, \infty) \times \Omega_n)} + \|e^{it\Delta_{\mathbb{H}}}U_+\|_{L_{t,x}^{\frac{5\alpha}{2}}((T, \infty) \times \mathbb{H})} + o(1),
		\end{align*}
		and this converges to zero by Theorem \ref{convergence-flow} and the monotone convergence theorem, by first taking $n \to \infty$ and then $T \to \infty$.
		
		Next, we consider the middle time interval $\{|t| \leq \lambda_n^2T\}$. By direct computation, we have
		\begin{align*}
			\Delta(\widetilde{U}_n \circ \Psi_n) &= (\partial_k\widetilde{U}_n \circ \Psi_n)\Delta\Psi_n^k + (\partial_{kl}\widetilde{U}_n \circ \Psi_n)\partial_j\Psi_n^l \partial_j\Psi_n^k,
		\end{align*}
		where $\Psi_n^k$ denotes the $k$th component of $\Psi_n$, and repeated indices are summed. Recall that $\Psi_n(x) = x + (0, \psi_n(\xp))$, hence we have
		\begin{align*}
			&\Delta\Psi_n^k=O(\partial^2\psi_n), \quad \partial_j\Psi_n^l=\delta_{jl}+O(\partial\psi_n), \\
			&\partial_j\Psi_n^l\partial_j\Psi_n^k=\delta_{jl}\delta_{jk}+O(\partial\psi_n)+O((\partial\psi_n)^2),
		\end{align*}
		where we use $O$ to denote a collection of similar terms.   Therefore,
		\begin{align*}
			(\partial_k\widetilde{U}_n\circ\Psi_n)\Delta\Psi_n^k&=O\bigl((\partial\widetilde{U}_n\circ\Psi_n)(\partial^2\psi_n)\bigr),\\
			(\partial_{kl}\widetilde{U}_n\circ\Psi_n)\partial_j\Psi_n^l\partial_j\Psi_n^k
			&=\Delta\widetilde{U}_n\circ\Psi_n+O\bigl(\bigl(\partial^2\widetilde{U}_n\circ\Psi_n\bigr)\bigl(\partial\psi_n+(\partial\psi_n)^2\bigr)\bigr)
		\end{align*}
		and so
		\begin{align*}
			(i\partial_t+\Delta_{\Omega_n})(\widetilde{U}_n\circ \Psi_n)-(|\widetilde{U}_n|^\alpha\widetilde{U}_n)\circ\Psi_n 
			&=[(i\partial_t+\Delta_{\mathbb{H}})\widetilde{U}_n-|\widetilde{U}_n|^4\widetilde{U}_n]\circ \Psi_n \\
			&\quad+O\bigl((\partial\widetilde{U}_n\circ\Psi_n)(\partial^2\psi_n)\bigr)+O\bigl(\bigl(\partial^2\widetilde{U}_n\circ\Psi_n\bigr)\bigl(\partial\psi_n+(\partial\psi_n)^2\bigr)\bigr).
		\end{align*}
		By a change of variables and \eqref{detpsin}, we get
		\begin{align}
			&\hspace{3ex}\norm(-\Delta_\Omega)^{\frac {s_c}2}\big((i\partial_t+\Delta_\Omega)\vn-|\tilde v_n|^\alpha\vn\big)\norm_{L_t^1L_x^2(\{|t|\le \lambda_n^2T\}\times\Omega)}\notag\\
			&=\norm(-\Delta_{\Omega_n})^{\frac{s_c}2}\big((i\partial_t+\Delta_{\Omega_n})(\tilde U_n\circ\Psi_n)-(|\widetilde{U}_n|^\alpha\tilde U_n)\circ \Psi_n\big)\norm_{L_t^1L_x^2(\{|t|\le \lambda_n^2T\}\times\Omega_n)}\notag\\
			&\lesssim \norm(-\Delta_{\Omega_n})^{\frac{s_c}2}\big(((i\partial_t+\Delta_{\mathbb{H}})\tilde U_n-|\tilde U_n|^\alpha\tilde U_n)\circ\Psi_n\big)\norm_{L_t^1L_x^2([-T,T]\times\Omega_n)}\notag\\
			&\quad+\norm(-\Delta_{\Omega_n})^{\frac {s_c}2}\big((\partial\tilde U_n\circ \Psi_n)\partial^2\psi_n)\norm_{L_t^1L_x^2([-T,T]\times\Omega_n)}\notag\\
			&\quad+\big\|(-\Delta_{\Omega_n})^{\frac {s_c}2}\big((\partial^2\tilde U_n\circ\Psi_n)\big(\partial\psi_n+(\partial\psi_n)^2\big)\big)\big\|_{L_t^1L_x^2([-T,T]\times\Omega_n)}\notag\\
			&\lesssim \|(-\Delta)^\frac{s_c}{2}\big((i\partial_t+\Delta_{\mathbb{H}})\tilde U_n -|\tilde U_n|^\alpha\tilde U_n\big)\|_{L_t^1L_x^2([-T,T]\times\mathbb{H})}\label{nn18}\\
			&\quad+\norm(-\Delta)^\frac{s_c}{2}\big((\partial \tilde U_n\circ\Psi_n)\partial^2\psi_n\big)\norm_{L_t^1L_x^2([-T,T]\times\Omega_n)}\label{nn16}\\
			&\quad+\big\|(-\Delta)^\frac{s_c}{2}\big((\partial^2 \tilde U_n\circ \Psi_n)\big(\partial\psi_n+(\partial\psi_n)^2\big)\big)\big\|_{L_t^1L_x^2([-T,T]\times\Omega_n)}\label{nn17}.
		\end{align}
		By direct computation, 	
		\begin{align}
			(i\partial_t+\Delta_{\mathbb{H}})\tilde U_n-|\tilde U_n|^\alpha\tilde U_n=(\chi_n-\chi_n^{\alpha+1})|U_n|^4U_n+2\nabla\chi_n\cdot\nabla w_n+\Delta\chi_n w_n.\label{E11143}
		\end{align}
		For fixed  $T>0$, using fractional product rule, \eqref{difference-half}, \eqref{psin},  \eqref{detpsin} and  $\lambda_n\rightarrow0$,  it is easy to see that (\ref{nn16}), (\ref{nn17}) and the  $\dot N^{s_c}(\mathbb{R} \times \mathbb{H} )$ norm of the last two terms in  (\ref{E11143}) converges to $ 0  $ as  $n\rightarrow\infty $.  
		
		Therefore, the proof of (\ref{nn14}) reduces to show that the  $\dot N^{s_c}(\mathbb{R} \times \mathbb{H} )$ norm of the first term in  (\ref{E11143}) converges to $ 0  $ as  $n\rightarrow\infty $.   To this end, we estimate 
		\begin{align*}
			& \|(\chi_n-\chi_n^{\alpha +1})|U_n|^{\alpha +1}U_n\|_{\dot N^{s_c}([-T,T]\times \mathbb{H} )} \notag\\
			&\lesssim \|(\chi_n-\chi_n^{\alpha +1})|U_n|^{\alpha +1}|\nabla |^{s_c}U_n\|_{L_t^{\frac{5\alpha }{2(\alpha +1)}}L_x^{\frac{30\alpha }{27\alpha -8}}([-T,T]\times \mathbb{H} )}   + \||U_n|^{\alpha +1}|\nabla |^{s_c}\chi_n\|_{L_t^{\frac{5\alpha }{2(\alpha +1)}}L_x^{\frac{30\alpha }{27\alpha -8}}([-T,T]\times \mathbb{H} )} \notag \\
			&\lesssim   \|U_n1_{|x|\sim L_n}\|_{L_{t,x}^{\frac{5\alpha }{2}}}^\alpha  \||\nabla |^{s_c}U_n\|_{L_t^{\frac{5}{2}}L_x^{\frac{30\alpha }{15\alpha -8}}}+ \|U_n1_{|x|\sim L_n}\|_{L_{t,x}^{\frac{5\alpha }{2}}}^\alpha  \||\nabla |^{s_c}U_n\|_{L_t^{\frac{5\alpha }{2}}L_x^{\frac{30\alpha }{15\alpha -8}}} \||\nabla |^{s_c}\chi_n\|_{L_x^{\frac{3}{s_c}}}     \\
			&\lesssim\|1_{|x|\sim L_n}U_\infty\|_{\scaa}^\alpha+\|U_\infty-U_n\|^\alpha _{L_{t,x}^\frac{5\alpha}{2}}\to0\quad\text{as}\quad n\rightarrow\infty .
		\end{align*}
		This completes the proof of (\ref{nn14}).

		\textbf{Step 5}. Constructing $v_n$ and approximating by   compactly supported functions. 
		
		Similar to Theorem \ref{Tembbedding1} and Theorem \ref{Tembedding2}, using (\ref{tildevn4}), (\ref{step3-embed3}),  (\ref{nn14})  and the stability theorem \ref{TStability}, for $ n $  large enough we obtain a global solution  $v_n$ to (\ref{NLS}) with initial data  $v_n(0)=\phi_n$, which satisfies (\ref{E11145}). Moreover, the similar argument used in Theorem \ref{Tembbedding1} and Theorem \ref{Tembedding2} also gives (\ref{Embed-2}) and we  omit the details. 
	\end{proof}
	
	\section{Reduction to Almost Periodic Solutions}\label{S5}
	The goal of this section is to establish Theorem \ref{TReduction}. The proof relys on demonstrating a Palais-Smale condition (Proposition \ref{Pps}) for minimizing sequences of blowup solutions to \eqref{NLS}, which leads to the conclusion that the failure of Theorem \ref{T1} would imply the existence of minimal counterexamples possessing the properties outlined in Theorem \ref{TReduction}.  
	We adopt the framework described in \cite[Section 3]{KillipVisan2010AJM}. This general methodology has become standard in related contexts; see, for instance, \cite{KenigMerle2006,KenigMerle2010,KillipVisan2013,TaoVisanZhang2008FM} for analogous results in different settings. Consequently, we will highlight the main steps, providing detailed discussions only when specific challenges arise in our scenario.
	
	Throughout this section, we use the notation 
	\begin{equation}
		S_I(u) := \int_I \int_{\Omega} |u(t, x)|^{\frac{5\alpha}{2}} \, dx \, dt.  
	\end{equation}  
	
	Assume Theorem \ref{T1} fails for some $s_c \in [\frac{1}{2}, \frac{3}{2})$. We define the function $L: [0, \infty) \to [0, \infty)$ as
	\[
	L(E) := \sup\{S_I(u) : u : I \times \Omega \to \mathbb{C} \text{ solving } \eqref{NLS} \text{ with } \sup_{t \in I} \|u(t)\|^2_{\dot{H}^{s_c}_D(\Omega)} \leq E\}.
	\]
	It is noteworthy that $L$ is non-decreasing, and Theorem \ref{TLWP} provides the bound
	\begin{equation}
		L(E) \lesssim E^{\frac{5\alpha}{4}} \quad \text{for sufficiently small } E.\label{E10252}
	\end{equation}
	This implies the existence of a unique critical value $E_c \in (0, \infty]$ such that $L(E) < \infty$ for $E < E_c$ and $L(E) = \infty$ for $E > E_c$. The failure of Theorem \ref{T1} implies $0 < E_c < \infty$.
	
	A pivotal component of the proof of Theorem \ref{TReduction} is verifying a Palais-Smale condition. Once the following proposition is established, the derivation of Theorem \ref{TReduction} proceeds along standard lines (see \cite{KillipVisan2010AJM}).

	\begin{proposition}[Palais--Smale condition modulo symmetries]\label{Pps}
		Let $u_n : I_n \times \Omega \to \mathbb{C}$ be a sequence of solutions to (\ref{NLS}) such that
		\[
		\limsup_{n \to \infty} \sup_{t \in I_n} \|u_n(t)\|_{\dot{H}_D^{s_c}(\Omega)}^2 = E_c,
		\]
		and suppose $t_n \in I_n$ are such that
		\begin{equation}
			\lim_{n \to \infty} S_{[t_n, \sup I_n]}(u_n) = \lim_{n \to \infty} S_{[\inf I_n, t_n]}(u_n) = \infty.  \label{4.2}
		\end{equation}
		Then the sequence $u_n(t_n)$ has a subsequence that converges strongly in $\dot{H}_D^{s_c}(\Omega)$.
	\end{proposition}
	We now outline the proof of this proposition, following the argument presented in \cite{KillipVisan2010AJM}. As in that framework, the key components are the linear profile decomposition (Theorem \ref{linear-profile} in our setting) and the stability result (Theorem \ref{TStability}).
	
	To begin, we translate the sequence so that each $t_n = 0$, and apply the linear profile decomposition (Theorem \ref{linear-profile}) to express
	\begin{equation}
		u_n(0) = \sum_{j=1}^J \phi_n^j + w_n^J, \label{E10251}
	\end{equation}
	with the properties specified in Theorem \ref{linear-profile}.
	
	Next, we proceed to construct the nonlinear profiles. For $j$ conforming to Case 1, we invoke Theorem \ref{TLWP} and define $v^j : I^j \times \mathbb{R}^d \to \mathbb{C}$ as the maximal-lifespan solution to \eqref{NLS} satisfying
	\[
	\begin{cases}
		v^j(0) := \phi^j & \text{if } t_n^j \equiv 0, \\
		v^j \text{ scatters to } \phi^j \text{ as } t \to \pm \infty & \text{if } t_n^j \to \pm \infty.
	\end{cases}
	\]
	We then define the nonlinear profiles $v_n^j(t,x) := v^j(t + t_n^j (\lambda_n^j)^2, x)$. By construction, $v_n^j$ is also a solution to \eqref{NLS} on the time interval $I_n^j := I^j - \{t_n^j (\lambda_n^j)^2\}$. For sufficiently large $n$, we have $0 \in I_n^j$ and 
	\begin{equation}
		\lim_{n \to \infty} \|v_n^j(0) - \phi_n^j\|_{\dot{H}^{s_c}_D(\Omega)} = 0. \notag
	\end{equation}
	
	For $j$ conforming to Cases 2, 3, or 4, we utilize the nonlinear embedding theorems from the previous section to construct the nonlinear profiles. Specifically, let $v_n^j$ be the global solutions to \eqref{NLS} constructed in Theorems \ref{Tembbedding1}, \ref{Tembedding2}, or \ref{Embed3}, as applicable.
	
	The $\dot{H}^{s_c}_D(\Omega)$ decoupling of the profiles $\phi^j$ in \eqref{profile-2}, along with the definition of $E_c$, ensures that for sufficiently large $j$, the profiles $v_n^j$ are global and scatter. Specifically, for $j \ge J_0$, the profiles fall within the small-data regime. 
	
	To complete the argument, we aim to show that there exists some $1 \leq j_0 < J_0$ such that
	\begin{equation}
		\limsup_{n \to \infty} S_{[0, \sup I^{j_0}_n)}(v_n^{j_0}) = \infty. \label{E10261}
	\end{equation}

	When a 'bad' nonlinear profile similar to (\ref{E10261}) emerges, it can be shown that such a profile is unique. This conclusion follows by adapting the approach in \cite[Lemma 3.3]{KillipVisan2010AJM}, demonstrating that $\dot{H}^{s_c}_D(\Omega)$ decoupling holds over time. Utilizing the 'critical' nature of $E_c$, we can exclude the existence of multiple profiles. Consequently, the decomposition (\ref{E10251}) has a single profile (i.e., $J^* = 1$), allowing us to express
	\begin{equation}
		u_n(0) = \phi_n + w_n \quad \text{with} \quad \lim_{n \to \infty} \|w_n\|_{\dot{H}^1_D(\Omega)} = 0. \label{7.7}
	\end{equation}
	If $\phi_n$ belongs to Cases 2, 3, or 4, then by Theorems \ref{Tembbedding1}, \ref{Tembedding2}, or \ref{Embed3}, there exist global solutions $v_n$ to (\ref{NLS}) with initial data $v_n(0) = \phi_n$ that satisfy a uniform space-time bound. Using Theorem \ref{TStability}, this bound extends to $u_n$ for sufficiently large $n$, leading to a contradiction with (\ref{4.2}). Thus, $\phi_n$ must align with Case 1, and (\ref{7.7}) simplifies to
	\begin{equation}
		u_n(0) = e^{it_n \lambda_n^2 \Delta_\Omega} \phi + w_n \quad \text{with} \quad \lim_{n \to \infty} \|w_n\|_{\dot{H}^{s_c}_D(\Omega)} = 0\notag
	\end{equation}
	where $t_n \equiv 0$ or $t_n \to \pm \infty$. If $t_n \equiv 0$, the desired compactness follows. Therefore, it remains to rule out the case where $t_n \to \pm \infty$.
	
	Assume $t_n \to \infty$ (the case $t_n \to -\infty$ is analogous). Here, the Strichartz inequality combined with the monotone convergence theorem gives
	\[
	S_{\geq 0}\left(e^{it\Delta_\Omega} u_n(0)\right) = S_{\geq 0}\left(e^{i(t + t_n \lambda_n^2) \Delta_\Omega} \phi + e^{it \Delta_\Omega} w_n\right) \longrightarrow 0 \quad \text{as} \quad n \to \infty.
	\]
	By small data theory, this result implies $S_{\geq 0}(u_n) \to 0$, contradicting (\ref{4.2}).
	
	To establish the existence of at least one bad profile, suppose, for contradiction, that no such profiles exist. In this case, the inequality
	\begin{equation}
		\sum_{j \geq 1} S_{[0,\infty)}(v_n^j) \lesssim_ {E_c} 1. \label{E10253}
	\end{equation}
	holds. For sufficiently large $n$, the solution lies within the small-data regime. Applying small-data local well-posedness, we obtain $S_{[0,\infty)}(v_n^j) \lesssim \|v_n^j\|_{\dot{H}^{s_c}_D(\Omega)}$, and the decoupling property (\ref{profile-2}) ensures that the tail is bounded by $E_c$.

	Next, we use \eqref{E10253} and the stability result (Theorem \ref{TStability}) to constrain the scattering size of $u_n$, contradicting \eqref{4.2}.
	To proceed, we define the approximations
	\begin{equation}
		u_n^J(t) = \sum_{j=1}^{J} v_n^j(t) + e^{it\Delta} w_n^J.
	\end{equation}
	By the construction of $v_n^j$, it is easy to verify that 
	\begin{equation}
		\limsup_{n \to \infty} \| u_n(0) - u_n^J(0) \|_{\dot{H}^{s_c}_D(\Omega)} = 0. \label{4.6}
	\end{equation}
	Furthermore, we claim:
	\begin{equation}
		\lim_{J \to \infty} \limsup_{n \to \infty} S_{[0,\infty)}(u_n^J) \lesssim_ {E_c} 1. \label{E10254}
	\end{equation}
	To justify \eqref{E10254}, observe that by \eqref{profile-1} and \eqref{E10253}, it suffices to prove
	\begin{equation}
		\lim_{J \to \infty} \limsup_{n \to \infty} \left| S_{[0,\infty)} \left( \sum_{j=1}^{J} v_n^j \right) - \sum_{j=1}^{J} S_{[0,\infty)}(v_n^j) \right| = 0. \label{4.8}
	\end{equation}
	Note that 
	\[
	\left|\left| \sum_{j=1}^{J} v_n^j \right|^{\frac{5\alpha }{2}} - \sum_{j=1}^{J} \left| v_n^j \right|^{\frac{5\alpha }{2}} \right|\lesssim_J \sum_{j \neq k} \left| v_n^j \right|^{\frac{5\alpha }{2}-1} \left| v_n^k \right|.
	\]
	It follows from  H\"older's inequality that 
	\begin{equation}
		\text{LHS} \eqref{4.8} \lesssim_J \sum_{j \neq k} \left\| v_n^j \right\|^{\frac{5\alpha }{2}-2}_{L_t^{\frac{5\alpha }{2}} L_x^{\frac{5\alpha }{2}} ([0,\infty) \times \Omega)} \left\| v_n^j v_n^k \right\|_{L_t^{\frac{5\alpha }{4}} L_x^{\frac{5\alpha }{4}} ([0,\infty) \times \Omega)}.
		\label{E1026s1}
	\end{equation}
	
	Following Keraani's argument \cite[Lemma 2.7]{Keraani2001}, with $j \neq k$, we can first use (\ref{approximate-1}), (\ref{Embed-2}) and (\ref{approximate-embed3}) to approximate $v^j$ and $v^k$ by compactly supported functions in $\mathbb{R} \times \mathbb{R}^3$, then using the asymptotic orthogonality \eqref{profile-4} to demonstrate
	\begin{equation}
		\limsup_{n \to \infty} \left(\|v_n^j v_n^k\|_{L_t^{\frac{5\alpha }{4}} L_x^{\frac{5\alpha }{4}} ([0,\infty) \times \Omega)}+ \|v_n^j(-\Delta _\Omega)^{\frac{s_c}{2}}v_n^k\|_{L_t^{\frac{5\alpha }{4}}L_x^{\frac{15\alpha }{15\alpha -8}}([0,\infty )\times \Omega)} \right) = 0.\label{E11161}
	\end{equation}
	Combining this with  \eqref{E1026s1}, we see that \eqref{4.8} (and hence \eqref{E10254}) is valid.
	
	With \eqref{4.6} and \eqref{E10254} in place, proving that $u_n^J$ asymptotically solves (\ref{NLS}) reduces to showing:
	\begin{equation}
		\lim_{J \to \infty} \limsup_{n \to \infty} \| (i \partial_t + \Delta) u_n^J - |u_n^J|^\alpha  u_n^J\|_{\dot N^{s_c}([0,\infty)\times \Omega)} = 0.\label{E11221}
	\end{equation}
	Once this is  established, we can apply the stability Theorem \ref{TStability}  to bound the scattering size of $u_n$, contradicting (\ref{4.2}).  This completes the proof of proposition \ref{Pps}.

	It sufficies to prove (\ref{E11221}), 	which  relys   on demonstrating:
	\begin{lemma}[Decoupling of nonlinear profiles]\label{LDecoupling of nonlinear profiles}Let  $F(u)=|u|^{\alpha }u$. Then 
		\begin{equation}
			\lim_{J \to \infty} \limsup_{n \to \infty} \|  F ( \sum_{j=1}^{J} v_n^j  ) - \sum_{j=1}^{J} F(v_n^j)   \|_{\dot N^{s_c}([0,\infty)\times \Omega)} = 0,\label{E11151}
		\end{equation}
		\begin{equation}
			\lim_{J \to \infty} \limsup_{n \to \infty} \|   F(u_n^J - e^{it \Delta} w_n^J) - F(u_n^J)   \|_{\dot N^{s_c}([0,\infty)\times \Omega)} = 0.\label{E11152}
		\end{equation}
	\end{lemma}
	In the energy-critical setting, i.e., $s_c = 1$, one can instead use the pointwise estimate
	\begin{equation}
		\left| \nabla \left( F\left( \sum_{j=1}^J v_n^j \right) - \sum_{j=1}^J F(v_n^j) \right) \right| \lesssim_J \sum_{j \neq k} |\nabla v_n^j| |v_n^k|^\alpha \label{E11153}
	\end{equation}
	and (\ref{E11161}) to prove (\ref{E11151}) and (\ref{E11152}); the key is to exhibit terms that all contain some $v_n^j$ paired against some $v_n^k$ for $j \neq k$. In the case $s_c = 0$, there are also pointwise estimates similar to (\ref{E11153}). However, when $s_c \neq 0, 1$, a new difficulty arises as the nonlocal operator $|\nabla|^{s_c}$ does not respect pointwise estimates in the spirit of (\ref{E11153}). 
	
	To address this issue, in the subcritical case ($s_c < 1$), Murphy \cite{Murphy2014} employs the Littlewood-Paley square function estimates, which hold for all $s > 0$ and $1 < r < \infty$:
	\begin{equation}
		\|(\sum N^{2s}|f_N(x)|^{2})^{1/2}\|_{L_x^r(\mathbb{R}^d)} \sim \|(\sum N^{2s}|f_{>N}(x)|^{2})^{1/2}\|_{L_x^r(\mathbb{R}^d)} \sim \||\nabla|^{s}f\|_{L_x^r(\mathbb{R}^d)}, \label{Eequvilat}
	\end{equation}
	to work at the level of individual frequencies. By utilizing maximal function and vector maximal function estimates, he adapts the standard arguments to this context.
	In the supercritical case ($s_c > 1$), Killip and Visan \cite{KillipVisan2010} employed the following equivalence (see, e.g., \cite{Strichartz1967JMM}):
	\begin{equation}
		\||\nabla|^{s}f\|_{L_x^q} \sim \|\mathcal{D}_s(f)\|_{L_x^q}, 
	\end{equation}
	where the operator $\mathcal{D}_s$ is defined as
	\[
	\mathcal{D}_s(f)(x) := \left( \int_0^\infty \left| \int_{|y| < 1} \frac{|f(x + ry) - f(x)|}{r^{1 + 2s}} \, dy \right|^2 dr \right)^{1/2},
	\]
	which behaves like $|\nabla|^s$ under symmetries. They then used the following pointwise inequality:
	\[
	\mathcal{D}_s\big(w \cdot [F'(u + v) - F'(u)]\big) \lesssim \mathcal{D}_s(w)|v|^\alpha  +  M(|w|)M(|v|)  \big[\mathcal{D}_s (u + v) + \mathcal{D}_s(u)\big],
	\]
	where $M$ denotes the Hardy-Littlewood maximal function. By combining this inequality with various permutations of the techniques discussed above, they adapted the standard arguments to this context.
	
	In this paper, we follow the arguments in \cite{Murphy2014,KillipVisan2010} and sketch the  proof of Lemma \ref{LDecoupling of nonlinear profiles}.
	\begin{proof}[\textbf{Proof of (\ref{E11151})}]
		By induction, it suffices to treat the case of two summands. To simplify notation, we write $f = v_n^j$ and $g = v_n^k$ for some $j \neq k$, and are left to show
		\begin{equation}
			\| |f+g|^\alpha (f+g) - |f|^\alpha f - |g|^\alpha g \|_{\dot N^{s_c}([0, \infty) \times \Omega)} \to 0 \quad \text{as } n \to \infty. \notag
		\end{equation}
		We first rewrite
		\[
		|f+g|^\alpha(f+g) - |f|^\alpha f - |g|^\alpha g 
		= \big( |f+g|^\alpha- |f|^\alpha \big)f + \big( |f+g|^\alpha - |g|^\alpha \big)g.
		\]
		By symmetry, it suffices to treat
		\begin{equation}
			\| \big( |f+g|^\alpha - |f|^\alpha \big)f \|_{\dot N^{s_c}([0, \infty) \times \Omega)}. \label{E11173}
		\end{equation}
		We then utilize Theorem \ref{TEquivalence} and the Littlewood-Paley square function estimates (\ref{Eequvilat}) to reduce (\ref{E11173}) to handling
		\begin{equation}
			\left\| \left( \sum_N \big||\nabla|^{s_c} P_N \big( \big(|f+g|^\alpha  - |f|^\alpha  \big)f \big)\big|^2 \right)^{\frac{1}{2}} \right\|_{L_t^{\frac{5\alpha}{2(\alpha +1)}} L_x^{\frac{30\alpha}{27\alpha - 8}}}. \label{E11177}
		\end{equation}
		Then the key step is to perform a decomposition such that all resulting terms to estimate have $f$ paired against $g$ inside a single integrand. For such terms, the asymptotic orthogonality (\ref{E11161}) can be used. 
		Following the arguments in \cite{Murphy2014}, we decompose (\ref{E11177}) into terms such that each term contains pairings of $f$ and $g$. For instance, one of the terms is
		\begin{equation}
			\|(\sum_N |N^{s_c}f_{>N}M(g|f|^{\alpha-1})|^2)^{1/2}\|_{L_t^{\frac{5\alpha}{2(\alpha +1)}} L_x^{\frac{30\alpha}{27\alpha - 8}}}. \label{E11178}
		\end{equation}
		Using H\"older's inequality and maximal function estimates, this term can be controlled as
		\begin{equation}
			\|(\sum_N |N^{s_c}f_{>N}|^2)^{1/2}\|_{L_t^{\frac{5\alpha }{2}}L_x^{\frac{30\alpha }{15\alpha -8}}} \||g||f|^{\alpha -1}\|_{L_{t,x}^{\frac{d+2}{2}}}. \notag
		\end{equation}
		By (\ref{Eequvilat}), the first term is bounded by $\||\nabla|^{s_c}v_n^j\|_{L_{t,x}^{\frac{2(d+2)}{d}}}$, which is further bounded by the construction of $v_n^j$. The second term vanishes as $n \to \infty$ due to the asymptotic orthogonality of parameters (\ref{E11161}). The other terms similar to (\ref{E11178}) can be handled similarly, thereby completing the proof of (\ref{E11151}).
	\end{proof}
	
	\begin{proof}[\textbf{Proof of (\ref{E11152})}]
		For this term, we will rely on (\ref{profile-1}) rather than (\ref{E11161}). The reasoning closely resembles the proof of (\ref{E11151}). Using the same approach as in the proof of (\ref{E11161}), we derive terms that involve either $e^{it\Delta}w_n^J$ or $|\nabla|^{s_c}e^{it\Delta}w_n^J$. 
		The terms where $e^{it\Delta}w_n^J$ appears without derivatives are relatively simple to address, as we can directly apply (\ref{profile-1}). On the other hand, the terms containing $|\nabla|^{s_c} e^{it\Delta} w_n^J$ demand a more detailed analysis. Specifically, we first use the local smoothing estimate from Corollary \ref{CLocalsmoothing}, followed by an application of (\ref{profile-1}) to demonstrate that these terms vanish as $n \to \infty$.
	\end{proof}

	We now apply the Palais-Smale condition in Proposition \ref{Pps} to prove Theorem \ref{TReduction}. 
	\begin{proof}[\textbf{Proof of Theorem \ref{TReduction}.}]
		Assume Theorem \ref{T1} is false. Using a standard argument (see, e.g., \cite[Theorem 5.2]{KillipVisan2013}), we can apply the Palais-Smale condition to construct a minimal counterexample $u:I \times \Omega \to \mathbb{C}$ satisfying  
		\begin{equation}
			S_{\ge0}(u) = S_{\le 0}(u) = \infty, \label{E11171}
		\end{equation}
		with its orbit $\{u(t): t \in I\}$ being precompact in $\dot{H}^{s_c}_D(\Omega)$. Additionally, since the modulation parameter $N(t) \equiv 1$ is compact, it follows that the maximal lifespan interval is $I = \mathbb{R}$ (see, e.g., \cite[Corollary 5.19]{KillipVisan2013}).  
		
		Next, we establish the lower bound in (\ref{E}) by contradiction. Suppose there exist sequences $R_n \to \infty$ and $\{t_n\} \subset \mathbb{R}$ such that  
		\[
		\int_{\Omega \cap \{|x| \leq R_n\}} |u(t_n, x)|^{\frac{3}{2}\alpha} \, dx \to 0.
		\]
		Passing to a subsequence, we obtain $u(t_n) \to \phi$ in $\dot{H}^{s_c}_D(\Omega)$ for some non-zero $\phi \in \dot{H}^{s_c}_D(\Omega)$. If $\phi$ were zero, the solution $u$ would have a $\dot{H}^{s_c}_D(\Omega)$ norm below the small data threshold, contradicting (\ref{E11171}). By Sobolev embedding, $u(t_n) \to \phi$ in $L^{\frac{3}{2}\alpha}$, and since $R_n \to \infty$,
		\begin{equation}
			\int_\Omega |\phi(x)|^{\frac{3}{2}\alpha} \, dx 
			= \lim_{n \to \infty} \int_{\Omega \cap \{|x| \leq R_n\}} |\phi(x)|^{\frac{3}{2}\alpha} \, dx 
			= \lim_{n \to \infty} \int_{\Omega \cap \{|x| \leq R_n\}} |u(t_n, x)|^{\frac{3}{2}\alpha} \, dx = 0.\notag
		\end{equation}
		This contradicts the fact that $\phi \neq 0$, thus completing the proof of Theorem \ref{TReduction}.
	\end{proof}

	\section{The cases $1<s_c<\frac{3}{2}$ and  $s_c=\frac{1}{2}$.}\label{S6}
	In this section, we rule out  the existence of almost periodic solutions as in   Theorem \ref{TReduction} in the cases  $1<s_c<3/2$ and  $s_c=\frac{1}{2}$.  The proof is based on    a space-localized Morawetz estimate as in the work of Bourgain \cite{Bourgain1999} on the radial energy-critical NLS.  See also \cite{Grillakis2000,Tao2005}.
	
	\begin{lemma}[Morawetz inequality]\label{L1091}
		Let  $1<s_c<\frac{3}{2}$  and let $u$ be a solution to  (\ref{NLS}) on the time interval $I$. Then for any $A \geq 1$ with $A |I|^{1/2} \geq \text{diam}(\Omega^c)$ we have
		\begin{equation}
			\int_I \int_{|x| \leq A |I|^{1/2}, x \in \Omega}  \frac{|u(t,x)|^{\alpha +2}}{|x|}\, dx \, dt \lesssim  (A|I|^{\frac{1}{2}})^{2s_c-1}\{ \|u\|_{L_t^{\infty }\dot H^{s_c}_D(I\times \Omega)} ^2+\|u\|_{L_t^{\infty }\dot H^{s_c}_D(I\times \Omega)} ^{\alpha +2}\}.\label{E1092}
		\end{equation}
	\end{lemma}
	\begin{proof}
		Let $\phi(x)$ be a smooth, radial bump function such that $\phi(x) = 1$ for $|x| \leq 1$ and $\phi(x) = 0$ for $|x| > 2$. We set $R \geq \text{diam}(\Omega^c)$ and denote $a(x) := |x| \phi\left(\frac{x}{R}\right)$. Then, for $|x| \leq R$ we have
		\begin{equation}
			\partial_j \partial_k a(x) \text{ is positive definite}, \quad \nabla a(x) = \frac{x}{|x|}, \quad \text{and} \quad \Delta \Delta a(x) < 0.
			\label{E1094}
		\end{equation}
		For $|x| > R$, we have the following rough bounds:
		\begin{equation}
			|\partial_k a(x)| \lesssim 1, \quad |\partial_j \partial_k a(x)| \lesssim \frac{1}{R}, \quad \text{and} \quad |\Delta \Delta a(x)| \lesssim \frac{1}{R^3}.\label{E1095}
		\end{equation}
		By the direct calculus, we have the following  identity
		\begin{equation}
			2\partial_t \text{Im}(\bar{u} \partial_j u) = - 4 \partial_k \text{Re}(\partial_k u \partial_j \bar{u}) +   \partial_j \Delta (|u|^2) - \frac{2\alpha }{\alpha +2} \partial_j (|u|^{\alpha +2}).\label{E1096}
		\end{equation}
		Multiplying  by $\partial_j a$ in both sides and integrating over $\Omega$, we obtain
		\begin{align}
			&2\partial_t \text{Im} \int_{\Omega} \bar{u} \partial_j u \partial_j a \, dx \notag\\
			&= -4 \text{Re} \int_{\Omega} \partial_k (\partial_k u \partial_j \bar{u}) \partial_j a \, dx+  \int_{\Omega} \partial_j \Delta (|u|^2) \partial_j a \, dx- \frac{2\alpha }{\alpha +2} \int_{\Omega} \partial_j (|u|^{\alpha +2}) \partial_j a \, dx.\label{E1091}
		\end{align}

		Now, we give the  upper bound of the  LHS of \eqref{E1091} which follows immediately from H\"older and the Sobolev embedding: 
		\begin{equation}
			2\left| \text{Im} \int_{\Omega} \bar{u} \partial_j u \partial_j a \, dx \right|\lesssim  \|u\|_{L_x^{\frac{6}{3-2s_c}}(\Omega)} \|\nabla u\|_{L_x^{\frac{6}{5-2s_c}}(\Omega)} \|\nabla a\|_{L_x^{\frac{3}{2s_c-1}}(\Omega)}\lesssim R^{2s_c-1} \|u\|^2_{\dot H_D^{s_c}(\Omega)} .\label{E1093}   
		\end{equation}
		
		Next, we find a lower bound on RHS of (\ref{E1091}).  By using the Gauss theorem, we 	get
		\begin{align*}
			-4 \text{Re} \int_{\Omega} \partial_k (\partial_k u \partial_j \bar{u}) \partial_j a \, dx
			&=4 \text{Re} \int_{\partial \Omega} \partial_k u \partial_{j}a\partial_j \bar{u} \vec{n}_k \, d\sigma(x) +4 \text{Re} \int_{\Omega} \partial_k u \partial_j \bar{u} \partial_k\partial_j a \, dx, 
		\end{align*}
		where $\vec{n}$ denotes the outer normal vector to $\Omega^c$. We write
		$\partial_j \bar{u}\vec{n}_j = \nabla \bar{u} \cdot \vec{n} = \bar{u}_n$ and  $\partial _jan_j=\nabla a\cdot \vec{n}=a_n$. 
		Moreover, from the Dirichlet boundary condition, the tangential derivative of $u$ vanishes on the boundary:
		\[
		\nabla u = (\nabla u \cdot \vec{n}) \vec{n} = u_n \vec{n}, \quad \text{and} \quad \partial_j \overline{u}_j\partial_j a = u_n a_n.
		\]
		Combining the analysis above and   (\ref{E1094}),  we obtain
		\begin{align}
			-4 \text{Re} \int_{\Omega} \partial_k (\partial_k u \partial_j \bar{u}) \partial_j a \, dx
			&\geq 4 \int_{\partial \Omega} a_n |u_n|^2 \, d\sigma(x) + 4 \int_{|x| \geq R} \partial_k u \partial_j \bar{u} \partial_k\partial_j a \, dx \\
			&\ge 4 \int_{\partial \Omega} a_n |u_n|^2 \, d\sigma(x) - \|\Delta a\|_{L_x^{\frac{3}{2(s_c-1)}}( \{x:|x|>R\})} \|\nabla u\|^2_{L_x^{\frac{6}{5-2s_c}}(\Omega)}\\  
			&\geq 4 \int_{\partial \Omega} a_n |u_n|^2 \, d\sigma(x) -  CR^{2s_c-3} \|u\|^2_{\dot H_D^{s_c}(\Omega)}.\label{E10111}
		\end{align}
		The second term on RHS of (\ref{E1091}) can be estimated by a similar argument:
		\begin{align}
			\int_{\Omega} \partial_j \Delta (|u|^2) \partial_j a \, dx 
			&=  \int_{\Omega} \partial_j ( \Delta (|u|^2) \partial_j a) dx 
			-  \int_{\Omega} \Delta (|u|^2) \Delta a \, dx\notag \\
			&= - \int_{\partial \Omega} \Delta (|u|^2) \partial_j a \vec{n}_j\, d\sigma(x) 
			-  \int_{\Omega} |u|^2 \Delta \Delta a \, dx \notag\\
			&= -2\int_{\partial \Omega} |\nabla u|^2 a_n \, d\sigma(x) 
			-   \int_{ |x|\le R} |u|^{2}\Delta ^2a\, dx  -\int _{|x|\ge R}|u|^{2}\Delta ^2a\, dx\notag\\
			&\geq -2 \int_{\partial \Omega} |u_n|^2 a_n \, d\sigma(x) 
			-   \|u\|_{L_x^{\frac{6}{3-2s_c}}(   \Omega)}^2 \|\Delta ^2a\|_{L_x^{\frac{3}{2s_c}}(  \{x:|x|>R\})}\notag\\
			&\ge   -2 \int_{\partial \Omega} |u_n|^2 a_n \, d\sigma(x) -CR^{2s_c-3} \|u\|_{\dot H_D^{s_c}(\Omega)}^2.\label{E10112}
		\end{align}

		Finally, it  remains to estimate the third term on RHS of (\ref{E1091}). By using (\ref{E1094})  and (\ref{E1095}),
		\begin{align}
			-&\frac{2\alpha }{\alpha +2} \int_{\Omega} \partial_j (|u|^{\alpha +2}) \partial_j a \, dx 
			= \frac{2\alpha }{\alpha +2} \int_{\Omega} |u|^{\alpha +2} \Delta a \, dx \notag\\
			&= \frac{4\alpha }{\alpha +2} \int_{|x| \leq R} \frac{|u|^{\alpha +2}}{|x|} \, dx -  \frac{4\alpha }{\alpha +2} \int _{\Omega \cap \{x:|x|>R\}}\Delta a |u|^{\alpha +2}dx\notag\\
			&\ge  \frac{4\alpha }{\alpha +2} \int_{|x| \leq R} \frac{|u|^{\alpha +2}}{|x|} \, dx - \|\Delta a\|_{L_x^{\frac{3}{2(s_c-1)}}( \{x:|x|>R\})} \| u\|_{L_x^{\frac{6}{3-2s_c}}(\Omega)}^{\alpha +2}\notag\\
			&\ge  \frac{4\alpha }{\alpha +2} \int_{|x| \leq R} \frac{|u|^{\alpha +2}}{|x|} \, dx -CR^{2s_c-3} \|u\|_{\dot H_D^{s_c}(\Omega)}^{\alpha +2}.\notag
		\end{align}
		Putting these together and using the fact that $a_n \geq 0$ on $\partial \Omega$, we have
		\begin{equation}
			\quad \text{LHS(\ref{E1091})} \gtrsim \int_{|x| \leq R} \frac{|u|^{\alpha +2}}{|x|} \, dx - R^{2s_c-3}  ( \|u\|_{\dot H_D^{s_c}(\Omega)}^2+ \|u\|_{\dot H_D^{s_c}(\Omega)}^{\alpha +2}  ).\label{E1097}
		\end{equation}
		
		Integrating (\ref{E1091}) over $I$ and using the upper bound for the LHD of  (\ref{E1091}) and the lower bound for the RHS of (\ref{E1091}), we finally deduce
		\[
		\int_I \int_{|x| \leq R, x \in \Omega} \frac{|u|^{\alpha +2}}{|x|} \, dx \, dt \lesssim R^{2s_c-1} \|u\|_{L_t^{\infty }\dot H^{s_c}_D(I\times \Omega)} ^2+  R^{2s_c-3}|I|\left\{\|u\|_{L_t^{\infty }\dot H^{s_c}_D(I\times \Omega)} ^2+\|u\|_{L_t^{\infty }\dot H^{s_c}_D(I\times \Omega)} ^{\alpha +2} \right\}.
		\]
		Taking $R = A |I|^{1/2}$ yields (\ref{E1092}). This completes the proof of the lemma.
	\end{proof}
	In the proof of Lemma \ref{L1091}, by taking $R \rightarrow +\infty$ and using the same argument as in \cite[Lemma 2.3]{CKSTT}  	to control the upper bound of the Morawetz action, we can obtain the following non-spatially localized Lin-Strauss Morawetz inequality.
	\begin{lemma}[Morawetz inequality]\label{L10911}
		Let  $s_c=\frac{1}{2}$  and let $u$ be a solution to  (\ref{NLS}) on the time interval $I$. Then  we have
		\begin{equation}
			\int_I \int_{ \Omega}  \frac{|u(t,x)|^{\alpha +2}}{|x|}\, dx \, dt  \lesssim   \|u\|_{L^\infty _t\dot H^{\frac{1}{2}}_D(\Omega)}^2 .\label{E109}
		\end{equation}
	\end{lemma}
	We now use Lemma \ref{L1091} and Lemma \ref{L10911} to prove the following. 
	\begin{theorem}\label{T1091}
		There are no almost periodic solutions   $u$ to (\ref{NLS}) as in Theorem  \ref{TReduction} with  $1<s_c<3/2$ or  $s_c=\frac{1}{2}$. 
	\end{theorem}
	
	\begin{proof}
		By contradiction, we suppose that   there exists a minimal  blow-up solution $u$ whose the orbit is precompact in Hilbert space $\dot{H}_D^{s_c}(\Omega)$ and  satisfies
		(\ref{E}).   
		We first claim that 
		\begin{equation}
			\int _{\Omega\cap \{x:|x|\le R\}}|u(t,x)|^{\alpha +2}dx\gtrsim _R1, \qquad \text{uniformly for } t\in \mathbb{R}. \label{E12221}
		\end{equation}
		Indeed, by (\ref{E}) and H?lder's inequality, it suffices to show that 
		\begin{equation}
			\int _{\Omega\cap \{x:|x|\le R\}}|u(t,x)|^{2}dx\gtrsim 1, \qquad \text{uniformly for } t\in \mathbb{R}. \label{E12222}
		\end{equation}
		Applying H?lder's inequality, Sobolev embedding, Bernstein and (\ref{E10101}), we obtain
		\begin{align}
			&\left|\int _{\Omega\cap \{x:|x|\le R\}}|u(t,x)|^2-|P^{\Omega}_{<C(\eta)}u(t,x)|^2dx\right|\notag\\
			&\lesssim R^{s_c} \|P^{\Omega}_{>C(\eta)}u(t,x)\|_{L_x^{2}(\Omega)}  \|u(t,x)\| _{L_x^{\frac{6}{3-2s_c}}(\Omega)} 
			\lesssim \eta R^{s_c} \|u\|_{L_t^{\infty }\dot H^{s_c}_D(I\times\Omega)}. \label{E1222x1} 
		\end{align}
		On the other hand, using H?lder's inequality and Sobolev embedding again, we have
		\begin{align}
			&\left|\int _{\Omega\cap \{x:|x|\le R\}} |u(t,x)|^{\frac{3\alpha }{2}}-|P^{\Omega}_{<C(\eta)}u(t,x)|^{\frac{3\alpha }{2}}dx\right|\notag\\
			&\lesssim \|P^{\Omega}_{>C(\eta)}u\|_{L_t^{\infty }L_x^{\frac{6}{3-2s_c}}(I\times \Omega)} \|u\|_{L_t^{\infty }L_x^{\frac{6}{3-2s_c}}(I\times \Omega)}^{\frac{3\alpha }{2}-1} \lesssim \eta \|u\|_{L_t^{\infty }\dot H^{s_c}_D(I\times \Omega)}^{\frac{3\alpha }{2}-1} . \notag  
		\end{align}
		Combining the  above  inequality with (\ref{E}), and further applying H?lder's inequality and Bernstein's estimate, we deduce
		\begin{align}
			1\lesssim& \int _{\Omega\cap \{x:|x|\le R\}}|P^{\Omega}_{<C(\eta)}u|^{\frac{3\alpha }{2}}dx\notag\\
			& \lesssim \|P^{\Omega}_{<C(\eta)}u\| _{L_x^{\infty }(\Omega)}^{\frac{3\alpha }{2}-2}\int _{\Omega\cap \{x:|x|\le R\}}|P^{\Omega}_{<C(\eta)}u|^{2}dx \lesssim \int _{\Omega\cap \{x:|x|\le R\}}|P^{\Omega}_{<C(\eta)}u|^{2}dx. \notag
		\end{align}
		By combining the   above  inequality with (\ref{E1222x1}) and choosing $\eta > 0$ sufficiently small, we establish (\ref{E12222}), thereby completing the proof of (\ref{E12221}).
		
		We now proceed to prove Theorem \ref{T1091}.
		Integrating  (\ref{E12221}) over  $I$ with length $|I| \geq 1$, we obtain
		\[
		|I| \lesssim _R \int_I  \int_{\Omega \cap \{|x| \leq R\}} \frac{|u(t, x)|^{\alpha +2}}{|x|} \, dx \, dt \lesssim_ R \int_I  \int_{\Omega \cap \{|x| \leq R |I|^{1/2} \}} \frac{|u(t, x)|^{\alpha +2}}{|x|} \, dx \, dt.
		\]
		On the other hand, for $R |I|^{1/2} \geq 1$, the Morawetz inequality (Lemma \ref{L1091}  and Lemma \ref{L10911}) yields that
		\[
		|I|\lesssim _R\int_I  \int_{\Omega \cap \{|x| \leq R |I|^{1/2} \}} \frac{|u(t, x)|^{\alpha +2}}{|x|} \, dx \, dt \lesssim |I| ^{s_c-\frac{1}{2}},
		\]
		with the implicit constant depending only on $ R $ and $   \|u\|_{L_t^{\infty }\dot H^{s_c}_D(I\times \Omega)} $.
		Choosing $I$ sufficiently large  depending  on $R$ and $\|u\|_{L_t^{\infty }\dot H^{s_c}_D(I\times \Omega)} $, we get a contradiction, which completes the proof of Theorem  \ref{T1091}.
	\end{proof}
	\section{The case $\frac{1}{2}<s_c<1$.}\label{S1/2-1}
	In this section, we rule out the almost periodic solutions in the case $1/2<s_c<1$. The key tool employed is a long-time Strichartz estimate, which will be established in subsection \ref{s1/2-1,1}. In subsection \ref{s1/2-1,2}, we derive a frequency-localized Lin-Strauss Morawetz inequality, and in subsection \ref{s1/2-1,3}, we use it to exclude almost periodic solutions.
	
	\subsection{Long-time Strichartz estimates}\label{s1/2-1,1}
	
	In this subsection, we prove a long-time Strichartz estimate tailored to the Lin-Strauss Morawetz inequality. This type of estimate, initially developed by Dodson \cite{Dodson2012}, has demonstrated its effectiveness in excluding minimal counterexamples. For references, see \cite{KillipVisan2012,Visan2012} for the energy-critical case, \cite{Murphy2014,Murphy2015,Yu2021} for the inter-critical regime, and \cite{Dodson2017,LuZheng2017,MiaoMurphyZheng2014,Murphy2015} for the super-critical setting. 
	In this work, we introduce for the first time a long-time Strichartz estimate applicable to the exterior domain Schr\"odinger equation. This estimate plays a pivotal role in subsection \ref{s1/2-1,2}, where it is used to handle the error terms arising from frequency projection in the Lin-Strauss Morawetz inequality.

	Throughout this subsection \ref{S1/2-1}, we use the following notation: 
	\begin{equation}
		A_I(N):= \|(-\Delta _\Omega)^{\frac{s_c}{2}}P^\Omega _{\le N}u\|_{L^2_tL_x^6(I\times \Omega)}.\notag 
	\end{equation}
	The main result of this subsection is the following.
	\begin{proposition}[Long time Strichartz estimate]\label{PLT2}
		Let $u :\mathbb{R} \times \Omega\to \mathbb{C}$ be an almost periodic solution  to (\ref{NLS}) with $1/2 < s_c < 1$.  Then for any $N > 0$, we have
		\begin{equation}
			A_I(N) \lesssim_u 1 + N^{s_c-1/2} |I|^{1/2}. \label{E10106}
		\end{equation}
		Moreover, for any $\varepsilon > 0$, there exists $N_0 = N_0(\varepsilon) > 0$ so that for any $N \leq N_0$,
		\begin{equation}
			A_I(N) \lesssim_u \varepsilon (1 + N^{s_c - 1/2} |I|^{1/2}).  \label{E10107}
		\end{equation}
	\end{proposition}
	We prove Proposition  \ref{PLT2} by induction. The inductive step relies on the following.
	\begin{lemma}\label{LLT2}
		Let  $\eta>0$,  $u$  and  $I$ be as above. For  any  $N>0$, we have 
		\begin{equation}
			\|(-\Delta _\Omega)^{\frac{s_c}{2}}P^\Omega_{\le N}F(u)\|_{L^2_tL^{6/5}_x(I\times \Omega)}\lesssim _u C_\eta  \|u_{\le N/\eta}\|_{L^\infty _t\dot H^{s_c}_{D}(\Omega)}N^{s_c-1/2}|I|^{1/2}+\sum _{M>N/\eta}\left(\frac{N}{M}\right)^{s_c}A_I(M).\notag  
		\end{equation}
	\end{lemma}
	\begin{proof}
		We fix $0 < \eta < 1$ and decompose the nonlinearity as follows:
		\[
		F(u) = F(u_{\leq N/\eta}) + \left[F(u) - F(u_{\leq N/\eta}) \right].
		\]
		Using the fractional chain rule (\ref{E12133}), H\"older and Sobolev embedding, we estimate
		\begin{align}
			&		\|(-\Delta _\Omega)^{\frac{s_c}{2}}P^\Omega _{\le N}F(u_{\leq N/\eta})\|_{L^2_t L^{6/5}_x ( I\times \Omega)}\notag
			\\
			&\lesssim \|   (-\Delta _\Omega)^{\frac{s_c}{2}}u_{\le N/\eta}\|^{\alpha }_{L^\infty_t L^2_x ( I\times \Omega)} \|(-\Delta _\Omega)^{\frac{s_c}{2}} u_{\leq N/\eta} \|_{L^2_t L^6_x (I\times \Omega)} \notag\\
			& \lesssim \|   (-\Delta _\Omega)^{\frac{s_c}{2}}P^\Omega _{\le c(\eta)}u_{\le N/\eta}\|^{\alpha }_{L^\infty_t L^2_x ( I\times \Omega)} \|(-\Delta _\Omega)^{\frac{s_c}{2}} u_{\leq N/\eta} \|_{L^2_t L^6_x (I\times \Omega)} \label{E1011x3}\\
			& \quad + \|(-\Delta _\Omega)^{\frac{s_c}{2}} P^\Omega _{>c(\eta)} u_{\leq N/\eta}\|^{\alpha }_{L^\infty_t L^2_x (I\times \Omega)} \|(-\Delta _\Omega)^{\frac{s_c}{2}} u_{\leq N/\eta} \|_{L^2_t L^6_x (I\times \Omega)}.\label{E1011x4}
		\end{align}
		For the first term, we use (\ref{E10101}) to estimate
		\begin{align}
			(\ref{E1011x3})& \lesssim \eta ^\alpha   \|(-\Delta _\Omega)^{\frac{s_c}{2}} u_{\leq N/\eta} \|_{L^2_t L^6_x (I\times \Omega)}\lesssim \eta ^{s_c}A_I(N/\eta). \label{E10102}
		\end{align}
		For the next term, we note that we only need to consider the case $c(\eta)  < N/\eta$, in which case we have $1 \lesssim_\eta   N^{s_c - 1/2}$. Then by Bernstein and  Lemma \ref{Lspace-time bound},   we have 
		\begin{align}
			(\ref{E1011x4})
			& \lesssim_u C_\eta N^{s_c - 1/2} \|(-\Delta _\Omega)^{\frac{s_c}{2}} u_{\leq N/\eta} \|_{L^\infty _t L^2_x (I\times \Omega)}  \lesssim_u C_\eta N^{s_c - 1/2}  \|u_{\leq N/\eta}\|_{L^\infty_t \dot{H}^{s_c}_D (I\times \Omega)}.\label{E10103}
		\end{align}
		Combining  (\ref{E10102}) and  (\ref{E10103}), we obtain 
		\begin{align}
			\|(-\Delta _\Omega)^{\frac{s_c}{2}}P^\Omega _{\le N} F(u_{\leq N/\eta}) \|_{L^2_t L^{6/5}_x(I\times \Omega)} & \lesssim_u \eta^{s_c} A_I(N/\eta) + C_\eta \|u_{\leq N/\eta}\|_{L^\infty_t \dot{H}^{s_c}_D (I\times \Omega)}  N^{s_c-1/2}|I|^{\frac{1}{2}}.\label{E10161}
		\end{align}
		
		Next, we use Bernstein, H\"older and  Sobolev embedding to estimate
		\begin{align*}
			&\|(-\Delta _\Omega)^{\frac{s_c}{2}}P^\Omega _{\le N}\left( F(u) - F(u_{\leq N/\eta}) \right) \|_{L^2_t L^{6/5}_x(I\times \Omega)} \notag\\
			&\lesssim N^{s_c} \|u\|^{\alpha }_{L^\infty_t  L_x^{3\alpha /2} (I\times \Omega)} \sum_{M > N/\eta} \|u_M\|_{L^2_t L^6_x(I\times \Omega)} \\
			& \lesssim  N^{s_c} \|(-\Delta _\Omega)^{\frac{s_c}{2}}u\|_{L^\infty _tL^2_x(I\times \Omega)} ^{\alpha } \sum _{M>N/\eta}M^{-s_c} \|(-\Delta _\Omega)^{\frac{s_c}{2}}u_M\|_{L^2_tL^6_x(I\times \Omega)}\notag\\
			& \lesssim_u \sum_{M > N/\eta} \left( \frac{N}{M} \right)^{s_c} A_I(M),
		\end{align*}
		which together with (\ref{E10161}) yields the desired estimate in  Lemma \ref{LLT2}. 
	\end{proof}
	
	We now turn to the proof of Proposition \ref{PLT2}. 
	\begin{proof}[\textbf{Proof of Proposition \ref{PLT2}}]
		We use induction to establish the result. For the base case, consider $N > 1$. By Lemma \ref{Lspace-time bound}, we have
		\begin{equation}
			A_I(N) \lesssim_u 1 + |I|^{1/2} \leq C_u \left[ 1 + N^{s_c-1/2} |I|^{1/2} \right]. \label{E10104}
		\end{equation}
		This inequality remains valid if $C_u$ is replaced with any larger constant.
		
		Next, assume that (\ref{E10104}) holds for frequencies $\geq 2N$. We will apply Lemma \ref{LLT2} to verify that it holds at frequency $N$. 
		
		Using Strichartz estimates and Lemma \ref{LLT2}, we get
		\begin{align}
			A_I(N) &\leq \tilde{C_u} \left[ \inf_{t \in I} \| u_{\leq N}(t) \|_{\dot H^{s_c}_D(\Omega)} + C_\eta \| u_{\leq N/\eta} \|_{L_t^\infty \dot{H}_D^{s_c}(I \times \Omega)} N^{s_c-1/2} |I|^{1/2} \right. \notag\\
			&\qquad + \left. \sum_{M \geq N/\eta} \left( \frac{N}{M} \right)^{s_c} A_I(M) \right] \notag\\
			&\leq \tilde{C_u} \left[ 1 + C_\eta N^{s_c-1/2} |I|^{1/2} + \sum_{M \geq N/\eta} \left( \frac{N}{M} \right)^{s_c} A_I(M) \right]. \label{E10105}
		\end{align}
		
		By the inductive hypothesis, we have
		\begin{align}
			A_I(N) &\leq \tilde{C_u} \left[ 1 + C_\eta N^{s_c-1/2} |I|^{1/2} + \sum_{M \geq N/\eta} \left( \frac{N}{M} \right)^{s_c} \big(C_u + C_u M^{s_c-1/2} |I|^{1/2}\big) \right] \notag\\
			&\leq \tilde{C_u} \left[ 1 + C_\eta N^{s_c-1/2} |I|^{1/2} \right] + C_u \tilde{C_u} \left[ \eta^{s_c} + \eta^{1/2} N^{s_c-1/2} |I|^{1/2} \right]. \notag
		\end{align}
		
		Choosing $\eta$ sufficiently small depending on $\tilde{C_u}$, we deduce
		\[
		A_I(N) \leq \tilde{C_u} (1 + C_\eta N^{s_c-1/2} |I|^{1/2}) + \frac{1}{2} C_u (1 + N^{s_c-1/2} |I|^{1/2}).
		\]
		
		Finally, by taking $C_u$ large enough to ensure $C_u \geq 2(1 + C_\eta)\tilde{C_u}$, we conclude
		\[
		A_I(N) \leq C_u (1 + N^{s_c-1/2} |I|^{1/2}).
		\]
		This completes the proof of (\ref{E10106}) via induction.
		
		With (\ref{E10106}) established, we proceed to prove (\ref{E10107}) by building on (\ref{E10105}).  In fact, for any $\eta > 0$ sufficiently small, the almost periodicity condition implies
		\[
		\lim_{N \to 0} \|u_{\leq N/\eta}\|_{L^\infty_t \dot H^{s_c}_D(I \times \Omega)} = 0.
		\]
		This completes the proof of Proposition \ref{PLT2}.
	\end{proof}
	
	\subsection{A Frequency-Localized Lin-Strauss Morawetz Inequality}\label{s1/2-1,2}
	\begin{proposition}[Frequency-localized Morawetz]\label{PMorawetz}
		Let $u : \mathbb{R} \times \Omega \to \mathbb{C}$ be an almost periodic solution to (\ref{NLS}) with $1/2 < s_c < 1$. For any $\eta > 0$, there exists $N_0 = N_0(\eta) \in (0,1)$ such that for $N < N_0$ and $I \subset \mathbb{R}$, the following estimate holds:
		\[
		\int_{I } \int_\Omega \frac{|u_{>N}(t,x)|^{\alpha + 2}}{|x|} \, dx \, dt \lesssim_u \eta \left( N^{1 - 2s_c} + |I| \right).
		\]
	\end{proposition}
	
	To establish Proposition \ref{PMorawetz}, we begin by truncating the low frequencies of the solution and focusing on $u_{>N}$ for a given $N > 0$. Since $u_{>N}$ is not an exact solution to (\ref{NLS}), additional error terms introduced by this frequency projection must be estimated. To handle these terms, we need the following lemma.
	
	\begin{lemma}[High and low frequency control]\label{LHLC}
		Let $u$ and $I$ be as above. With  all spacetime norms over $I \times \Omega$, the following hold:
		\begin{itemize}
			\item[(i)] Let $(q,r)$ be an admissible pair. For any $N > 0$ and $0 \le s < 1/2$, we have
			\begin{equation}
				\|(-\Delta_\Omega)^{\frac{s}{2}}u_{>N}\|_{L_t^q L_x^r} \lesssim_u N^{s - s_c}(1 + N^{2s_c - 1} |I|)^{1/q}. \label{E1010x1}
			\end{equation}
			\item[(ii)] For any $\eta > 0$ and $0 < s < s_c$, there exists $N_1 = N_1(\eta)$ such that for $N < N_1$, we have
			\begin{equation}
				\|(-\Delta_\Omega)^{\frac{s}{2}}u_{>N}\|_{L_t^\infty L_x^2} \lesssim_u \eta N^{s - s_c}.\label{E1010x2}
			\end{equation}
			\item[(iii)] For any $\eta > 0$, there exists $N_2 = N_2(\eta)$ such that for $N < N_2$, we have
			\begin{equation}
				\|(-\Delta_\Omega)^{\frac{s_c}{2}} u_{\leq N}\|_{L_t^2 L_x^6} \lesssim_u \eta(1 + N^{2s_c - 1} |I|)^{1/2}.\label{E1010x3}
			\end{equation}
		\end{itemize}
	\end{lemma}
	
	\begin{proof}
		For (\ref{E1010x1}), we first apply interpolation, (\ref{Ebound}), and (\ref{E10106}) to get
		\begin{equation}
			\|(-\Delta_\Omega)^{\frac{s_c}{2}}u_{<N}\|_{L_t^q L_x^r} \lesssim \|u\|_{L_t^\infty \dot H^{s_c}_D(\Omega)}^{1-\frac{2}{q}} \|(-\Delta_\Omega)^{\frac{s_c}{2}}u_{<N}\|_{L_t^2 L_x^6}^{\frac{2}{q}} \lesssim (1 + N^{2s_c - 1}|I|)^{1/q}.\notag
		\end{equation}
		Using Bernstein's inequality, we deduce
		\begin{align}
			\|u_{>N}\|_{L_t^q L_x^r} &\lesssim \sum_{M > N} M^{s-s_c} \|(-\Delta_\Omega)^{\frac{s_c}{2}}u_M\|_{L_t^q L_x^r} \notag\\
			&\lesssim_u \sum_{M > N} M^{s-s_c}(1 + M^{2s_c - 1}|I|)^{1/q} \lesssim_u N^{s-s_c}(1 + N^{2s_c - 1}|I|)^{1/q}.\notag
		\end{align}
		
		For (\ref{E1010x2}), we utilize the almost periodicity property (\ref{E10101}) and Bernstein's inequality to obtain
		\begin{align}
			\|(-\Delta_\Omega)^{\frac{s}{2}}u_{>N}\|_{L_t^\infty L_x^2} &\lesssim c(\eta)^{s-s_c} \|(-\Delta_\Omega)^{\frac{s_c}{2}}u_{>c(\eta)}\|_{L_t^\infty L_x^2} + N^{s-s_c} \|(-\Delta_\Omega)^{\frac{s_c}{2}}u_{N \leq \cdot \leq c(\eta)}\|_{L_t^\infty L_x^2} \notag\\
			&\lesssim_u c(\eta)^{s-s_c} + \eta N^{s-s_c}.\notag
		\end{align}
		Choosing $N_1(\eta) = \eta^{1/(s_c - s)}c(\eta)$ yields the desired result (\ref{E1010x2}).
		
		The final inequality (\ref{E1010x3}) directly corresponds to (\ref{E10107}). 
	\end{proof}
	
	We now proceed to prove Proposition \ref{PMorawetz}.

	\begin{proof}[\textbf{Proof of Proposition \ref{PMorawetz}}]
		Throughout the proof, we take the norms over $I \times \Omega$ and thus we omit it for short.
		
		Assume that $0 < \eta \ll 1$ and taking
		\[
		N < \min\{N_1(\eta), \eta^2 N_2(\eta^{2s_c})\},
		\]
		where $N_1$ and $N_2$ are as in Lemma \ref{LHLC}. As a consequence, (\ref{E1010x1}) implies that
		\begin{equation}
			\|u_{>N/\eta^{2}}\|_{L_t^2 L_x^6} \lesssim_u \eta N^{-s_c}(1 + N^{2s_c - 1} |I|)^{1/2}.  \label{E1010x4}
		\end{equation}
		Moreover, using the fact $N/\eta^2 < N_2(\eta^{2s_c})$ we can apply (\ref{E1010x3})  to show 
		\begin{equation}
			\|(-\Delta _\Omega)^{\frac{s_c}{2}} u_{\leq N/\eta^2}\|_{L_t^2 L_x^6} \lesssim_u \eta(1 + N^{2s_c - 1} |I|)^{1/2}.\label{E1010x5}
		\end{equation}
		
		Then we define the Morawetz action by the following
		\[
		\text{Mor}(t) = 2 \text{Im} \int_{\Omega} \frac{x}{|x|} \cdot \nabla u_{>N}(t,x) \overline{u_{>N}(t,x)} dx.
		\]
		Since $(i \partial_t + \Delta_\Omega) u_{>N} = P^{\Omega}_{>N}(F(u))$, we have, by the direct calculus,
		\begin{align}
			\partial_{t}\text{Mor}(t) &=-4\text{Re} \int _\Omega \partial_{k}(\partial_{k}u_{>N}\partial_{j}\overline{u}_{>N})\frac{x_j}{|x|}dx+\int _\Omega\partial_{j}\Delta (|u|^{2})\frac{x_j}{|x|}dx + 2\int _{\Omega}\frac{x}{|x|} \cdot \{P^{\Omega}_{>N}(F(u)), u_{>N}\}_p dx,\notag
		\end{align}
		where the momentum bracket $\{ \cdot , \cdot \}_p$ is defined by $\{ f, g \}_p := \text{Re}( f \nabla \overline{g} - g \nabla \overline{f} )$.  Using the same argument  as that used to derive (\ref{E10111}), (\ref{E10112}), and the fact that  $\partial_{jk}|x|$ is positive definite, we have 
		\begin{equation}
			-4\text{Re} \int _\Omega \partial_{k}(\partial_{k}u_{>N}\partial_{j}\overline{u}_{>N})\frac{x_j}{|x|}dx+\int _\Omega\partial_{j}\Delta (|u|^{2})\frac{x_j}{|x|}dx\ge 2\int _{\partial \Omega}|\nabla u_{>N}\cdot \vec{n}|^2\frac{x}{|x|}\cdot  \vec{n}d\sigma (x)>0,\notag
		\end{equation}
		where  $\vec{n}$ denotes the outer normal to  $\Omega^c$. Hence
		\[
		\partial_t \text{Mor}(t) \ge2 \int_{\Omega} \frac{x}{|x|} \cdot \{P^{\Omega}_{>N}(F(u)), u_{>N}\}_p dx.
		\]
		The fundamental theorem of calculus yields that
		\begin{equation}
			\int_{I \times \Omega} \frac{x}{|x|} \cdot \{P^{\Omega}_{>N}(F(u)), u_{>N}\}_p dx \lesssim \|\text{Mor}(t)\|_{L_t^\infty(I)}.  \label{E1010x6}
		\end{equation}
		By direct computation, one hase $$ \{F(u), u\}_p = - \frac{\alpha }{\alpha +2} \nabla (|u|^{\alpha +2}).$$ Thus, we can write the truncate momentum bracket to the following
		\begin{align}
			&\{P^{\Omega}_{>N}(F(u)), u_{>N}\}_p \notag\\
			&= \{F(u), u\}_p - \{F(u_{\leq N}), u_{\leq N}\}_p
			- \{F(u) - F(u_{\leq N}), u_{\leq N}\}_p - \{P^{\Omega}_{\leq N}(F(u)), u_{>N}\}_p\notag\\
			&= - \frac{\alpha }{\alpha +2} \nabla (|u|^{\alpha +2} - |u_{\leq N}|^{\alpha +2}) - \{F(u) - F(u_{\leq N}), u_{\leq N}\}_p 
			- \{P^{\Omega}_{\leq N}(F(u)), u_{>N}\}_p\notag\\
			&:= I + II + III.\notag
		\end{align}
		Integrating by parts,   $I$  contributes to left-hand side of  $\eqref{E1010x6}$ a multiple of:
		\begin{equation}
			\int_{I }\int _\Omega \frac{|u_{>N}(t,x)|^{\alpha +2}}{|x|} dx dt\label{E1010x7}
		\end{equation}
		and to the right-hand side of  (\ref{E1010x6}) a multiple of:
		\begin{align}
			\left\| \frac{1}{|x|} (u_{\leq N})^{\alpha +1} u_{>N} \right\|_{L_{t,x}^1} +\left\| \frac{1}{|x|} u_{\leq N} (u_{>N})^{\alpha +1} \right\|_{L_{t,x}^1 }:=I_1+I_2.  \label{E1010w2}
		\end{align}
		For the second term $II$, we   use the  divergence theorem to  deduce  that  
		\begin{align}
			II\lesssim \left\| \frac{1}{|x|} u_{\leq N} [F(u) - F(u_{\leq N})] \right\|_{L_{t,x}^1}    + \left\| \nabla u_{\leq N} [F(u) - F(u_{\leq N})] \right\|_{L_{t,x}^1}:=II_1+II_2.\label{E1010x10}
		\end{align}
		Finally, for the last  term $III$, we use  integrating by parts when the derivative acts on $u_{>N}$ to  get that 
		\begin{align}
			III\lesssim 	\left\| \frac{1}{|x|} u_{>N} P^{\Omega}_{\leq N} (F(u)) \right\|_{L_{t,x}^1}  + \left\| u_{>N} \nabla P^{\Omega}_{\leq N} (F(u)) \right\|_{L_{t,x}^1}.\label{E1010x12}
		\end{align}
		Thus, building upon  (\ref{E1010x6}), we conclude that  it remains to show:
		\begin{equation}
			\| \text{Mor} \|_{L_t^\infty (I)} \lesssim_u \eta N^{1 - 2s_c},\label{E1010x13}
		\end{equation}
		and that the error terms (\ref{E1010w2}), (\ref{E1010x10}) and  (\ref{E1010x12}) are  controlled by $\eta (N^{1 - 2s_c} +|I|)$.
		
		First, we begin to  prove  (\ref{E1010x13}). Making use of the  Bernstein estimate,  the Hardy inequality and  (\ref{E1010x2}) to obtain
		\begin{align}
			&	\| \text{Mor} \|_{L_t^\infty(I)} \lesssim \| |\nabla |^{-1/2} \nabla u_{>N} \|_{L_t^\infty L_x^2} \left\| |\nabla |^{1/2} \left( \frac{x}{|x|} u_{>N} \right) \right\|_{L_t^\infty L_x^2}\notag\\
			&\lesssim_u \| |\nabla |^{1/2} u_{>N} \|_{L_t^\infty L_x^2}^2 \lesssim  \|(-\Delta _\Omega)^{\frac{1}{4}}u_{>N}\|_{L^\infty _tL^2_x}^{2} \lesssim_u \eta N^{1 - 2s_c}.\notag
		\end{align}

		We next turn to give the estimate of the error terms (\ref{E1010w2})--(\ref{E1010x12}).   To estimate  $I_1$,  we first note that by interpolation, (\ref{Ebound}) and (\ref{E1010x3}), 
		\begin{equation}
			\|(-\Delta _\Omega)^{\frac{s_c}{2}}u_{\le N}\|^{\alpha +1}_{L_t^{2(\alpha +1)}L_x^{\frac{6(\alpha +1)}{3\alpha +1}}}\lesssim  \|(-\Delta _\Omega)^{\frac{s_c}{2}}u\|_{L_t^\infty L_x^2}^{\alpha } \|(-\Delta _\Omega)^{\frac{s_c}{2}}u_{\le N}\|_{L_t^2L_x^6}\lesssim    \eta (1+N^{2s_c-1}|I|)^{\frac{1}{2}}.\notag
		\end{equation}
		It then follows from H\"older, Hardy's inequality,  Sobolev embedding, Bernstein and (\ref{E1010x1}) that 
		\begin{align}
			&I_1
			\lesssim \left\| |x|^{-\frac{1}{\alpha +1}}u_{\leq N}  \right\|_{L_t^{2(\alpha +1)} L_x^{\frac{6(\alpha +1)}{5}}} ^{\alpha +1}\| u_{>N} \|_{L_t^2 L_x^6} 	 \lesssim \| (-\Delta _\Omega)^{\frac{1}{2(\alpha +1)}} u_{\leq N}  \|_{L_t^{2(\alpha +1)} L_x^{\frac{6(\alpha +1)}{5}}}^{\alpha +1} \| u_{>N} \|_{L_t^2 L_x^6}\notag\\
			&\lesssim   \|(-\Delta _\Omega)^{\frac{3\alpha -2}{4(\alpha +1)}}u_{\le N}\|_{L_t^{2(\alpha +1)}L_x^{\frac{6(\alpha +1)}{3\alpha +1}}}^{\alpha +1} \|u_{>N}\|_{L^2_tL^6_x}    \lesssim  N^{1-s_c} \|(-\Delta _\Omega)^{\frac{s_c}{2}}u_{\le N}\|^{\alpha +1}_{L_t^{2(\alpha +1)}L_x^{\frac{6(\alpha +1)}{3\alpha +1}}} \| u_{>N} \|_{L_t^2 L_x^6}\notag\\
			&  \lesssim_u \eta N^{1 - 2s_c} (1 + N^{2s_c - 1} |I|).\notag
		\end{align}
		For  $I_2$, we divide it into two cases.  If $|u_{\leq N}| \lesssim  |u_{>N}|$, then  it is   the term  $I_1$, which we have already treated.   Thus,   it suffices to consider the case $|u_{\leq N}| \ll |u_{>N}|$, which can  be absorbed   into the left-hand side of (\ref{E1010x6}), if we can show
		\begin{equation}
			\left\| \frac{1}{|x|} |u_{>N}|^{\alpha +2} \right\|_{L_{t,x}^1} < \infty.  \label{E1010w3}
		\end{equation}
		To prove (\ref{E1010w3}), we apply  Hardy's inequality  and  Sobolev embedding to obtain 
		\begin{align}
			\left\| \frac{1}{|x|} |u_{>N}|^{\alpha +2} \right\|_{L_{t,x}^1}
			&\lesssim \left\| |x|^{-\frac{1}{\alpha +2}} u_{>N} \right\|_{L_{t,x}^{\alpha +2}}^{\alpha +2} 
			\lesssim \left\| (-\Delta _\Omega)^{\frac{1}{2(\alpha +1)}}u_{>N} \right\|_{L_{t,x}^{\alpha +2}}^{\alpha +2} \lesssim \left\| (-\Delta _\Omega)^{\frac{3\alpha -2}{4(\alpha +2)}}u_{>N} \right\|_{L_t^{\alpha +2} L_x^{\frac{6(\alpha +2)}{3\alpha +2}}}^{\alpha +2} .\notag
		\end{align}
		Moreover,  by the  Bernstein inequality and Lemma \ref{Lspace-time bound},  we have
		\begin{equation}
			\left\| (-\Delta _\Omega)^{\frac{3\alpha -2}{4(\alpha +2)}}u_{>N} \right\|_{L_t^{\alpha +2} L_x^{\frac{6(\alpha +2)}{3\alpha +2}}}^{\alpha +2} 
			\lesssim_u N^{1-2s_c} \| (-\Delta _\Omega)^{\frac{s_c}{2}} u \|_{L_t^{\alpha +2} L_x^{\frac{6(\alpha +2)}{3\alpha +2}}}^{\alpha +2} 
			\lesssim_u N^{1-2s_c} \left( 1 +|I| \right) < \infty.\notag
		\end{equation}

		Next, we turn to   $II_1$. By triangle inequality, we have
		\[
		II_1
		\lesssim \left\| \frac{1}{|x|} (u_{\leq N})^{\alpha +1} u_{>N} \right\|_{L_{t,x}^1}
		+ \left\| \frac{1}{|x|} u_{\leq N} (u_{>N})^{\alpha +1} \right\|_{L_{t,x}^1},
		\]
		which is exactly    (\ref{E1010w2}). Thus   $II_1$  is done.
		For  $II_2$, we estimate 
		\begin{equation}
			II_2\lesssim   \|\nabla u_{\le N}|u_{>N}|^{\alpha }u_{>N}\|_{L_{t,x}^1}+ \|\nabla u_{\le N}|u_{<N}|^{\alpha }u_{>N}\|_{L_{t,x}^1}.\notag  
		\end{equation}
		For the first term, we choose 	  $\varepsilon >0$   sufficiently small  such that  $s:=\frac{1}{2}-\frac{2-\varepsilon \alpha }{2\alpha }>0$. It then follows from   H\"older's inequality, Sobolev embedding, Bernstein's estimate, Theorem \ref{TEquivalence}, (\ref{E10101}) and  (\ref{E1010x1}) that 
		\begin{align}
			&\|\nabla u_{\le N}|u_{>N}|^{\alpha }u_{>N}\|_{L_{t,x}^1} \lesssim   \|\nabla u_{\le N}\|_{L^\infty _t L_x^{\frac{3}{1+\varepsilon }}} \||u_{>N}|^{\alpha }u_{>N}\|_{L^1_tL_x^{\frac{3}{2-\varepsilon }}}\notag\\
			&\lesssim N^{1-s_c+3(\frac{1}{2}-\frac{1+\varepsilon }{3})} \|(-\Delta _\Omega)^{\frac{s_c}{2}}u_{\le N}\|_{L^\infty _tL^2_x}  \|u_{>N}\|_{L^\infty _tL_x^{\frac{3\alpha }{2}}} ^{\alpha -1}\|u_{>N}\|_{L^2_tL_x^{\frac{6\alpha }{2-\varepsilon \alpha }}}    ^2\notag\\
			&\lesssim \eta   N^{1-s_c+3(\frac{1}{2}-\frac{1+\varepsilon }{3})}  \|(-\Delta _\Omega)^{\frac{s}{2}}u_{>N}\|_{L^2_tL^6_x}^2 \lesssim  \eta N^{1-s_c+3(\frac{1}{2}-\frac{1+\varepsilon }{3})}N^{2(s-s_c)}(1+N^{2s_c-1}|I|)\notag\\ 
			&\lesssim  \eta N^{1-2s_c}(1+N^{2s_c-1}|I|).\notag
		\end{align}

		For the second term, we use H\"older, Sobolev embedding, Bernstein, (\ref{E1010x1}) and (\ref{E1010x3}) to estimate 
		\begin{align}
			&\|\nabla u_{\le N}|u_{<N}|^{\alpha }u_{>N}\|_{L_{t,x}^1}\notag\\
			&\lesssim   \|\nabla u_{\le N}\|_{L^\infty _tL^2_x} \|(-\Delta _\Omega)^{\frac{s_c}{2}}u_{<N}\|_{L^4_tL^3_x}^\alpha  \|u_{>N}\|_{L_t^{\frac{4}{4-\alpha }}L_x^{\frac{6}{\alpha -1}}}\notag\\
			&\lesssim  N^{1-s_c}  \|(-\Delta _\Omega)^{\frac{s_c}{2}}u_{\le N}\|_{L^\infty _tL^2_x}^{\frac{\alpha }{2}+1} \|(-\Delta _\Omega)^{\frac{s_c}{2}}u_{<N}\|_{L^2_tL^6_x}^{\frac{\alpha }{2}} \|u_{>N}\|_{L^2_tL^6_x}^{\frac{4-\alpha }{2}} \|u_{>N}\|_{L^\infty _tL^2_x}^{\frac{\alpha }{2}-1}\notag\\
			&\lesssim N^{1-s_c}\eta ^{\alpha } (1+N^{2s_c-1}|I|)^{\frac{\alpha }{4}}N^{-s_c}(1+N^{2s_c-1}|I|)^{\frac{4-\alpha }{4}}N^{-s_c(\frac{\alpha }{2}-1)}\notag\\
			&\lesssim  \eta N^{1-2s_c}(1+N^{2s_c-1}|I|),\notag      
		\end{align}
		where noting that $\alpha <4$ when  $s_c=\frac{3}{2}-\frac{2}{\alpha }$.		Thus  $II_2$ is acceptable.
		
		Finally, we turn to  $III$. By Hardy's inequality  and Theorem \ref{TEquivalence}, 
		\begin{align}
			(\ref{E1010x12}) &\lesssim \|u_{>N}\|_{L_t^2 L_x^6} \left\| \frac{1}{|x|} P^{\Omega}_{\leq N}(F(u)) \right\|_{L_t^2 L_x^{6/5}} + \|u_{>N}\|_{L_t^2 L_x^6} \|\nabla P^{\Omega}_{\leq N}(F(u))\|_{L_t^2 L_x^{6/5}}\notag\\
			&\lesssim     \|u_{>N}\|_{L_t^2 L_x^6} \|(-\Delta _\Omega)^{\frac{1}{2}} P^{\Omega}_{\leq N}(F(u))\|_{L_t^2 L_x^{6/5}}.\notag
		\end{align}
		Thus, by (\ref{E1010x1}) it remains to prove
		\[
		\| (-\Delta _\Omega)^{\frac{1}{2}}P^{\Omega}_{\leq N}(F(u))\|_{L_t^2 L_x^{6/5}} \lesssim_u \eta N^{1- s_c}(1 + N^{2s_c - 1}|I|)^{1/2}.
		\]
		To this end, we use H\"older's inequality, Bernstein's estimate, the fractional chain rule (\ref{E12133}),  (\ref{Ebound}),  (\ref{E1010x4}), and (\ref{E1010x5}) to estimate
		\begin{align}
			& \|(-\Delta _\Omega)^{\frac{1}{2}}P^{\Omega}_{\leq N}(F(u))\|_{L_t^2 L_x^{6/5}}\notag\\
			&\lesssim N\|F(u) - F(u_{\leq N/\eta^2})\|_{L_t^2 L_x^{6/5}} + N^{1 - s_c} \|(-\Delta _\Omega)^{\frac{s_c}{2}}F(u_{\leq N/\eta^2})\|_{L_t^2 L_x^{6/5}}\notag\\
			&\lesssim    N\|u\|_{L_t^\infty  L_x^{3\alpha /2}}^{\alpha } \|u_{>N/\eta^2}\|_{L_t^2 L_x^6} + N^{1 - s_c} \|(-\Delta _\Omega)^{\frac{s_c}{2}}u\|_{L_t^\infty   L_x^{2}}^\alpha  \|(-\Delta _\Omega)^{\frac{s_c}{2}}u_{\leq N/\eta^2}\|_{L_t^2 L_x^6}\notag\\
			& \lesssim_u \eta N^{1 - s_c}(1 + N^{2s_c - 1} |I|)^{1/2}.
		\end{align}
		This completes the proof of Proposition \ref{PMorawetz}. 
	\end{proof}
	
	\subsection{Exclusion of the Minimal Counterexample}\label{s1/2-1,3}
	
	In this subsection, we rule out the almost periodic solutions to (\ref{NLS}) for the case $1/2 < s_c < 1$. The proof relies on the frequency-localized Lin-Strauss Morawetz inequality derived in Subsection \ref{s1/2-1,2}.
	\begin{theorem}\label{Ts1/2-1}
		There are no almost periodic solutions to (\ref{NLS}) as described in Theorem \ref{TReduction} when $1/2 < s_c < 1$.
	\end{theorem}
	\begin{proof}
		Assume, for contradiction, that $u$ is such a solution. Let $\eta > 0$ and $I \subset [0, \infty)$ be a compact time interval. By Proposition \ref{PMorawetz}, for sufficiently small $N$, we have
		\begin{equation}
			\int_{I }\int_\Omega \frac{|u_{>N}(t,x)|^{\alpha + 2}}{|x|} \, dx \, dt \lesssim_u \eta(N^{1-2s_c} + |I|).\label{E1011x1}
		\end{equation}
		Next, we establish a lower bound for the left-hand side of (\ref{E1011x1}). Since the orbit $\{u(t): t \in \mathbb{R}\}$ is precompact in $\dot H^{s_c}_D(\Omega)$ and $\dot H^{s_c}_D(\Omega) \hookrightarrow L^{\frac{3\alpha}{2}}(\Omega)$, for any $\varepsilon > 0$, there exists $c(\varepsilon) > 0$ such that 
		\begin{equation}
			\int_\Omega |u_{<c(\varepsilon)}(t,x)|^{\frac{3\alpha}{2}} dx < \varepsilon \quad \text{uniformly for } t \in \mathbb{R}. \notag
		\end{equation}
		
		Combining this with (\ref{E}), we deduce that for sufficiently small $N$, 
		\begin{equation}
			\int_{\Omega \cap \{|x| \le R\}} |u_{>N}(t,x)|^{\frac{3\alpha}{2}} dx \gtrsim 1 \quad \text{uniformly for } t \in \mathbb{R}. \notag
		\end{equation}
		
		Using H?lder's inequality, we further obtain that for sufficiently small $N$,
		\begin{equation}
			1 \lesssim \int_{\Omega \cap \{|x| \le R\}} |u_{>N}(t,x)|^{\frac{3\alpha}{2}} dx \lesssim \left(\int_{\Omega \cap \{|x| \le R\}} |u_{>N}(t,x)|^{\alpha + 2} dx\right)^{\frac{3\alpha}{2(\alpha + 2)}} R^{\frac{3(4-\alpha)}{2(\alpha + 2)}}. \notag
		\end{equation}
		
		Thus, for sufficiently small $N$, it follows that
		\begin{equation}
			\int_{I }\int_\Omega \frac{|u_{>N}(t,x)|^{\alpha + 2}}{|x|} \, dx \, dt \ge \frac{1}{R} \int_I \int_{\Omega \cap \{x: |x| < R\}} |u_{>N}(t,x)|^{\alpha + 2} dx \, dt \gtrsim R^{-\frac{4}{\alpha}} |I|.\label{E1011x2}
		\end{equation}
		
		Combining (\ref{E1011x1}) and (\ref{E1011x2}), and choosing $\eta = \eta(R)$ sufficiently small, we deduce that $|I| \lesssim_u N^{1-2s_c}$ uniformly in $I$. This leads to a contradiction when $|I|$ is taken to be sufficiently large. The proof of Theorem \ref{Ts1/2-1} is complete.
	\end{proof}
	
	\begin{remark}\label{R128}
		Finally, we discuss the possibility of extending Theorem \ref{T1} to the case $0 < s_c < \frac{1}{2}$. In this case, we also want to employ a frequency-localized Lin-Strauss Morawetz inequality to preclude almost periodic solutions. To achieve this, we truncate the high frequencies of the solution and work with $u_{<N}$ for some $N > 0$. Similar to Lemma \ref{LHLC}, we first establish a long-time Strichartz estimate and then prove (see e.g. \cite[Lemma 5.7]{Murphy2015}):
		\begin{equation}
			\|(-\Delta _\Omega)^{\frac{s}{2}} u_{\le N}\|_{L_t^2 L_x^6(I \times \Omega)} \lesssim N^{s-s_c}(1 + N^{2s_c-1} |I|)^{\frac{1}{2}}, \quad \text{for } \quad s > \frac{1}{2}. \label{E128}
		\end{equation}
		Note that the condition $s > 1/2$ is essential for deriving the estimate in (\ref{E128}) using the long-time Strichartz estimate. This condition ensures the convergence of the frequency summation.
		However, when using the Morawetz inequality to exclude almost periodic solutions, what we actually require is the estimate $\||\nabla|^{s} u_{\le N}\|_{L_t^2 L_x^6(I \times \Omega)}, s > \frac{1}{2}$. Notably, since $s > \frac{1}{2},p=6$ does not satisfy the index relationships in Theorem \ref{TEquivalence}, we cannot derive the required estimate from (\ref{E128}). Therefore, in the case $0 < s_c < \frac{1}{2}$, we still lack the necessary estimates to exclude almost periodic solutions.
	\end{remark}

\end{document}